\documentclass{amsart}

\usepackage{bm}

 \usepackage{upgreek}

\usepackage{amsopn, amsthm, amsgen, amscd,amsmath, amssymb}
\DeclareMathAlphabet{\mathbbm}{U}{bbm}{m}{n}
\usepackage{color}
\usepackage{tikz}
  \usepackage{graphicx}
\usepackage{epstopdf}
\usepackage[small,nohug,heads = LaTeX]{diagrams} \diagramstyle[labelstyle=\scriptstyle]

\definecolor{CadetBlue}{cmyk}{0.62, 0.57, 0.23, 0 }
\definecolor{black}{cmyk}{1, 0.5, 0, 0 }
\definecolor{RedViolet}{cmyk}{0.07, 0.9, 0, 0.34 }
\definecolor{SeaGreen}{cmyk}{0.69, 0, 0.5, 0}

\DeclareMathAlphabet{\mathpzc}{OT1}{pzc}{m}{it}

\newcommand{\R}{\mathbb R}
\newcommand{\C}{\mathbb C}
\newcommand{\D}{\mathbb D}
\newcommand{\E}{\mathbb E}
\newcommand{\F}{\mathbb F}

\newcommand{\N}{\mathbb N}
\newcommand{\PR}{\mathbb P}
\newcommand{\Q}{\mathbb Q}
\newcommand{\Z}{\mathbb Z}

\newcommand{\T}{\mathbb T}

\newcommand{\SI}{\mathbb S}

\newtheorem{theo}{Th\'{e}or\`{e}me}
\newtheorem{lemm}{Lemme}
\newtheorem{prop}{Proposition}
\newtheorem{coro}{Corollaire}

\newtheorem*{conj}{Conjecture}

\newtheorem*{ques}{Question}

\theoremstyle{definition}
\newtheorem{defi}{Definition}

\theoremstyle{remark}

\newtheorem{exam}{Exemple}

\newtheorem{note}{Note}

\newcommand{\bast}{{}^{\ast}\!}

\newcommand{\bbull}{{}^{\bullet}}

\catcode`é=\active\defé{\' e}
\catcode`è=\active\defè{\`e}
\catcode`ë=\active\defë{\"e}
\catcode`ê=\active\defê{\^e}
\catcode`à=\active\defà{\`a}
\catcode`ù=\active\defù{\`u}
\catcode`ô=\active\defô{\^o}
\catcode`â=\active\defâ{\^a}
\catcode`û=\active\defû{\^u}
\catcode`î =\active\defî{\^\i }
\catcode`ï =\active\defï{\"\i }
\catcode`ç=\active\defç{\c c}

\title[]{Th\'{e}orie Quasicristalline des Nombres: \\  Recherche d'une Th\'{e}orie de Drinfeld-Hayes en Charact\'{e}ristique Z\'{e}ro}
\author{T.M. Gendron}
\address{Instituto de Matem\'{a}ticas, Unidad Cuernavaca, UNAM, Av. Universidad s/n, 
Lomas de Chamilpa, Cuernavaca, Morelos, CP62180, M\'{E}XICO}
\email{tim@matcuer.unam.mx}
\author{E. Leichtnam}
\address{Institut Math\'{e}matique de Jussieu et CNRS, Universit\'{e} Paris Diderot, 
B\^{a}timent Sophie Germain, 
8 Place Aurélie Nemours, Paris 75013, FRANCE}
\email{eric.leichtnam@imj-prg.fr}
\author{P. Lochak}
\address{Institut Mathématique de Jussieu et CNRS,
Université Pierre et Marie Curie,
4~place Jussieu,
Paris 75005, FRANCE }
\email{pierre.lochak@imj-prg.fr}
\subjclass[2010]{}
\begin{document}
\vspace{2cm}
\maketitle
\begin{abstract}
Cet article développe la structure nécessaire à la formulation d'une version de la théorie de Drinfeld-Hayes en caractéristique nulle, en utilisant la théorie liée à l'arithmétique des anneaux quasicristallins attachés aux corps de nombres.
(On prendra garde que l'adjectif se
réfère aux quasi-cristaux au sens d'Yves Meyer \cite{Meyer}, sans rapport avec la théorie cristalline initiée par A.~Grothendieck.)
\end{abstract}
\tableofcontents

\section*{Introduction}  
Rassembler sous la dénomination commune de {\it corps globaux} les corps de nombres (alias les extensions finies de $\Q$) 
et les corps de fonctions (alias les extensions finies de ${\bf Q}:=\F_{q}(T)$) a sans doute constitué l'un des gestes les plus
productifs en théorie des nombres. Cette philosophie,  promue par A.Weil dans son ouvrage si influent (\cite{Weil}), 
a largement \'{e}tabli l'habitude de prouver d'abord les conjectures pour les corps de fonctions (donc en caractéristique 
positive), en utilisant les structures supplémentaires attachées à ces derniers, à savoir d'abord :
\begin{enumerate}
\item La {\it géometrie des courbes} que l'on trouve comme en coulisses dans les extensions de ${\bf Q}$ 
mais pas pour les corps de nombres ;
\item l'{\it arithm\'{e}tique de rang $1$}, c'est-\`{a}-dire l'existence d'anneaux de Dedekind avec groupe d'unit\'{e}s fini, 
qui d\'{e}coule du caractère non archimédien de la somme par rapport aux valuations à l'infini, tandis que les valuations 
à l'infini des extensions de $\Q$ sont elles archimédiennes.
\end{enumerate}
L'aspect géométrique des extensions de ${\bf Q}$ \'{e}tait utilis\'{e} par Weil pour d\'{e}montrer l'hypoth\`{e}se de Riemann \cite{Weil1} dans ce cas, ce qui l'a conduit à la formulation des fameuses conjectures, dites de Weil \cite{Weil2}. 
Les preuves de Weil et de Deligne conti\-nuent \`{a} inspirer certaines approches du cas ``classique'', comme le montrent 
les travaux d'A. Connes \cite{Connes} et C. Deninger \cite{Deninger}.  Il s'agit en particulier de chercher  l'analogue 
d'un morphisme \'{e}tale de courbes qui induirait une extension finie $K/\Q$.  

D'autre part l'arithm\'{e}tique des anneaux de Dedekind de rang $1$ a été synth\'{e}tis\'{e}e par V. Drinfeld \cite{Drinfeld} 
et D. Hayes \cite{Hayes} dans une th\'{e}orie du corps de classes en carac\-téristique positive qui combine des idées 
venues de la th\'{e}orie de la multiplication complexe et du th\'{e}or\`{e}me de Kronecker-Weber.  Plus g\'{e}n\'{e}ralement, 
V. Drinfeld \cite{Drinfeld2} a appliqué sa th\'{e}orie à la résolution d'un cas de la réciprocit\'{e} de Langlands en caractéristique positive. 

Ce qui précède rend très naturel le projet de chercher une version de la th\'{e}orie de Drinfeld-Hayes 
pour les corps de nombres. Une piste \'{e}tait fournie dans les articles \cite{Ge-C} et \cite{DGIII} qui concernent 
l'\'{e}tude de l'{\it invariant modulaire quantique}
\[ j^{\rm qt}: \R/{\rm GL}(2,\Z )\multimap \R\cup \{ \infty\},\]
une fonction {\it multivalu\'{e}e}, discontinue et invariante par rapport \`{a} l'action projective de ${\rm GL}(2,\Z )$ sur $\R$.  
La d\'{e}finition de $j^{\rm}(\uptheta)$ utilise les approximations
diophantiennes de $\uptheta$. 

Il existe aussi une definition de $j^{\rm qt}$ en caract\'{e}ristique positive
\[{\bf j}^{\rm qt}: {\bf R}/{\rm GL}(2,{\bf Z} )\multimap {\bf R}\cup \{ \infty\},\]
o\`{u} ${\bf R}={\bf Q}_{\infty}$ est l'analogue des nombres r\'{e}els et ${\bf Z}=\F_{q}[T]\subset {\bf Q}$ est l'analogue des entiers rationnels.  Dans ce cas, pour $f\in {\bf R}$ une unité
fondamentale
quadratique sur ${\bf Q}$, $j^{\rm qt}(f)$ est {\it fini},
et on peut identifier 
\begin{align}\label{ffjexplicit} {\bf j}^{\rm qt}(f)
=\{ j(\mathfrak{a}_{i})|\;  \mathfrak{a}_{i}\subset {\bf A}_{\infty_{1}} ,\; i=1,\dots ,d\}, \end{align} 
o\`{u} (voir \cite{DGIII} pour plus de précisions)
\begin{itemize}
\item ${\bf A}_{\infty_{1}}$ est un sous-anneau de Dedekind de ${\bf O}_{\bf K}$ = l'anneau des entiers 
de ${\bf K}={\bf Q}(f)$, ``petit'' au sens
que ${\bf A}^{\times}_{\infty_{1}}=\F_{q}^{\times}$ ;
\item Les $\mathfrak{a}_{i}\subset {\bf A}_{\infty_{1}}$ sont des id\'{e}aux ;
\item $j:{\sf Cl}_{{\bf A}_{\infty_{1}}}\longrightarrow {\bf R}$ est un invariant modulaire des classes d'id\'{e}aux 
de ${\bf A}_{\infty_{1}}$.
\end{itemize}  

Si l'on note ${\bf H}_{{\bf O}_{\bf K}}$ le corps de classes associ\'{e} \`{a} ${\bf O}_{\bf K}$,
le th\'{e}or\`{e}me principal de \cite{DGIII} dit que
\[ {\bf H}_{{\bf O}_{\bf K}}= {\bf K}(\prod_{\upalpha\in {\bf j}^{\rm qt}(f)} \upalpha )= {\bf K}(\prod_{i=1}^{d}  j(\mathfrak{a}_{i})) .\] 
La preuve utilise la th\'{e}orie de Drinfeld-Hayes, autrement dit une th\'{e}orie du corps de classes bas\'{e}e 
sur le ``petit'' anneau $ {\bf A}_{\infty_{1}}$ (\cite{Drinfeld}, \cite{Hayes}).  On a la conjecture suivante:
 
\begin{conj}[Demangos-Gendron \cite{DGIII}]
Soient $\uptheta\in\R\setminus \Q$ une unité fondamentale quadratique et $K=\Q (\uptheta )$ le corps quadratique associé,
de discriminant $D$.   Alors,
$j^{\rm qt}(\uptheta )$ est un ensemble de Cantor autosimilaire d'ordre $D$ et
\[ H_{K}= K({\sf N}^{\rm avg}(j^{\rm qt}(\uptheta )))\]
où  $H_{K}$ est le corps de classes de Hilbert de $K$ et ${\sf N}^{\rm avg}(j^{\rm qt}(\uptheta ))$ 
est un produit pond\'er\'e des éléments de $j^{\rm qt}(\uptheta )$.
\end{conj}

Dans cet article, on démontre que $j^{\rm qt}(\uptheta )$ est l'image continue d'un ensemble de Cantor ;
en particulier il est soit de Cantor, soit fini (voir  Corollaire \ref{jqtestCantor}).
\vskip .2cm

En vue de prouver l'analogue en caract\'{e}ristique nulle des r\'esultats  r\'{e}sum\'es plus haut 
il est d\'{e}sirable de disposer de l'\'{e}quivalent de la th\'{e}orie de Drinfeld-Hayes dans ce cas.
On constate imm\'{e}diatement que le premier pas consiste à trouver un anneau analogue $A_{\upsigma_{1}}$ 
à ${\bf A}_{\infty_{1}}$ en ca\-ract\'{e}ristique nulle, o\`{u} $\upsigma_{1}:K\longrightarrow\C$ est un plongement de $K/\Q$. 

Dans le cas où $K=\Q(\upvarphi )$, avec $\upvarphi$ le nombre d'or, une description explicite de $A_{\upsigma_{1}}$ 
\'{e}tait implicite dans les calculs de renormalisation donn\'{e}s en  \cite{Ge-C} (voir \S\S 2,3).   
Richard Pink (\cite{Pink}) a par la suite suggéré la définition élégante et transparente suivante : 
soit $K/\Q$ une extension quadratique réelle, $\upsigma_{i}$, $i=1,2$, les deux plongements de $K$ dans $\R$.  
On fixe $\upsigma_{1}$ en identifiant $K$ avec $\upsigma_{1}(K)$.  Soit
 \[  A_{\upsigma_{1}} = \{ \upalpha\in \mathcal{O}_{K} |\; |\upalpha' |\leq 1 \} ,\quad \upalpha' = \upsigma_{2}(\upalpha ) .\]
 Si l'on \'{e}crit $K=\Q (\uptheta )$ o\`{u} $\uptheta$ est une unit\'{e} fondamentale et que l'on d\'{e}finit 
 les semigroupes multiplicatifs 
\[   \mathfrak{a}_{x} =\{ \upalpha\in \mathcal{O}_{K}|\;   |\upalpha'|< \uptheta^{-x}\} \subset  \mathfrak{a}^{+}_{x} =\{ \upalpha\in \mathcal{O}_{K}|\;   |\upalpha'|\leq \uptheta^{-x}\} , \]
Pink a demontr\'{e} que 
\[ j^{\rm qt}(\uptheta ) =  \bigcup_{x\in [0,1)} \{  j(\mathfrak{a}_{x}),  j(\mathfrak{a}^{+}_{x}) \},\]
o\`{u} $j$ est un analogue de l'invariant modulaire (voir \S \ref{modinvsec}).
Cette formule est l'analogue exacte de (\ref{ffjexplicit}): en fait, les id\'{e}aux $\mathfrak{a}_{i}$ qui 
apparaissent dans  (\ref{ffjexplicit}) s'écrivent explicitement
\[\mathfrak{a}_{i}=\{ g\in {\bf K}|\; |g'|_{\infty_{1}} <q^{-i}\} . \] 
Pink a observé également que les semigroupes $\mathfrak{a}_{x}$ sont des {\it quasicristaux} (et même des {\it ensembles mod\`{e}les}). Cependant il leur manque pour constituer de véritables id\'{e}aux d'être stables par rapport \`{a} la somme.  
Dans cet article, partant de l'obser\-vation de R. Pink, nous explorons la possibilit\'{e} de d\'{e}velopper 
une th\'{e}orie de Drinfeld-Hayes bas\'{e}e sur l'anneau quasicristallin $A_{\upsigma_{1}}$.

Nous commençons au \S \ref{DrinfeldHayes} avec une r\'{e}capitulation de la th\'{e}orie de Drinfeld-Hayes.  
Au \S \ref{anneauxqc} nous donnons la d\'{e}finition g\'{e}n\'{e}rale d'un anneau quasicristallin arithmétique 
puis discutons au \S 3 le monoïde des classes d'id\'{e}aux ${\sf Cl}^{\sf qc}_{A_{\upsigma_{1}}}$.
Nous définissons ensuite (\S \ref{zeta}) la fonction z\^{e}ta $\upzeta_{\mathfrak{a}}$ associ\'{e} \`{a} 
un id\'{e}al quasicristallin et d\'{e}montrons l'existence d'un prolongement méromorphe à tout le plan complexe.  
Au \S \ref{modinvsec} on utilise  la fonction z\^{e}ta $\upzeta_{\mathfrak{a}}$  pour d\'{e}finir l'invariante modulaire 
$j:{\sf Cl}^{\sf qc}_{A_{\upsigma_{1}}}\rightarrow \R$ et montrer sa continuit\'{e} par rapport \`{a} la topologie
cantorienne de ${\sf Cl}^{\sf qc}_{A_{\upsigma_{1}}}$. Nous commençons ensuite (\S \ref{solenoid}) 
un développement plus g\'{e}om\'{e}trique en introduisant pour chaque idéal quasicristallin un soléno\"{\i}de 
associé qui repr\'{e}sente essentiellement une d\'{e}finition analytique d'un module de Drinfeld quasicristallin.
Enfin au \S \ref{exp} on introduit les fonctions trigonom\'{e}triques quasicristallines qui découlent de formules 
du produit, ainsi que l'exponentielle associ\'{e}e, dans le cadre de ces modules quasicristallins.

\vspace{3mm}
 
 \noindent {\em Remerciements}. Le premier auteur voudrait exprimer sa gratitude \`{a} 1) La Fondation des Sciences 
 Math\'{e}matiques de Paris pour le financement de son s\'{e}jour \`{a} Paris à l'automne 2017 ainsi que 
 l'\'{e}quipe d'alg\`{e}bres d'operateurs de l'Institut de Math\'{e}matiques de Jussieu -- Paris Rive Gauche pour 
 son hospitalit\'{e} et ses encoura\-gements et 2) Boris Zilber et Merton College de l'Universit\'{e} de Oxford pour
 le financement de son s\'{e}jour \`{a} Oxford au printemps 2018.  Il a également b\'{e}n\'{e}fici\'{e}
 de conversations avec Paul Nelson, Ryszard Nest, Richard Pink (qui a fait des commentaires tr\`{e}s utiles sur une version ant\'{e}rieure de l'article), Georges Skandalis et Boris Zilber.

 \section{Bref r\'{e}sum\'{e} de la th\'{e}orie de Drinfeld-Hayes}\label{DrinfeldHayes}  La th\'{e}orie quasicristalline d\'{e}velopp\'{e}e ici est bas\'{e}e sur la th\'{e}orie des modules de Drinfeld-Hayes de rang 1, present\'{e}e par exemple dans \cite{Hayes}, \cite{Goss}, \cite{Thakur} et que nous résumons ci-dessous.
  
 Soient ${\bf Q}=\F_{q}(T)\supset {\bf Z}:=\F_{q}[T]$ le corps des fonctions rationnelles à coefficients dans le corps fini $\F_{q}$ ($q=p^{k}$, $p$ un nombre premier)
 avec son sous-anneau des polyn\^{o}mes.  Alors ${\bf Q}$ est le corps de fonctions 
 de la droite projective $\PR^{1}=\overline{\bf Q}\cup \{ \infty\}$ et 
 ${\bf Z}$ l'anneau des fonctions reguli\`{e}res sur $\PR^{1} - \{ \infty\}$.   Soit ${\bf R}:={\bf Q}_{\infty}$ 
 le complet\'{e} de ${\bf Q}$ à la place $v_{\infty}$ : $v_{\infty}(f)= {\rm deg}_{T^{-1}}(f) = - {\rm deg}_{T}(f)$ et
 $|f|=q^{-v_{\infty}(f)}$. 
 On identifie ${\bf R}$ à  $\F_{q}[[T^{-1}]]$, le corps des séries de Laurent
 en la variable $T^{-1}$.  On note enfin ${\bf C}=(\overline{\bf R})_{\infty}$, qui est un 
 corps complet et alg\'{e}briquement clos.
 
 On dit qu'une extension ${\bf K}/{\bf Q}$ est {\it g\'eom\'{e}trique} si le corps des constantes ${\bf K}\cap\overline{\F_{q}}$ 
 de ${\bf K}$ est aussi $\F_{q}$.  Dans ce cas il existe un morphisme $\uppi:\Upsigma_{\bf K}\rightarrow \PR^{1}$ de courbes
  d\'{e}fines
 sur $\F_{q}$  qui induit cette extension.  Toutes les extensions considér\'{e}es ici sont  géom\'{e}triques. 
 \`{A} une telle extension on associe la cl\^{o}ture int\'{e}grale ${\bf O}_{\bf K}$ de ${\bf Z}$ dans ${\bf K}$ ;  
 on a \[{\bf O}_{\bf K} = {\rm Reg}(\Upsigma_{\bf K}\setminus\uppi^{-1}(\infty)),\]
 o\`{u} ${\rm Reg}$ dénote les fonctions r\'{e}guli\`{e}res.  Si $\infty_{1}\in \uppi^{-1}(\infty)$ on peut d\'{e}finir 
 \[ {\bf A}_{\infty_{1}} =  {\rm Reg}(\Upsigma_{\bf K}\setminus \{ \infty_{1}\}) \subset {\bf O}_{\bf K},\]
 qui est un anneau de Dedekind ; ${\bf A}_{\infty_{1}}$ est ``petit'' au sens que ${\bf A}_{\infty_{1}}^{\times}$ est fini, ce 
 qui implique l'existence des places ${\bf A}_{\infty_{1}}\hookrightarrow {\bf C}$ d'images {\it discrètes}.  
 \`{A} ces deux anneaux de Dedekind ${\bf O}_{\bf K}$ et  ${\bf A}_{\infty_{1}}$ on peut associer
 les corps de classes de Hilbert au sens de Rosen (\cite{Rosen})
\begin{diagram} 
{\bf H}_{{\bf A}_{\infty_{1}}} \\
\dLine \\
{\bf H}_{{\bf O}_{\bf K}} \\
\dLine \\
{\bf K}
 \end{diagram}
dont les groupes de Galois sont isomorphes aux groupe des classes id\'{e}aux corres\-pondants: 
\[ {\rm Gal}({\bf H}_{{\bf A}_{\infty_{1}}} /{\bf K})\cong {\sf Cl}_{{\bf A}_{\infty_{1}} }\quad \text{et} \quad 
{\rm Gal}({\bf H}_{{\bf O}_{\bf K}} /{\bf K})\cong {\sf Cl}_{{\bf O}_{\bf K} } .\]

\begin{exam}
Le cas de ${\bf K}\subset {\bf R}$ extension de degr\'{e} $2$ sur ${\bf Q}$, induite par le morphisme de degr\'{e} $2$ $\uppi:\Upsigma_{\bf K}\rightarrow \PR^{1}$, est particulièrement intéressant.  On a deux préimages de $\infty\in \PR^{1}$, soit
 $\uppi^{-1}(\infty)=\{ \infty_{1},\infty_{2}\}$.  Dans ce contexte
 \[{\bf A}_{\infty_{1}}  = \{ g\in {\bf K}|\; |g'|_{\infty_{1}}\leq 1\},\] 
 o\`{u} $g'$ est la conjugu\'{e}e de Galois de $g$ et  $|g'|_{\infty_{1}}=q^{-v_{\infty_{1}}(g')}$. 
 Fixons une unit\'{e} fondamentale $f\in {\bf O}_{\bf K}^{\times}$ tel que $|f|_{\infty_{1}}=q^{d}>1$. 
 Suivant \cite{DGIII} on a l'identification
\[ {\bf A}_{\infty_{1}} = \F_{q}[f,fT,\dots ,fT^{d-1}] \subset {\bf O}_{\bf K}=\F_{q}[f,T].\]
Dans cette même référence on a également démontré (cf. Theorem 4 de \cite{DGIII}) que :
 \[ {\rm Gal}({\bf H}_{{\bf A}_{\infty_{1}}} / {\bf H}_{{\bf O}_{\bf K}} )\cong {\sf Z}:=\{ \mathfrak{a}_{0},\dots ,\mathfrak{a}_{d-1}\} \subset {\sf Cl}_{{\bf A}_{\infty_{1}} } \]
 o\`{u} 
 \[ \mathfrak{a}_{i}=(f,fT,\dots ,fT^{i})= \{ g\in {\bf A}_{\infty_{1}} |\; |g'|_{\infty_{1}}\leq q^{i-d} \} \subset  {\bf A}_{\infty_{1}}. \]
\end{exam}

\`{A} tout id\'{e}al fractionaire de ${\bf A}_{\infty_{1}}$ on associe 
un ${\bf A}_{\infty_{1}}$-module de rang 1 comme suit.  On d\'{e}finit d'abord
l'exponentielle associ\'{e} \`{a} $\mathfrak{a}$
\[ \exp_{\mathfrak{a}}(z) = z\prod_{0\not=\upalpha\in\mathfrak{a}} \left( 1-\frac{z}{\upalpha}\right) :\]
qui est un épimorphisme {\it additif} $({\bf C},+)\rightarrow ({\bf C},+)$ de noyau $\mathfrak{a}$.  
Pour tout $\upalpha\in {\bf A}_{\infty_{1}}$ il existe alors (\cite{Goss}) un polyn\^{o}me {\it additif }\footnote{C'est-à-dire 
de la forme
$f(X)=\sum c_{i}\uptau^{i}$ o\`{u} $\uptau (X)=X^{q}$ est le morphisme de Frobenius.} 
 $\uprho_{\upalpha}(X)$  tel que le diagramme ci-dessous soit commutatif:
\begin{diagram}
{\bf C}/\mathfrak{a} & \rTo^{\upalpha\cdot} &{\bf C}/\mathfrak{a}  \\
\dTo^{\exp_{\mathfrak{a}}}_{\cong} & & \dTo_{\exp_{\mathfrak{a}}}^{\cong}\\
{\bf C} &\rTo_{\uprho_{\upalpha}} & {\bf C}
\end{diagram}
Le {\em module de Drinfeld} (de rang $1$) associ\'{e} \`{a} $\mathfrak{a}$, soit $\D_{\mathfrak{a}}=({\bf C},\uprho)$, 
est défini alors comme ${\bf C}$ muni de la structure de ${\bf A}_{\infty_{1}}$-module d\'{e}finie par les $\uprho_{\upalpha}$.
En tant que  ${\bf A}_{\infty_{1}}$-module, on a
\[ \D_{\mathfrak{a}} \cong {\bf C}/\mathfrak{a}. \]

Il existe une normalisation naturelle de $\D_{\mathfrak{a}}$ par $\upxi_{\mathfrak{a}}\in {\bf C}$ qui joue un r\^{o}le analogue \`{a} la normalisation de $\Z$ par $\uppi i\in\C$ dans la th\'{e}orie de
l'exponentielle classique.  On
remplace le r\'{e}seau $\mathfrak{a}\subset {\bf C}$ par $\Uplambda_{\mathfrak{a}}:= \upxi_{\mathfrak{a}}\mathfrak{a}$ et  on consid\`{e}re l'exponentielle correspondante 
\[ e_{\mathfrak{a}}(z):= \exp_{\Uplambda_{\mathfrak{a}}} (z)=\upxi_{\mathfrak{a}} \exp_{\mathfrak{a}}(\upxi_{\mathfrak{a}}^{-1}z).\]
On choisit $\upxi_{\mathfrak{a}}$ de manière à obtenir un module de Drinfeld $\widetilde{\D}_{\mathfrak{a}}=({\bf C},\tilde{\uprho})$  conjugu\'{e} \`{a} $\D_{\mathfrak{a}}$ et {\it normalis\'{e} par rapport au signe} (voir au \S 12 de \cite{Hayes}).  Dans ce cas, pour chaque $\upalpha\in {\bf A}_{\infty_{1}}$ 
les coefficients de $\tilde{\uprho}_{\upalpha}$ appartiennent au corps de classes de Hilbert ${\bf H}_{{\bf A}_{\infty_{1}}}$ 
et de plus (voir \cite{Hayes} et \cite{Thakur} Theorem 3.4.2) Hayes a montr\'{e} qu'ils engendrent ce corps :
\[  {\bf H}_{{\bf A}_{\infty_{1}}} = {\bf K}(\text{coefficients de $\tilde{\uprho}_{\upalpha}$}).\]
Le module $\widetilde{\D}_{\mathfrak{a}}$ s'appelle {\em module de Hayes}.  Il a permis à ce dernier de développer 
une th\'{e}orie explicite du corps de classes bas\'{e}e sur le ``petit'' anneau de Dedekind $ {\bf A}_{\infty_{1}}$
en utilisant les modules $\tilde{\D}_{\mathfrak{a}}$. 

 Le but de cet article est de proposer un candidat pour représenter
un analogue des modules de Hayes pour les anneaux d'entiers des corps de nombres.  En particulier, on définira les analogues 
des objets suivants :
\begin{itemize}
\item[-] le petit anneau  ${\bf A}_{\infty_{1}}$ : ce sera un {\it anneau quasicristallin PVS (Pisot-Vijayaraghavan et Salem)}, voir \S \ref{anneauxqc} ;
\item[-] un idéal $\mathfrak{a}$ de ${\bf A}_{\infty_{1}}$, appelé ci-dessous {\it idéal quasicristallin}, voir \S \ref{anneauxqc} ;
\item[-] le groupe de classes d'idéaux ${\sf Cl}_{{\bf A}_{\infty_{1}}}$, ici le {\it monoïde des classes des idéaux quasicristallins}, voir \S \ref{monoides} ;
\item[-] le quotient ${\bf C}/\mathfrak{a}$, noté $\hat{\SI}_{\mathfrak{a}}$, alias ici le {\it solénoïde quasicristallin} associé à $\mathfrak{a}$, voir \S \ref{solenoid} ;
\item[-] l'exponentielle $\exp_{\mathfrak{a}}$,  voir \S \ref{exp} ;
\item[-] le module de Drinfeld $\D_{\mathfrak{a}}$, devenu le {\it module de Drinfeld quasicristallin} et noté 
\[ \E_{\mathfrak{a}}\subset\C^{\ast}.\]  
Sa structure de ``module'' provient d'applications multivaluées
\[\uprho_{\upalpha}: \D_{\mathfrak{a}}\longrightarrow \D_{\mathfrak{a}}, \] induisant des fonctions bien
définies sur le cercle $\SI^{1}=\C^{\ast}/\R^{\ast}_{+}$.
\end{itemize}

 \section{Anneaux et id\'{e}aux quasicristallins}\label{anneauxqc}
 
 Soit $X\subset\R^{n}$.  On dit que $X$ est un {\it ensemble de Delaunay} s'il est
 \begin{enumerate}
 \item[1.] {\it uniform\'{e}ment discret}: il existe $r>0$ tel que pour tout $x\in X$, $B_{r}(x)\cap X=\{ x\}$, o\`{u} $B_{r}(x)$ est le boule ouvert de radius $r$ centr\'{e} en $x$.
 \item[2.]  {\it relativemente dense}: il existe $R>r>0$ tel que pour tout $v\in \R^{n}$, $B_{R}(v)\cap X$ est non nulle.
 \end{enumerate}
  Un {\em quasicristal}  (au sens d'Yves Meyer \cite{Meyer}) est un ensemble de Delaunay 
   $\Uplambda\subset\R^{n}$
   qui est un {\em presque r\'{e}seau} :
il existe un ensemble fini $F\subset \R^{n}$ tel que
 \[  \Uplambda -\Uplambda \subset \Uplambda +F.\]

\begin{exam}[Ensembles Modèles]\label{ensmod} Soit $\R^{N}=V_{1}\oplus V_{2}$ une d\'{e}composition en somme directe de deux sous-espaces vectoriels, $\uppi_{i}$ les projections sur $V_{i}$, $i=1,2$.  Soient $\Upgamma\subset\R^{N}$ un r\'{e}seau
et $D\subset V_{2}$ un ensemble relativement compact (appel\'{e} {\em fen\^{e}tre}).
L'{\em ensemble mod\`{e}le} associ\'{e} est défini par
\[  \mathcal{M}= \mathcal{M}( \Upgamma, D)= \mathcal{M} (V_{1},V_{2}, \Upgamma, D) = \{ \uppi_{1}(v)|\; v\in \Upgamma, \uppi_{2}(v)\in D\}\subset V_{1}. \]
Normalmente on prendra $N=n+m$, $V_{1}=\R^{n}$ et $V_{2}=\R^{m}$, afin d'avoir $\mathcal{M}\subset\R^{n}$.
Plus g\'{e}n\'{e}ralement, on peut remplacer $\R^{N}$ par $\R^{n}\oplus H$ o\`{u} $H$ est un groupe localement compact.  En prenant $\Upgamma\subset
\R^{n}\oplus H$ un r\'{e}seau et $D\subset H$ relativement compact,  on
d\'{e}finit de même $ \mathcal{M}= \mathcal{M} (\R^{n},H, \Upgamma, D)$.
Dans les deux cas, il en r\'esulte que $\mathcal{M}$ est un quasicristal (\cite{Meyer}, \cite{Moody}).  Tout quasicristal est contenu dans un ensemble mod\`{e}le, mais  il
existe des quasicristaux qui ne sont pas des ensembles mod\`{e}les (\cite{Moody}). 
\end{exam}



\begin{defi}\label{AQdef} Un {\em anneau quasicristallin} (euclidien de rang $n$) est un quasicristal $A\subset \R^{n}$  
qui est un mono\"{\i}de multiplicatif contenant
$\{0, \pm 1\}$.  Un sous-mono\"{\i}de $\mathfrak{a}\subset A$ satisfaisant $A\cdot\mathfrak{a}\subset\mathfrak{a}$ est appel\'e
{\em id\'{e}al quasicristallin} (entier)\footnote{On note (voir \cite{Moody}) qu'un id\'{e}al quasicristallin entier et non zéro est aussi un quasicristal.}.  Plus g\'{e}n\'{e}ralement,
un  id\'{e}al quasicristallin fractionnaire est un quasicristal $\mathfrak{a}\subset\R^{n}$ tel que $A\cdot \mathfrak{a}\subset \mathfrak{a}$.
\end{defi}

\begin{exam} Soient $\Upgamma\subset\R^{n+m}$ un réseau qui est aussi un anneau avec $1$ et $D\subset \R^{m}$ un monoïde multaplicatif relativement compact qui contient $\uppi (1)$ (par exemple, on peut prendre
$D=$ la boule unitaire en $\R^{m}$ ou le produit $[-1,1]\times\cdots\times[-1,1]$).  L'ensemble modèle associé $A=\mathcal{M}(\Upgamma, D)$ est donc un anneau quasicristallin.  Par chaque idéal
$\mathfrak{A}\subset\Upgamma$ et chaque sous-monoïde relativement compact $D'\subset D$, $\mathfrak{a}=\mathcal{M}(\mathfrak{A}, D')$ est un idéal quasicristallin de $A$.
\end{exam}


Dans la suite on se concentre sur le cas des id\'{e}aux quasicristallins associ\'{e}s aux sous mono\"{\i}des multiplicatifs d'un corps de nombres
$K/\Q$ de degré $d>1$.  On note les $d$ plongements $\upsigma:K\hookrightarrow \C$ ainsi que l'espace de Minkowski 
\[ K\hookrightarrow K_{\infty}=\{\boldsymbol z=(z_{\upsigma})\in \C^{d}| \overline{z}_{\upsigma}=z_{\overline{\upsigma}}\}\cong \R^{r}\times \C^{s},\] o\`{u} $r=$ 
le nombre de plongements r\'{e}els, $s=$ le nombre des paires de plongements
complexes et $d=r+2s$.   On fixe un sous ensemble 
\[ \Upsigma=\{\upsigma_{1},\dots ,\upsigma_{k}\},\quad \upsigma_{i}:K\hookrightarrow K_{\upsigma_{i}}= \left\{ \begin{array}{l}
\R \\
\text{\small{ou}} \\
\C \\
\end{array}
\right.\] de plongements tel que 
 si $\upsigma\in\Upsigma$ est complexe, $\overline{\upsigma}\in\Upsigma$.  L'ensemble $\Upsigma$ d\'{e}finit
un sous-espace $K_{\Upsigma}\subset K_{\infty}$ de fa\c{c}on \'{e}vidente et un plongement
 \[ \Upsigma=(\upsigma_{1},\dots ,\upsigma_{k}): K\hookrightarrow K_{\Upsigma}.\]
 Si $\Upsigma'$ est le complement de $\Upsigma$, on a $K_{\infty}=K_{\Upsigma}\oplus K_{\Upsigma'}$.
On dit que $A\subset K$ est un {\em anneau quasicristallin} s'il existe $\Upsigma=(\upsigma_{1},\dots ,\upsigma_{k})$ 
comme ci-dessus et tel que l'image
$\Upsigma(A)\subset K_{\Upsigma}$ est un anneau quasicristallin de $K_{\Upsigma}$.  
On identifiera alors $A$ avec son image $\Upsigma (A)\subset K_{\Upsigma}$.

\begin{exam}\label{ExempleAnneaux} Fixons $\Upsigma$ comme ci-dessus. On peut associer \`{a} chaque $\upsigma\in\Upsigma'$  la fen\^{e}tre $D_{\upsigma}=\{x\in K_{\upsigma}|\; |x|\leq 1\}$.
Alors l'ensemble mod\`{e}le 
\[ A_{\Upsigma}=\mathcal{M} (K_{\Upsigma},K_{\Upsigma'},\mathcal{O}_{K}, D_{\Upsigma}),\]
o\`{u} $D_{\Upsigma}=\prod_{\upsigma\in\Upsigma'} D_{\upsigma}$, d\'{e}finit un anneau quasicristallin $A_{\Upsigma}\subset K_{\Upsigma}$.
Si $\Upsigma=\{ \upsigma\}$ on note \[ A_{\upsigma}  \subset  \left\{ \begin{array}{l}
\R \\
\text{\small{ou}} \\
\C
\end{array}
\right. \] l'anneau quasicristallin associ\'{e}, qu'on dira de {\em rang} 1, r\'{e}el ou complexe selon la nature du
plongement $\upsigma$.  

Si par exemple  $K=\Q (\uptheta )\subset\R$ est une extension quadratique r\'{e}elle de $\Q$,  alors on a deux plongements $\upsigma_{1}={\rm id}$ et $\upsigma_{2}$ de sorte que
\[ A_{\upsigma_{1}} =\{ \upalpha\in \mathcal{O}_{K}|\; |\upalpha'|\leq 1\} ,\quad  A_{\upsigma_{2}} =\{ \upalpha' \in \mathcal{O}_{K}|\; |\upalpha|\leq 1\}, \] o\`{u} $\upalpha'$ dénote 
le conjugu\'{e}e sous Galois.  Dans ce cas, on a
\[ A_{\upsigma_{1}}\cap A_{\upsigma_{2}} =\pm 1, \quad   \langle A_{\upsigma_{i}}\rangle = \mathcal{O}_{K},\;\; i=1,2,\]
 où $\langle X\rangle$ est le groupe engendré par $X$ et la deuxième égalité est consequence du fait que $\mathcal{O}_{K}=\Z[\uptheta]$, pour $\uptheta$ une unité fondamental.
Si $K=\Q (\upomega)$ quadratique complexe, il n'y a qu'une paire de places conjuguées ($\upsigma$  et $\bar \upsigma$)
avec $A_{\upsigma}=\mathcal{O}_{K}$.  
 Si $\Upsigma$ comprend  tous les plongements de $K$, on obtient $A_{\Upsigma}=\mathcal{O}_{K}$. \end{exam}

Soient $\uptheta$ un entier alg\'{e}brique re\'{e}l tel que $\uptheta>1$, $\uptheta^{\upsigma_{i}}$ ses  conjugu\'{e}es, $\upsigma_{i}\not={\rm id}$. On rappelle que $\uptheta$ est un 
{\em nombre de  Pisot-Vijayaraghavan (PV)}  si pour tout~$i$, $|\uptheta^{\upsigma_{i}}|<1$, voir \cite{Pisot}, \cite{Cassels}.  
Si $|\uptheta^{\upsigma_{i}}|\leq 1$ pour tout $i$, avec egalit\'{e} pour au moins une valeur de $i$, 
on dit que $\uptheta$ est un {\em nombre de  Salem (S)} (\cite{Salem}).   

De même un entier alg\'{e}brique complexe $\upomega$ avec $|\upomega|>1$ est un
{\em nombre de  PV complexe} (\cite{BertinZaimi}) si tous ses conjugu\'{e}es $\upomega^{\upsigma_{i}}\not=\overline{\upomega}$ satisfont $|\upomega^{\upsigma_{i}}|< 1$; c'est un {\em nombre de  S complexe} si 
 tous ses conjugu\'{e}es $\upomega^{\upsigma_{i}}\not=\overline{\upomega}$ satisfont $|\upomega^{\upsigma_{i}}|\leq 1$ 
 avec de nouveau \'{e}galit\'{e} pour au moins un $i$.  On note $PV$ l'ensemble de tous les nombres de Pisot-Vijayaraghavan complexes, $S$
 celui des nombres de Salem complexes et $PVS$ l'union de ces deux ensemles ($PVS=PV\cup S$).
Les intersections  ${PV}_{\R} = {PV}\cap\R$, ${S}_{\R} = { S}\cap\R$, $ { PVS}_{\R} = { PVS}\cap\R$, 
consistent en les nombres $PV$, $S$ et $PVS$ classiques ainsi que leurs opposés.  Si $K/\Q$ est finie, on note 
$PV_{K}=PV\cap K$ et de même pour $S$ et $PVS$. On a alors l'énoncé suivant :
 \begin{prop}
Soit $A_{\upsigma}\subset K$ comme dans l'Exemple \ref{ExempleAnneaux}.  Alors on a l'\'{e}galit\'{e}
\[ A_{\upsigma} =PVS_{K} \cup \upmu_{K}\cup \{ 0\}\]
o\`{u} $\upmu_{K}\subset \mathcal{O}_{K}^{\times}$ est le sous-groupe multiplicatif des racines de l'unit\'{e}.
\end{prop}

{\it Preuve.} On identifie $A_{\upsigma}$ à son plongement complexe en utilisant la place $\upsigma$. 
Pour tout $\upalpha\in A_{\upsigma}$, $|\upalpha^{\uptau}|\leq 1$ pour tout $\uptau\not=\upsigma,\bar{\upsigma}$.  
Si $|\upalpha|>1$ alors $\upalpha\in PVS_{K} $ et sinon $|\upalpha|=1$ et donc $\upalpha\in\upmu_{K}$.

\hfill $\square$

Un tel $A_{\upsigma}$ sera appelé {\em anneau quasicristallin $PVS$} associ\'{e} \`{a} $K$ et $\upsigma$;   
dans ce que suit nous nous concentrerons sur l'\'{e}tude de tels anneaux quasicristallins de rang 1.
En ce que suit on choisit une place complexe de chaque paire de places conjuguées : puis l'ensemble des places $\upsigma_{i}$ dont $\upsigma_{i},\overline{\upsigma}_{i}\not=\upsigma$ contient $r+s-1$  \'{e}l\'{e}ments.
Notons 
\[ \uppi':K_{\infty}\longrightarrow K_{\{ \upsigma\}'}\] la projection sur $K_{\{ \upsigma\}'}$, o\`{u} on rappelle que $\{ \upsigma\}'=$ complement de $\upsigma$ dans l'ensemble des places.
Soient ${\bf z}\in K_{\upsigma'}$ et $\boldsymbol \updelta=(\updelta_{1},\dots ,\updelta_{r+s-1})>0$ (i.e.\ $\updelta_{i}>0$ pour tout $i$).  On \'{e}crira
\[     |{\bf z}|< \boldsymbol\updelta^{\bf x}, \quad {\bf x}\in \R^{r+s-1}, \]
si et seulement si $|z_{i}|<\updelta_{i}^{x_{i}}$ pour tout $i$.
Si
\[ u\in U_{K}:=\mathcal{O}^{\times}_{K}/\upmu_{K} \cong \Z^{r+s-1}\]
est une unit\'{e} d'ordre infini, on d\'{e}finit
pour chaque ${\bf x} \in \R^{r+s-1}$ 
\[ \mathfrak{a}_{\bf x}(u) := \{ \upalpha\in\mathcal{O}_{K}|\; |\uppi'(\upalpha)|<|\uppi'(u)|^{\bf x} \}.
 \]
Il est clair que $\mathfrak{a}_{\bf x}(u)$ est un ensemble mod\`{e}le et un mono\"{\i}de multiplicatif, donc un id\'{e}al quasicristallin fractionnaire puisque $A_{\upsigma}$ agit dessus par multiplication.  Pour $u_{1}, u_{2},u\in U_{K}$, on a les inclusions 
\[\mathfrak{a}_{\bf x}(u_{1})\cdot  \mathfrak{a}_{\bf x}(u_{2})
\subset  \mathfrak{a}_{\bf x}(u_{1}u_{2}),\quad  
 \mathfrak{a}_{{\bf x}}(u)\cdot \mathfrak{a}_{{\bf y}}(u) \subset\mathfrak{a}_{{\bf x}+{\bf y}}(u), \]
ainsi que les \'{e}galit\'{e}s
\[ \mathfrak{a}_{-{\bf x}}(u) = \mathfrak{a}_{\bf x}(u^{-1}) ,\quad u\cdot \mathfrak{a}_{{\bf x}}(u) = \mathfrak{a}_{{\bf x}+{\bf 1}}(u )  \]
o\`{u} ${\bf 1}=(1,\dots ,1)$.  Pour tout $u\in U_{K}$, on a
\[  \mathfrak{a}_{\boldsymbol 0}:= PV_{K}\cap A_{\upsigma}= \mathfrak{a}_{{\bf 0}}(u), \]
l'ensemble de nombres de PV de $A_{\upsigma}$, qui est un id\'{e}al quasicristallin entier maximal de $A_{\upsigma}$ ;
son complémentaire est donné par
\[ A_{\upsigma}\setminus\mathfrak{a}_{0} = {S}_{K}\cup \upmu_{K}.\]

Dans la suite on fixe $u$ une {\em unit\'{e} de PV}\;\footnote{Il en existe toujours ; voir par exemple la preuve du Th\'{e}or\`{e}me 3.26 de Narkiewicz \cite{Nark}.}, c'est-\`{a}-dire une unit\'{e} $u$ d'ordre infini tel que $|\uppi'(u)|<{\bf 1}$. Nous omettrons 
parfois $u$ dans la notation, \'{e}crivant simplement $\mathfrak{a}_{\bf x}$.  
Outre les quasicristaux $\mathfrak{a}_{\bf x}$, il sera commode de disposer des quasicristaux associ\'{e}s
 \`{a} un vecteur $\boldsymbol{+} = (+_{1},\dots ,+_{d-1})  \in  \{0,1 \}^{d-1}$: si l'on note 
 \[ <_{+_{i}}=\left\{   
 \begin{array}{ll}
 < & \text{ si 
$ +_{i}=0$,} \\
\leq & \text{ si $+_{i}=1$}
\end{array}\right. ,\] on d\'{e}finit
\[ \mathfrak{a}_{\bf x}^{\boldsymbol +}=\{ \upalpha\in\mathcal{O}_{K}|\; |\uppi'(\upalpha)| <_{\boldsymbol +} |\uppi'(u)|^{\bf x} \}\supset \mathfrak{a}_{\bf x}.\]
Lorsque $\boldsymbol + = (1,\dots ,1)$, on note 
\[  \overline{\mathfrak{a}}:= \mathfrak{a}^{\boldsymbol +} ,\]
de sorte que
$ A_{\upsigma}=\overline{\mathfrak{a}}_{\boldsymbol 0}$  et $\mathfrak{a}^{\boldsymbol +}_{\boldsymbol x}=
\mathfrak{a}_{\boldsymbol x}$ quand ${\boldsymbol +}=(0,\dots, 0)$.
  
  Les quasicristaux $\mathfrak{a}_{\bf x}\subseteq\mathfrak{a}_{\bf x}^{\boldsymbol +}$  sont des id\'{e}aux 
 quasicristallins entiers si ${\bf x}\geq {\bf 0}$ (parce que $u$ est PV), et autrement
des id\'{e}aux quasicristallins fractionnaires. Clairement,
\[ \mathfrak{a}_{\bf x}\subsetneq\mathfrak{a}_{\bf x}^{\boldsymbol +}\;\;\Longrightarrow\;\; \exists i\text{ tel que }
 \upsigma_{i}( u)^{x_{i}}\in  \mathcal{O}_{K}. \]
Par exemple, si $K=\Q(\uptheta )$ est une extension quadratique réelle et $\uptheta>1$ une unit\'{e} fondamentale, 
alors pour $x\in\R$,
\[  \mathfrak{a}_{ x}(\uptheta )\subsetneq\mathfrak{a}_{x}^{+}(\uptheta) \Leftrightarrow 
  \uptheta^{x}\in  \mathcal{O}_{K},\] 
  et dans ce cas \[    \mathfrak{a}_{x}^{+}(\uptheta)= \mathfrak{a}_{x}(\uptheta )\cup\{ \pm  \uptheta^{x}\}.\]
  
\begin{note}\label{DefZ} Le choix de l'unit\'{e} $u$ est sans conséquence. Plus précisément la famille 
\[ \label{defZ}  Z :=\{ \mathfrak{a}^{\boldsymbol +}_{\bf x}(u) \}_{{\boldsymbol +}, {\bf x}}\]
est indépendante du choix de $u$.   En effet, si $\tilde{u}$ est autre unité de PV, il existe ${\bf y}$ tel que $|\uppi'(u)|^{\bf y}=|\uppi '(\tilde{u})|$.
\end{note}

\begin{note}\label{carzero} On peut considérer les analogues des quasicristaux et des ensembles mod\`{e}les dans 
le cas de caract\'{e}ristique positive (voir \S  \ref{DrinfeldHayes} ci-dessus
pour les notations).  En rempla\c{c}ant $\R$ par
${\bf R}$ on peut d\'{e}finir pareillement les notions de quasicristal et d'ensemble mod\`{e}le.  On
 a alors
\[  {\bf A}_{\infty_{1}}= {\bf A}_{\upsigma} = \mathfrak{a}_{0}^{+} = \{ g\in {\bf O}_{\bf K}|\; |\uppi'(g)|_{\infty_{1}}\leq {\bf 1}\}\]
qui est un anneau au sens usuel.
Dans ce contexte les $\mathfrak{a}_{\bf x}(u)$ s'identifient à
\[  \mathfrak{a}_{\bf n}:= \{g\in {\bf O}_{\bf K}|\; |\uppi'(g)|_{\infty_{1}}\leq q^{-{\bf n}}\}  , \quad \text{\bf n} = (n_{1},\dots , n_{r+s-1})\in\Z^{r+s-1}.\]
où $\mathfrak{a}_{\bf n}$ est un id\'{e}al fractionnaire au sens usuel.  Plus g\'{e}n\'{e}ralement tout id\'{e}al quasicristallin fractionnaire qui est un ensemble mod\`{e}le est un id\'{e}al fractionnaire au sens usuel de l'anneau ${\bf A}_{\infty_{1}}$.  
On en conclut que les analogues en caractéristique nulle des ``petits'' anneaux ${\bf A}_{\infty_{1}}$ sont les anneaux quasicristallins PV de type $A_{\upsigma}$.

\end{note}

\begin{note} Notons enfin que les constructions faites ci-dessus en rang 1 ont des \'{e}quivalents dans le cas où 
$A_{\Upsigma}$ est un anneau quasicristallin mod\`{e}le de rang $k$, $\Upsigma=\{ \upsigma_{1},\dots ,\upsigma_{k}\}$.  On d\'{e}finit en ce cas l'ensemble $PV_{\Upsigma}$ des nombres de PV associ\'{e}s \`{a} $\Upsigma$ comme 
\[ {PV}_{\Upsigma} =\{ \upalpha\in \mathcal{O}_{K}|\; |\upsigma_{1}(\upalpha)\cdots \upsigma_{k}(\upalpha )|>1,\;\; |\upsigma'(\upalpha )|<1 \text{ pour tout }\upsigma'\in\Upsigma'\}.\]
On a donc toujours $A_{\Upsigma}=PV_{\Upsigma}\cup \upmu_{K}$ et l'analogue des id\'{e}aux mod\`{e}les 
$\mathfrak{a}_{\bf x}(u )$ se d\'{e}finit par l'intermédiaire d'un vecteur ${\bf x}\in \R^{r+s-k}$ .  
\end{note}

\section{Mono\"{\i}des des classes d'id\'{e}aux quasicristallins}\label{monoides}
Dans ce qui suit on \'{e}tudie l'arithm\'{e}tique des id\'{e}aux quasicristallins.
On montre tout d'abord l'existence d'un produit mono\"{\i}dal :  

 \begin{prop} Soient $\mathfrak{a},\mathfrak{b}\subset A\subset \R^{n}$ deux id\'{e}aux quasicristallins contenus dans un anneau quasicristallin $A$.  
  Alors, le produit mono\"{\i}dal
\[ \mathfrak{a}\cdot\mathfrak{b}:=\{ \upalpha\upbeta|\; \upalpha\in\mathfrak{a},\; \upbeta\in\mathfrak{b}\}\subset \mathfrak{a}\cap\mathfrak{b}\] 
est aussi un quasicristal. \end{prop}

\begin{proof}[D\'{e}monstration] 
On note tout d'abord que $\mathfrak{a}\cdot\mathfrak{b}$ est relativement dense  car
\[ \upalpha\cdot\mathfrak{b}\subset \mathfrak{a}\cdot \mathfrak{b} \]
pour n'importe quel $\upalpha\in \mathfrak{a}$ non z\'{e}ro, et $\upalpha\cdot\mathfrak{b}$ est relativement dense. 
On rappelle ensuite qu'un ensemble relativement dense est un quasicristal si et seulement si il est contenu 
dans un ensemble mod\`{e}le (g\'{e}n\'{e}ral) $\mathcal{M}$ (voir \cite{Moody}, Theorem 9.1.i).
En particulier, $A\subset \mathcal{M}$ pour un certain ensemble mod\`{e}le $\mathcal{M}$.
Mais alors $\mathfrak{a}\cdot\mathfrak{b}\subset A\subset \mathcal{M}$, 
de sorte que $\mathfrak{a}\cdot\mathfrak{b}$ est bien un quasicristal.

 \end{proof}

Soit $A=A_{\upsigma}\subset K$ un anneau quasicristallin avec les notations de l'exemple 3 ci-dessus. 
Il est clair que si $\mathfrak{a}$ est un id\'{e}al quasicristallin fractionnaire, il en est de m\^eme de 
$\upgamma\mathfrak{a}$ pour tout $\upgamma\in K$.  La classe projective 
 \[ [\mathfrak{a}]=\{ \upgamma\mathfrak{a}|\; \upgamma\in K\}\] s'appelle la 
 {\em classe de l'id\'{e}al quasicristallin}.  L'ensemble des 
 classes d'id\'{e}aux quasicristallins de $A$ s'\'{e}crit 
 \[ {\sf Cl}(A)=\{ [\mathfrak{a}]\ |\; \mathfrak{a}\text{ un id\'{e}al quasicristallin
 fractionnaire de }A\}.\]   

\begin{prop}  
${\sf Cl}(A) $ est un mono\"{\i}de multiplicatif avec élément identit\'{e}
$1=[A]$. Cependant ce monoïde 
 n'est pas un groupe. En particulier :
 \begin{enumerate}
 \item[1.] les $\mathfrak{a}_{\bf x}$ ne sont pas inversibles pour tout ${\bf x}\in K_{\upsigma'}$ ;  
 \item[2.]  Supposons qu'il existe $\upalpha\in\overline{\mathfrak{a}}_{\bf x}$ tel que $\uppi'(\upalpha)=\uppi'(u)^{\bf x}\in\mathcal{O}_{K}^{\times}$ pour tout $i$ ; alors  $\overline{\mathfrak{a}}_{\bf x}$ est inversible d'inverse $\overline{\mathfrak{a}}_{-{\bf x}}$.
 \end{enumerate}
\end{prop}

\begin{proof}[D\'{e}monstration]   1. Si $\mathfrak{b}$ est un inverse pour $\mathfrak{a}$, il existe $\upalpha\in \mathfrak{a}$ tel que
$\upalpha^{-1}\in\mathfrak{b}$, puisque le produit $\mathfrak{a}\mathfrak{b} = A$ contient $1$.  
Soit $\upbeta\in\mathfrak{a}$ tel que
$|\upalpha_{i}'|<|\upbeta_{i}'|<u_{i}^{-x_{i}}$ pour tout $i$ ; un tel élément existe parce que 
$\mathcal{O}_{K}\subset K_{\upsigma}'$ est dense.  Alors $\upbeta/\upalpha\in \mathfrak{a}\mathfrak{b}$ mais
n'est pas élément de $A$ ; contradiction.    
2. On peut supposer que $\upalpha\in\mathcal{O}_{K}^{\times}$, puis
 $1=\upalpha\cdot \upalpha^{-1}\in\overline{\mathfrak{a}}_{\bf x}\cdot \overline{\mathfrak{a}}_{-{\bf x}}$, ce qui implique que  $\overline{\mathfrak{a}}_{\bf x}\cdot \overline{\mathfrak{a}}_{-{\bf x}}= A$.
\end{proof}


Soit $\mathfrak{a}$ un id\'{e}al quasicristallin et soit 
\[ \mathfrak{A} := \mathfrak{a} \mathcal{O}_{K} \]
son {\it extension} en un $\mathcal{O}_{K}$-id\'{e}al, pour définition l'ensemble des sommes finies de la forme $\sum_{i} \upgamma_{i}\upalpha_{i}$, $\upgamma_{i}\in\mathcal{O}_{K}$, $\upalpha_{i}\in 
\mathfrak{a}$.  On note que
$  \mathfrak{a}\mathfrak{b}\mathcal{O}_{K} = (\mathfrak{a}\mathcal{O}_{K}) ( \mathfrak{b}\mathcal{O}_{K}) $ (produit des idéaux),
d'où le fait que cette opération d'extension
 induit un morphisme de mono\"{\i}des
\[ \Upphi: {\sf Cl}(A)\longrightarrow {\sf Cl}(K) .\]

\begin{prop} L'application $\Upphi$ est un épimorphisme de mono\"{\i}des. Son noyau contient l'ensemble 
$Z$ defini dans la {\it Note} {\rm \ref{DefZ}}.
\end{prop}

\begin{proof}[D\'{e}monstration] Choisissons $u\in A$ une unit\'{e} PV.  
Si $\mathfrak{B}=(\upalpha,\upbeta )\subset\mathcal{O}_{K}$ est un id\'{e}al, on peut supposer après multiplication
 par $\upgamma=u^{n}\in\mathcal{O}_{K}^{\times}$, que $\upalpha,\upbeta$ satisfont $|\uppi'( \upalpha)|,|\uppi' 
 (\upbeta)|<{\bf 1}$, i.e.\ $\upalpha,\upbeta\in A$.
Alors $\mathfrak{b}:=\mathfrak{B}\cap A$ d\'{e}finit un \'{e}l\'{e}ment de ${\sf Cl}(A)$ tel que 
$\Upphi (\mathfrak{b})=\mathfrak{B}$.
Puisque tout $\mathfrak{a}\in Z$ contient une unit\'{e} de la forme $\upgamma =u^{n}$, 
$Z$ est bien contenu dans le noyau de $\Phi$.
\end{proof}


 
 On note que le produit mono\"{\i}dal est ``incomplet'' dans le sens qu'il ne co\"{\i}ncide pas avec le produit des 
 id\'{e}aux quand $\mathfrak{a}$, $\mathfrak{b}$ en sont. Par exemple, on peut faire
 toutes les constructions ci-dessus en caract\'{e}ristique positive, en considérant les extensions finies 
 ${\bf K}$ du corps de fonctions ${\bf Q}=\F_{q}(T)$.  
 En ce cas les ensembles mod\`{e}les sont des id\'{e}ax d'anneaux du type ${\bf A}_{\infty_{1}}$ 
 (voir au \S 1 et la note 2 ci-dessus) mais le produit introduit ici n'est pas le produit de ces id\'{e}aux :
 il  manque ici l'operation d'addition qui n'est pas assez bien contrôlée en caract\'{e}ristique nulle, 
 dans le cadre de quasicristaux. Cependant si l'on se restreint au id\'{e}aux quasicristallins qui sont 
 des ensembles mod\`{e}les, il est possible de d\'{e}finir un produit qui incorpore la somme d'une 
 manière satisfaisante, et tel qu'il co\"{\i}ncide  avec le produit des id\'{e}aux lorsque 
 $\mathfrak{a}$, $\mathfrak{b}$ en sont.  Dans ce qui suit, nous d\'{e}veloppons la définition d'un tel produit.
 
Soient $\mathfrak{A}$ un $\mathcal{O}_{K}$-id\'{e}al fractionnaire (r\'{e}alis\'{e} comme un r\'{e}seau de $K_{\infty}$) et
 $D \subset K_{\upsigma'}$ une fen\^{e}tre qui est un produit d'intervalles de la forme 
 \[ D^{\boldsymbol +}_{\bf x}(u):=
  \prod (-u^{x_{i}},u^{x_{i}})_{+_{i}}, \quad (-u^{x_{i}},u^{x_{i}})_{+_{i}} =
 \left\{ \begin{array}{ll}
 (-u^{x_{i}},u^{x_{i}}) & \text{si $+_{i}=0$} \\
  \text{$[ -u^{x_{i}},u^{x_{i}} ]$}& \text{si $+_{i}=1$}  
 \end{array}
 \right. .\]
 Soit
 \[ \mathfrak{a} =\mathfrak{a}^{\boldsymbol +}_{\mathfrak{A},{\bf x}}(u)=\mathcal{M} (\mathfrak{A}, D) 
 \] l'ensemble
 mod\`{e}le associ\'{e}.   On observe que le choix de la fen\^{e}tre $D$ n'est pas forc\'{e}ment unique : 
 par exemple, il est possible (en fait probable) qu'on ait l'\'{e}galit\'{e}
 \[ \mathcal{M} (\mathfrak{A}, D^{\boldsymbol +}_{\bf x}(u))=\mathcal{M} (\mathfrak{A}, D_{\bf x}(u))\]
pour tout $\boldsymbol +$, o\`{u} $ D_{\bf x}(u)$ correspond \`{a} ${\boldsymbol +}=(0,\dots ,0)$.
Dans ce cas, on dit que l'ensemble mod\`{e}le est {\em g\'{e}n\'{e}rique} et on le spécifie comme ensemble 
mod\`{e}le, en privilégiant la fen\^{e}tre  la plus petite, c'est-\`{a}-dire $D_{\bf x}(u)$, plutôt 
que les autres $D^{\boldsymbol +}_{\bf x}(u)$.  En g\'{e}n\'{e}ral, il existe, pour chaque ensemble 
mod\`{e}le de ce type, une fenêtre plus petite: l'intersection
\[  \bigcap_{\mathfrak{a}=\mathcal{M}(\mathfrak{A},D^{\boldsymbol +}_{\bf x}(u)) }D^{\boldsymbol +}_{\bf x}(u),\]
qui est aussi de la forme $D^{\boldsymbol +}_{\bf x}(u)$ por un choix appropri\'{e} de $\boldsymbol +$.
De cette fa\c{c}on, chaque ensemble 
mod\`{e}le basé sur le réseau $\mathfrak{A}$ est d\'{e}termin\'{e} et se d\'{e}termine par sa  fen\^{e}tre plus petite. 
Dans le cas g\'{e}n\'{e}rique on note simplement $\mathfrak{a}$ pour
 \[ \mathfrak{a}_{\mathfrak{A},{\bf x}}(u) = \mathcal{M} (\mathfrak{A}, D^{\boldsymbol +}_{\bf x}(u)))=\mathcal{M} (\mathfrak{A}, D_{\bf x}(u)) ;\]
 dans les autres cas on précisera la distinction 
 \[  \mathfrak{a}^{\boldsymbol +}= \mathfrak{a}^{\boldsymbol +}_{\mathfrak{A},{\bf x}}(u) 
 \supsetneq  \mathfrak{a}=
 \mathfrak{a}_{\mathfrak{A},{\bf x}}(u).
 \]
 
 \begin{lemm}\label{uniquerep}  Si $\mathcal{M}(\mathfrak{A}, D)=\mathcal{M}(\mathfrak{B}, D')$ où $D, D'$ son les fenêtres plus petites (au sens expliqué dans les paragraphes dessus) alors $\mathfrak{A}=\mathfrak{B}$ et $D=D'$. 
 \end{lemm}
 
 \begin{proof} Au moins un des ensembles $\mathfrak{A}\setminus\mathfrak{B}$, $\mathfrak{B}\setminus\mathfrak{A}$ est infini si $\mathfrak{A}\not=\mathfrak{B}$: supposons que
 $\mathfrak{A}\setminus\mathfrak{B}$ est infini.   Soit $D''=D\cap D'$.  Puisque l'action de $u\in\mathcal{O}_{K}^{\times}$
 stabilise $\mathfrak{A}\setminus\mathfrak{B}$, il en suit qu'il existe $\upalpha\in \mathcal{M}(\mathfrak{A}, D'')\setminus \mathcal{M}(\mathfrak{B}, D')\subset \mathcal{M}(\mathfrak{A}, D)\setminus \mathcal{M}(\mathfrak{B}, D')$.   Donc $\mathfrak{A}=\mathfrak{B}$ et $D=D'$.
 \end{proof}

 On dit que l'ensemble mod\`{e}le non g\'{e}n\'{e}rique $\overline{\mathfrak{a}}$ est {\em ferm\'{e}}; si de plus il existe $\upalpha\in\overline{\mathfrak{a}}$ tel que $\upalpha^{-1}\in\mathfrak{A}^{-1}$ et
 $|\uppi'(\tilde{u})| = |\uppi'(u)|^{-{\bf x}}$, on dit que $\overline{\mathfrak{a}}$ est {\it unitaire}.
  Dans tous les cas on appelle $\mathfrak{a}$
un {\em id\'{e}al quasicristallin mod\`{e}le} : c'est bien un id\'{e}al quasicristallin de $A$ au sens précédent. 

Soient maintenant $\mathfrak{a}=\mathcal{M}(\mathfrak{A}, D)$, $\mathfrak{b}=\mathcal{M}(\mathfrak{B}, D')$ 
deux id\'{e}aux quasicristallins mod\`{e}les.
On peut alors d\'{e}finir un produit qui se restreint au produit usuel des id\'{e}aux fractionnaires 
quand $\mathfrak{a}$, $\mathfrak{b}$ en sont:
\[ \mathfrak{a}\ast\mathfrak{b}:=\mathcal{M}  (\mathfrak{A}\ast\mathfrak{B}, D\cdot D' )\supset \mathfrak{a}\cdot\mathfrak{b} ,\]
o\`{u} $\mathfrak{A}\ast\mathfrak{B}$ est le produit usuel d'id\'{e}aux fractionnaires et \[ D\cdot D'=\{ {\bf x}={\bf t}{\bf t}'|\; {\bf t}\in D,\; {\bf t}'\in D'\}.\]

En raison de notre convention pour le choix de fen\^{e}tre, le produit $ \mathfrak{a}\ast\mathfrak{b}$ est bien d\'{e}fini
et c'est aussi un id\'{e}al quasicristallin mod\`{e}le. L'ensemble des id\'{e}aux quasicristallins mod\`{e}les, soit
 \[ {\sf I}^{\sf mod}(A) ,\] constitue un mono\"{\i}de lorsqu'il est muni du 
 produit $\ast$ dont  l'\'{e}l\'{e}ment neutre est $1:=A=\mathcal{M} (\mathcal{O}_{K}, D_{\boldsymbol 0})$.

\begin{prop} L'ensemble 
 \[{\sf I}^{\sf mod}_{0} (A) := \{\text{id\'{e}aux quasicristallins mod\`{e}les unitaires}\} \]
  est un sous-groupe du monoïde ${\sf I}^{\sf mod}(A)$.
  \end{prop}
  \begin{proof}[Démonstration]  
  Soient $ \overline{\mathfrak{a}}_{\mathfrak{A},{\bf x}} \in {\sf I}^{\sf mod}_{0} (A)$ et $\upalpha\in  \overline{\mathfrak{a}}_{\mathfrak{A},{\bf x}} $
  l'élément qu'en fait unitaire, alors $\upalpha^{-1}\in\overline{\mathfrak{a}}_{\mathfrak{A}^{-1},-{\bf x}}$ en fait du dernier unitaire et on a 
  \[  \overline{\mathfrak{a}}_{\mathfrak{A},{\bf x}}\ast \overline{\mathfrak{a}}_{\mathfrak{A}^{-1},-{\bf x}}=A=1.\]
  Si $ \overline{ \mathfrak{a}}_{\mathfrak{A},{\bf x}},  \overline{\mathfrak{b}}_{\mathfrak{B},{\bf y}}$ sont unitaires par $\upalpha$, $\upbeta$, alors $\overline{\mathfrak{a}}_{\mathfrak{A}\ast\mathfrak{B},
 {\bf x}+{\bf y}} $
 est unitaire par rapport à $\upalpha\upbeta$ 
et
 \[ \overline{ \mathfrak{a}}_{\mathfrak{A},{\bf x}}\ast \overline{\mathfrak{b}}_{\mathfrak{B},{\bf y}} = \overline{\mathfrak{a}}_{\mathfrak{A}\ast\mathfrak{B},
 {\bf x}+{\bf y}}.  \]
\end{proof}

On appellera {\em principal} l'id\'{e}al quasicristallin
\[ \upalpha A:=
 \mathcal{M}(\upalpha\mathcal{O}_{K}, D_{\upalpha})\]
o\`{u} $\upalpha\in K^{\times}$ et
\[ D_{\upalpha} = [-|\upsigma_{1}(\upalpha)|,|\upsigma_{1}(\upalpha)|]\times \dots \times   [-|\upsigma_{r+s-1}(\upalpha)|,|\upsigma_{r+s-1}(\upalpha)|].\]
De manière équivalente, si l'on note
\begin{align}\label{lognotation} {\bf x}(\upalpha )& := \log_{|\uppi'(u)|} (|\uppi'(\upalpha)|) \\ 
& := \left(  \log_{|\upsigma_{1}(u)| }|\upsigma_{1}(\upalpha) | ,\dots ,
 \log_{|\upsigma_{r+s-1}(u)| }|\upsigma_{r+s-1}(\upalpha) | 
 \right),\nonumber \end{align}
 alors on a
 \[ \upalpha A = \overline{\mathfrak{a}}_{(\upalpha ),{\bf x}(\upalpha )}. \]

On remarque que pour tout $\upalpha\in A$, $\upalpha A\subset A$, autrement dit $\upalpha A$ est un id\'{e}al 
quasicristallin entier.  
En g\'{e}n\'{e}ral les id\'{e}aux quasicristallins principaux sont
tous unitaires et pour cela forment un sous-groupe
\[  {\sf P}^{\sf mod}(A)= \left\{ \upalpha A
:\;\; \upalpha\in K^{\times}\right\} \subset {\sf I}^{\sf mod}_{0} (A) \subset {\sf I}^{\sf mod}(A) .
 \]
On note le quotient
\[{\sf Cl}^{\sf mod}( A) :=  {\sf I}^{\sf mod}( A)/ {\sf P}^{\sf mod}( A) ;\]
c'est le {\em mono\"{\i}de des classes d'id\'{e}aux quasicristallins mod\`{e}les}. Il contient  
\begin{align}\label{inversibles} {\sf Cl}_{0}^{\sf mod}( A) ,\end{align}
 le sous-groupe des classes d'\'{e}l\'{e}ments de ${\sf I}^{\sf mod}_{0} (A)$.





Soient $u_{1},\dots ,u_{r+s-1}$ un choix de syst\`{e}me d'unit\'{e}s fondamentales.  Les vecteurs  
\[ {\bf x}(u_{1}), \dots , {\bf x}(u_{r+s-1})  \, ,\]
 de $\R^{r+s-1}$ (voir (\ref{lognotation}) ci-dessus)
 sont linéairement independants ; on note $\Upupsilon$ le groupe additif qu'ils engendrent.   
 
 \begin{lemm}\label{DirLemma} $\Upupsilon\subset \R^{r+s-1}$ est un r\'{e}seau.\footnote{Ce lemme est essentielmente le th\'{e}or\`{e}me d'unités de Dirichlet avec une modification petit: 
 on utilise $\log_{|\uppi'(u)|}$ au lieu du logarithme naperian.}
 \end{lemm} 
 
 \begin{proof}[D\'{e}monstration]  Le rang de $\Upupsilon$ comme groupe abelien est $r+s-1$ ; reste à 
 montrer que $\Upupsilon$ est discret.  Consid\`{e}rons le cube centr\'{e}e \`{a} l'origine
 $$C=\{{\bf x}:\; |x_{i}|\leq R\} \subset \R^{r+s-1}$$ 
 avec $R>0$ : il suffit de vérifier que $\Upupsilon\cap C$ est fini.  
 La préimage $D\subset \R_{+}^{r+s-1}$ de $C$ par l'inverse du logarithme est un produit : 
 $$D=\{ {\bf y}|\; r_{1}\leq |y_{i}|\leq r_{2}\}$$ 
 et donc  la préimage $\widetilde{D}$ de $D$ dans $K_{\upsigma'}$ pour la valeur absolue 
 $$|\cdot |:K_{\upsigma'}\rightarrow \R_{+}^{r+s-1}$$
 est un produit d'anneaux de dimensions 1 ou 2, uniformément éloignés de l'origine comme de l'infini. 
 On en conclut que tout \'{e}l\'{e}ment $\uppi'(\updelta )\in\uppi' (\mathcal{O}_{K}^{\times})\cap \widetilde{D}$ 
satisfait des inégalités  
$$0<c_{1}<|\updelta|< c_{2}$$
 avec $c_{1}, c_{2}$ des constants qui ne d\'{e}pendent que de $C$ (parce que $\updelta\cdot \prod \upsigma_{i}(\updelta )=\pm1$).  En particular, la préimage de $\uppi' (\mathcal{O}_{K}^{\times})\cap \widetilde{D}$ par $\uppi'$ 
 est contenue dans un  compact et contient donc un nombre fini d'\'{e}l\'{e}ments de 
 $\mathcal{O}_{K}^{\times}\subset\mathcal{O}_{K}$.  Ainsi $\Upupsilon\cap C$ est bien un ensemble fini. 
 \end{proof}

On note le tore quotient
 \[ \T_{\Upupsilon}:=\R^{r+s-1}/\Upupsilon. \]
Du fait de l'identit\'{e}
 \[    u_{i} \cdot \mathfrak{a}^{\boldsymbol +}_{\mathcal{O}_{K},{\bf x}} =\mathfrak{a}^{\boldsymbol +}_{\mathcal{O}_{K},{\bf x}+{\bf x}(u_{i})} ,
 \quad i=1,\dots ,r+s-1 ,\]
 l'ensemble
\begin{align}\label{defZ} {\sf Z} :=  \left\{ \mathfrak{a}^{\boldsymbol +}_{\mathcal{O}_{K},{\bf x}} \mod \text{${\sf P}^{\sf mod} $}( A)\right\}_{{\bf x}\in\T_{\Upupsilon}, \;\boldsymbol +}\end{align}
est un sous-monoïde de ${\sf Cl}^{\sf mod}( A)$.

\begin{theo}\label{suiteexacte} L'application
 \[  {\sf Cl}^{\sf mod}( A) \longrightarrow {\sf Cl}(K),\quad \mathfrak{a}\longmapsto \mathfrak{a} \mathcal{O}_{K}= \text{\rm le $\mathcal{O}_{K}$-id\'{e}al engendr\'{e} par $\mathfrak{a}$} ,\]
 est surjective, de noyau le sous-mono\"{\i}de ${\sf Z}$.
 \end{theo}
 
 \begin{proof}[D\'{e}monstration]  On commence pour montrer que l'application est bien un homomorphisme.
Soient $\mathfrak{a},\mathfrak{b}$ deux id\'{e}aux quasicristallins mod\`{e}les : il faut montrer que
 \begin{align}\label{epimorphisme}  (\mathfrak{a}\ast\mathfrak{b})\mathcal{O}_{K}=\left(\mathfrak{a}\mathcal{O}_{K}\ast\mathfrak{b }\mathcal{O}_{K}\right)  \mod \text{${\sf P}$}(K), \end{align}
 o\`{u} $ {\sf P}(K)$ désigne l'ensemble des id\'{e}aux fractionnaires principaux de $\mathcal{O}_{K}$.  
 En multipliant par des \'{e}l\'{e}ments appropri\'{e}s de $\mathcal{O}_{K}$,  on peut supposer sans perte de 
 g\'{e}n\'{e}ralit\'{e} que $\mathfrak{a},\mathfrak{b}$ sont des quasicristaux entiers i.e.\ des ensembles mod\`{e}les bas\'{e}s sur les id\'{e}aux entiers  $\mathfrak{A},\mathfrak{B}\subset\mathcal{O}_{K}$ avec 
 des fen\^{e}tres  paramétr\'{e}es par ${\bf x}, {\bf y} >0$. 
 Supposons pour plus de clart\'{e} que $\mathfrak{a}$ et $\mathfrak{b}$ sont tous deux g\'{e}n\'{e}riques ; 
 la preuve dans les autres cas est identique.  Soient alors $\upalpha_{1},\upalpha_{2},\dots$ resp.\ 
 $\upbeta_{1},\upbeta_{2},\dots$, resp.\  $\upomega_{1},\upomega_{2},\dots$ une liste de tous les \'{e}l\'{e}ments positifs de $\mathfrak{a}$, resp.\ $\mathfrak{b}$, resp.\ $\mathfrak{a}\ast\mathfrak{b}$.
 Un \'{e}l\'{e}ment de $(\mathfrak{a}\ast\mathfrak{b})\mathcal{O}_{K}$ est donc de la forme
 \[ \upeta =  \sum_{I} \upgamma_{I}\upomega_{I} ,\quad \upomega_{I}:= \prod_{i\in I}  \upomega_{i} , \]
 o\`{u} $I\subset \N$ est fini  et $\upgamma_{I}\in\mathcal{O}_{K}$ est nul pour presque tout $I\subset \N$.  
 Par définition de $\mathfrak{a}\ast\mathfrak{b}$ comme ensemble mod\`{e}le bas\'{e} sur le r\'{e}seau $\mathfrak{A}\ast\mathfrak{B}$, on a 
 \[ \upomega_{i}=\sum_{j} \upalpha_{ij}\upbeta_{ij},\quad  \upalpha_{ij} \in \mathfrak{A},\;\; \upbeta_{ij}\in\mathfrak{B},\quad |\uppi'(\omega_{i})|<|\uppi'(u)|^{-({\bf x}+{\bf y})}.\]  
 On note que $\upalpha_{ij}\in\mathfrak{A}$, resp.\ $\upbeta_{ij}\in\mathfrak{B}$ n'est pas forc\'{e}ment un \'{e}l\'{e}ment de $\mathfrak{a}$, resp.\ $\mathfrak{b}$ ; il est possible par exemple
 que $|\uppi'(\upalpha_{ij})|\not< |\uppi'(u)|^{-{\bf x}}$.  N\'{e}anmoins,  en multipliant  $\upeta$ par $1=u^{-2MN}u^{2MN}$, $M\geq |I|$ pour tout $I$ o\`{u} $\upgamma_{I}\not=0$, on obtient
 \begin{align*} \upeta &  =
 \sum_{I} u^{-2NM}\upgamma_{I}u^{2NM}\prod_{i\in I} \left(\sum_{j} \upalpha_{ij}\upbeta_{ij}\right)  \\
& = \sum_{I} u^{-2N|I|}\upgamma_{I}\prod_{i\in I} \left(\sum_{j} (u^{N}\upalpha_{ij})(u^{N}\upbeta_{ij})\right)  
 \end{align*}
 et pour $N$ suffisamment grand, $u^{N}\upalpha_{ij}\in\mathfrak{a}$ et $u^{N}\upbeta_{ij}\in\mathfrak{b}$.  
 On peut donc supposer sans perte de généralité 
 que $\upalpha_{ij}\in\mathfrak{a}$, $\upbeta_{ij}\in\mathfrak{b}$.  Puis
 \begin{align}
 \upeta &=  \sum_{I} \upgamma_{I}\prod_{i\in I}\left(\sum_{j} \upalpha_{ij}\upbeta_{ij}\right) \label{1ligne} \\
 &= \sum_{I}\upgamma_{I} \sum\left(\prod \upalpha_{ij}\prod \upbeta_{ij}\right) \label{2ligne}
 \end{align}
 o\`{u} les sommes de produits qui apparaissent à la ligne (\ref{2ligne}) proviennent du dévelop\-pement 
des produits de sommes de la ligne (\ref{1ligne}).  Il est alors clair 
que l'expression (\ref{2ligne}) appartient au produit d'id\'{e}aux $\mathfrak{a}\mathcal{O}_{K}\ast \mathfrak{b}\mathcal{O}_{K}$.  Autrement dit
 \[ (\mathfrak{a}\ast\mathfrak{b})\mathcal{O}_{K}\subset\left(\mathfrak{a}\mathcal{O}_{K}\ast\mathfrak{b }\mathcal{O}_{K}\right)  \mod \text{${\sf P}$}(K) . \]
Inversement soit $ \upeta\in \mathfrak{a}\mathcal{O}_{K}\ast \mathfrak{b}\mathcal{O}_{K}$; on peut \'{e}crire
 \[   \upeta = \sum_{k}\left[ \left( \sum_{I} \upgamma_{I,k}\upalpha_{I}\right)\cdot \left( \sum_{J}\upgamma_{J,k}\upbeta_{J}\right)\right],
  \quad \sum \upgamma_{I,k}\upalpha_{I} \in \mathfrak{a}\mathcal{O}_{K},\;\;  \sum\upgamma_{J,k}\upbeta_{J}\in\mathfrak{b}\mathcal{O}_{K} ,\]
  o\`{u} $I,J\subset \N$ sont finis et $\upalpha_{I}$, $\upbeta_{J}$ sont des produits d'\'{e}l\'{e}ments de $\mathfrak{a}$, $\mathfrak{b}$, $\upgamma_{I,k},\upgamma_{J,k}\in \mathcal{O}_{K}$.  En d\'{e}veloppant les produits, on obtient des $\mathcal{O}_{K}$ combinations linaires  d'expressions de la forme $\upalpha_{I}\upbeta_{J}\in 
  (\mathfrak{a}\ast\mathfrak{b}) \mathcal{O}_{K}$.  Alors, $\upeta\in (\mathfrak{a}\ast\mathfrak{b}) \mathcal{O}_{K}$ et on a prouv\'{e}  l'\'{e}galit\'{e} (\ref{epimorphisme}) i.e.\ que l'application
  du th\'{e}or\`{e}me est un homomorphisme. Si $\mathfrak{A}=(\upalpha,\upbeta)$ est un $\mathcal{O}_{K}$-id\'{e}al, l'id\'{e}al quasicristallin mod\`{e}le 
  \[ \mathfrak{a}=\mathfrak{a}_{\mathfrak{A},{\bf x}}, \quad
  |\uppi'(\upalpha)|, \;|\uppi'(\upbeta)|<|\uppi'(u)|^{-{\bf x}}\] satisfait $\mathfrak{A}=\mathfrak{a}\mathcal{O}_{K}$, et donc
   l'application est également surjective.
 Tout élément de ${\sf Z}$ est entier et contient une unit\'{e} (une puissance de $u$), d'où le fait que $\sf Z$ est
 contenu dans le noyau.  Inversement si 
 $\mathfrak{a} =\mathfrak{a}_{\mathfrak{A},{\bf x}}$ est dans le noyau, $\mathfrak{A}\in {\sf P}(K)$ 
 et il existe $\upgamma\in K$ tel que $\upgamma\mathfrak{a}=\mathfrak{a}_{\mathcal{O}_{K},{\bf y}}$. 
 \end{proof}
 

Dans ce qui suit nous munirons $ {\sf Cl}^{\sf mod}( A)$ d'une topologie d'ensemble de Cantor. 
Soit ${\sf I}_{\mathfrak{A}}^{\sf mod}(A)$ l'ensemble des id\'{e}aux quasicristallins de $A$ bas\'{e}s 
sur $\mathfrak{A}\subset K$,  $\mathcal{O}_{K}$-id\'{e}al fractionnaire :
 \[ {\sf I}_{\mathfrak{A}}^{\sf mod}(A) = \{ \mathfrak{a}=\mathcal{M} (\mathfrak{A}, D)\}\subset {\sf I}^{\sf mod}(A).\] 
 (Par le Lemma \ref{uniquerep}, il existe un seul $\mathfrak{A}$ tel que $\mathfrak{a}=\mathcal{M} (\mathfrak{A}, D)$.)
  On peut identifier chaque element $\mathfrak{a}$ de
 $  {\sf I}_{\mathfrak{A}}^{\sf mod}(A)$ avec sa fonction caract\'{e}ristique 
 \[ \upchi_{\mathfrak{a}}:\mathfrak{A}\longrightarrow \{ 0,1\},\quad  \upchi_{\mathfrak{a}}(\upalpha) = 1\Leftrightarrow \upalpha\in\mathfrak{a}. \] 
 On munit $ {\sf I}_{\mathfrak{A}}^{\sf mod}(A)$ de la topologie induite du plongement  $ {\sf I}_{\mathfrak{A}}^{\sf mod}(A)\hookrightarrow {\sf 2}^{\mathfrak{A}}$, o\`{u} ${\sf 2}=\{ 0,1\}$ et ${\sf 2}^{\mathfrak{A}}$ 
 a la topologie du produit.
Chaque inclusion $\mathfrak{A}\subset\mathfrak{B}$ induit un plongement ${\sf 2}^{\mathfrak{A}}\hookrightarrow {\sf 2}^{\mathfrak{B}}$ qui est un homeomorphisme sur son image, defini en prolongeant  $f:\mathfrak{A}\rightarrow {\sf 2}$ \`{a}
$\mathfrak{B}\setminus \mathfrak{A}$
par $0$.  Notons que par la demonstration du Lemme \ref{uniquerep}, l'image de ${\sf I}_{\mathfrak{A}}^{\sf mod}(A) $ en ${\sf 2}^{\mathfrak{B}}$ est disjoint de ${\sf I}_{\mathfrak{B}}^{\sf mod}(A) $.
C'est-à-dire, \[ {\sf I}^{\sf mod}(A)=\bigsqcup_{\mathfrak{A}} {\sf I}_{\mathfrak{A}}^{\sf mod}(A) \subset \lim_{\longrightarrow} {\sf 2}^{\mathfrak{A}}\] est
 topologis\'{e} comme union disjoint dans la limite inductive $\varinjlim {\sf 2}^{\mathfrak{A}}$.

\begin{prop}  L'action multiplicative de $K^{\times}$ sur ${\sf I}^{\sf mod}(A)$ induit une action bien-definie du quotient $K^{\times}/\upmu_{K}$ par homeomorphismes, qui est propre et discontinue.  
 La topologie quotient en ${\sf Cl}^{\sf mod}( A)$ est de Hausdorff.
 \end{prop}
 
 \begin{proof}[D\'{e}monstration]
 L'action de $\upmu_{K}$ laisse invariante chaque $\mathcal{O}_{K}$-idéal $\mathfrak{A}$ et
 chaque fenêtre $D$, donc sa action sur ${\sf I}^{\sf mod}(A)$ est triviale.
Chaque $\upgamma\in K^{\times}/\upmu_{K}$ agit sans des points fixes:  en effet, si $\upgamma\not\in\mathcal{O}^{\times}_{K}/\upmu_{K}$,
il est claire parce que $\upgamma\mathfrak{A}\not=\mathfrak{A}$ pour chaque $\mathcal{O}_{K}$-idéal; si $\upgamma\in\mathcal{O}^{\times}_{K}/\upmu_{K}$,  $\upgamma\not\in\upmu_{K}$, chaque fenêtre $D$ satisfait $\overline{\upgamma D}\not=\overline{D}$.  Alors, $\upgamma$ definit une bijection de ${\sf I}^{\sf mod}(A)$ qui respecte la topologie du dernier. 

Soit $\mathfrak{a}\in {\sf I}_{\mathfrak{A}}^{\sf mod}(A)$ defini par la fenêtre $D$.   Soit \[ X=\{ \upgamma \in  \mathcal{O}^{\times}_{K}/\upmu_{K}|\; \upgamma D\subset D\},\quad X^{\complement} =  \left( \mathcal{O}^{\times}_{K}/\upmu_{K}\right) \setminus X. \] Par le Lemma \ref{DirLemma}, l'action multiplicative de $\uppi'( \mathcal{O}^{\times}_{K})$ n'a pas des points d'accumulation.  Donc, il existent
\begin{itemize}
\item $\upalpha\in\mathfrak{a}$ tel que $\uppi'(\upalpha )\not\in\bigcup_{\upgamma\in X} \upgamma D$.
\item $\upbeta\not\in\mathfrak{a}$ tel que $\uppi'(\upbeta )\in\bigcap_{\upgamma\in X^{\complement}} \upgamma D$
\end{itemize}
On définit l'ouvert $U\ni \mathfrak{a}$
de fonctions caractéristiques $\upchi$ qui satisfont
\begin{enumerate}
\item[1.] $\upchi (\upalpha )=1$.
\item[2.]  $\upchi\in  {\sf I}_{\mathfrak{A}}^{\sf mod}(A)$ i.e.\ $\upchi|_{K\setminus \mathfrak{A}}\equiv 0$ et en particulier, $\upchi (\upbeta )=0$.
\end{enumerate}
 Si $\upgamma\not\in \mathcal{O}_{K}^{\times}/\upmu_{K}$, $\upgamma U\subset  {\sf I}_{\upgamma \mathfrak{A}}^{\sf mod}(A)$ qui a l'intersection nulle avec $U\subset  {\sf I}_{\mathfrak{A}}^{\sf mod}(A)$.
Si $\upgamma\in X$, $\upchi_{\upgamma\mathfrak{a}}(\upalpha )=0$ i.e.\ $\upgamma U\cap U$ est nulle; de mode pareil, si $\upgamma\in X^{\complement}$,  $\upchi_{\upgamma\mathfrak{a}}(\upbeta )=1$ et
$\upgamma U\cap U$ est aussi nulle.   On conclut que l'action est propre et discontinue.
 \end{proof}

 \begin{note}\label{codages} Lorsque $K/\Q$ est quadratique et r\'{e}elle, $\mathfrak{a}=\mathfrak{a}_{\mathfrak{A},{\bf x}}^{\boldsymbol +}$, ${\bf x}\geq \boldsymbol 0$ et $\mathfrak{A}\subset\mathcal{O}_{K}=\Z[\uptheta]$, $\uptheta$ une unit\'{e} fondamentale,
 on peut identifier canoniquement $\upchi_{\mathfrak{a}}$ 
avec une {\em fonction de codage} \[ \overline{\upchi}_{\mathfrak{a}}:\N\longrightarrow \{ 0,1\}.\]
 Plus pr\'{e}cisement, $ \overline{\upchi}_{\mathfrak{a}}(a)=1$ si et seulement si il existe $b\in\N$ 
tel que $\upchi_{\mathfrak{a}}(a\uptheta +b)=1$.  La correspondance $\mathfrak{a}\mapsto \overline{\upchi}_{\mathfrak{a}}$ est injective sur ${\sf Z}$ car la projection $\upalpha = a\uptheta+b\rightarrow a$ 
est injective sur $\mathcal{O}_{K}$. 
\end{note}

Soit maintenant
\[ {\sf Z}_{0} := {\sf Z}\cap{\sf Cl}_{0}^{\sf mod}( A),\]
o\`{u} ${\sf Cl}_{0}^{\sf mod}( A)$ est d\'{e}fini dans (\ref{inversibles}). On a alors le résultat suivant :

 \begin{theo} ${\sf Cl}^{\sf mod}( A)$ est un mono\"{\i}de de Cantor\footnote{Un {\it mono\"{\i}de de Cantor} est un ensemble de Cantor qui est un mono\"{\i}de dont le produit est continue.} et ${\sf Cl}_{0}^{\sf mod}( A)$ est dense dans 
 ${\sf Cl}^{\sf mod}( A)$.  De plus le sous-groupe ${\sf Z}_{0}$ est dense dans ${\sf Z}$ et de rang topologique\footnote{Le {\it rang topologique} est le nombre minimal de générateurs d'un sous-monoïde dense.} 
 au plus $r+s-1$.  
 \end{theo}
 
 \begin{proof}[D\'{e}monstration]  Par le Th\'{e}or\`{e}me \ref{suiteexacte}, pour prouver que  ${\sf Cl}^{\sf mod}( A)$ est 
 de Cantor,  il suffit de montrer que c'est le cas de ${\sf Z}$. On note d'abord que comme espace topologique
 ${\sf Z}$ s'identifie homeomorphiquement avec le sous espace de ${\sf I}_{\mathcal{O}_{K}}^{\sf mod}(A)\subset {\sf 2}^{\mathcal{O}_{K}}$ d\'{e}finit par les fonctions $\upchi_{\mathfrak{a}^{\boldsymbol +}_{\bf x}}$, ${\bf x}\in \T_{\Upupsilon}$, (voir 
 (\ref{defZ})). 
 On commence pour prouver que ${\sf Z}$ est parfait : si $\mathfrak{a}_{\bf x}=\mathfrak{a}_{\mathcal{O}_{K}, {\bf x}}\in {\sf Z}$ est g\'{e}n\'{e}rique,
alors $\mathfrak{a}_{\bf x}$ est une limite à gauche :
\[ \lim_{{\bf y}>{\bf x}} \mathfrak{a}_{\bf y}=\mathfrak{a}_{\bf x}.\]  
En effet, si ${\bf y}>{\bf x}$, $\mathfrak{a}_{\bf y}\subset \mathfrak{a}_{\bf x}$ et par g\'{e}n\'{e}ricit\'{e},
tout $\upalpha\in \mathfrak{a}_{\bf x}$ appartient à $\mathfrak{a}_{\bf y}$ si ${\bf y}$ est assez proche \`{a} ${\bf x}$,  
ce qui implique que 
$\upchi_{\mathfrak{a}_{\bf y}}\rightarrow \upchi_{\mathfrak{a}_{\bf x}}$ dans la topologie faible (de Tychonoff).
Inversement, si $\overline{\mathfrak{a}}_{\bf x}$ est ferm\'{e}, on a $\mathfrak{a}_{\bf y}\supset \overline{\mathfrak{a}}_{\bf x}$ pour tout ${\bf y}<{\bf x}$ et
 \[ \lim_{{\bf y}<{\bf x}} \mathfrak{a}_{\bf y}=\overline{\mathfrak{a}}_{\bf x}.\] 
 Pour $\mathfrak{a}^{\boldsymbol +}_{\bf x}$ non g\'{e}n\'{e}rique g\'{e}n\'{e}ral, on consid\`{e}re ${\bf y}\rightarrow {\bf x}$ tel que $y_{i}>x_{i}$ si $+_{i}=0$
 et $y_{i}<x_{i}$ si $+_{i}=1$ puis $\lim \mathfrak{a}_{\bf y} =\mathfrak{a}^{\boldsymbol +}_{\bf x}$.  Le même genre d'argument montre que ${\sf Z}$ est fermé.
 L'ensemble ${\sf Z}$ étant parfait et fermé dans 
un ensemble de Cantor, il est lui aussi de Cantor. 
 L'ensemble ${\sf Z}_{0}$  est dense parce que l'ensemble des 
 param\`{e}tres ${\bf x}$ tels que $\uppi'(u)^{\bf x}\in \mathcal{O}_{K}$ est dense en $K_{\upsigma'}$
et d'après ce qui précède on peut trouver une suite d'éléments de ${\sf Z}_{0}$ qui converge vers 
n'importe quel element de ${\sf Z}$.
   Soient ${\bf x}_{1},\dots ,{\bf x}_{r+s-1}\in \T_{\Upupsilon}$ tels que \[ \uppi'(u)^{{\bf x}_{i}} =: \uppi'(\upalpha_{i})\in \uppi' \left( \mathcal{O}_{K}\setminus (\mathcal{O}^{\times}_{K}  \cup \Q)\right),\;\; i=1,\dots ,r+s-1,\]
   et tels que ${\bf x}_{1},\dots ,{\bf x}_{r+s-1}$ sont independants dans le sens 
   qu'il existe un choix de préimages $\tilde{\bf x}_{i}$ dans $\R^{r+s-1}$ qui forment une base ; c'est 
  possible, parce que l'ensemble  $ \uppi' \left( \mathcal{O}_{K}\setminus (\mathcal{O}^{\times}_{K}  \cup \Q) \right)$ 
  est dense en $ K_{\upsigma'}$. Observons qu'il n'existe {\it pas} de combination entière non triviale nulle
   \[ n_{1}{\bf x}_{1} +\cdots +n_{r+s-1}{\bf x}_{r+s-1}=0 \ ;\]
 en effet l'existence d'une telle relation impliquerait que $\prod\upalpha_{i}\in\mathcal{O}^{\times}_{K}$.  
 De plus, puisque il existe des préimages qui forment une base de $\R^{r+s-1}$, il suit
   que le groupe $T:=\langle {\bf x}_{1},\dots ,{\bf x}_{r+s-1}\rangle\subset \T_{\Upupsilon}$ est dense.
 Soient $\overline{\mathfrak{a}}_{{\bf x}_{1}},\dots ,\overline{\mathfrak{a}}_{{\bf x}_{r+s-1}}\in {\sf Z}_{0}$ 
 les id\'{e}aux quasicristallins ferm\'{e}s associ\'{e}s \`{a} ${\bf x}_{1},\dots ,{\bf x}_{r+s-1}$.
   Alors, d'après la densit\'{e} du groupe $T$ des param\`{e}tres et l'analyse de la convergence dans ${\sf Z}$
   en termes du param\`{e}tre ${\bf x}$ il suit que le groupe engendr\'{e} par les $\overline{\mathfrak{a}}_{{\bf x}_{i}}$
   est bien dense dans ${\sf Z}$. 
   
 \end{proof}
 

   


 \section{La fonction z\^{e}ta d'un id\'{e}al quasicristallin}\label{zeta}

Soit $\mathfrak{a}\subset\R$ un id\'{e}al quasicristallin de dimension 1; lorsque $\mathfrak{a}$ 
est un id\'{e}al mod\`{e}le nous le distinguerons en utilisant la notation 
\[ \mathfrak{m}=\mathcal{M}(\mathfrak{A}, D) = \mathcal{M}(K_{\upsigma}, K_{\upsigma'}, \mathfrak{A}, D)\subset\R, \] o\`{u} $D=D_{\bf x}^{\boldsymbol +}(u)\subset K_{\upsigma'}$ (voir \S\S \ref{anneauxqc}, \ref{monoides}
pour des pr\'{e}cisions).  
La {\em fonction z\^{e}ta} associée s'écrit
\[ \upzeta_{\mathfrak{a}}(s) = \sum_{0<\upalpha\in\mathfrak{a}} \upalpha^{-s}.\]
Dans cette section nous montrons l'existence d'un prolongement méromorphe 
de $\upzeta_{\mathfrak{a}}$ au plan tout entier. Nous donnons également, lorsque
$\mathfrak{a}=\mathfrak{m}$ est un id\'{e}al quasicristallin mod\`{e}le, 
une \'{e}quation fonctionnelle g\'{e}n\'{e}ralis\'{e}e pour $\upzeta_{\mathfrak{m}}$.
En vue de ce dernier objectif on consid\`{e}re plus g\'{e}n\'{e}ralement la situation suivante.  
Soit $f\in \mathcal{S}(K_{\upsigma'})$, l'espace de Schwartz de $K_{\upsigma'}$, 
avec $f$ paire i.e. telle que $f(-x)=f(x)$. 
 On d\'{e}finit les $f${\em -poids} sur $\mathfrak{m}$ (plus bas la dépendance par rapport à $f$ sera omise
de la notation) selon
\[  \upchi (\upalpha )=\upchi_{f} (\upalpha ) :=  f(\upalpha '),\quad \upalpha\in\mathfrak{m}, \;\; \upalpha':= \uppi'(\upalpha )\in D.\] 
La {\em fonction} L {\em quasicristalline} associ\'{e}e \`{a} $\mathfrak{m}$ et $\upchi$ est alors d\'efinie par
\[ L (\mathfrak{m},\upchi, s) := \sum_{0<\upalpha\in\mathfrak{m}} \frac{ \upchi (\upalpha)}{\upalpha^{s}}.\]

\begin{prop}\label{analyticcont} Soient $\mathfrak{a}$ un id\'{e}al quasicristallin de rang {\rm 1} et $\mathfrak{m}=\mathcal{M}(\mathfrak{A}, D)\subset\R$ un id\'{e}al quasicristallin mod\`{e}le avec poids $\upchi= \upchi_f$.   Les fonctions $\upzeta_{\mathfrak{a}}(s)$, $L(\mathfrak{m},\upchi, s)$ convergent uniform\'{e}ment sur tout compact du demi-plan ouvert ${\rm Re}(s)>1$.  
\end{prop}

\begin{note} L'argument qui suit marchera pour n'importe quel ensemble de Delaunay $X\subset (0,\infty)$: c'est-a-dire la fonction zêta $\upzeta_{X}(s):=\sum_{x\in X}x^{-s}$ aussi 
converge uniform\'{e}ment sur tout compact du demi-plan ouvert ${\rm Re}(s)>1$.
\end{note}

\begin{proof}[D\'{e}monstration] Soient $0<\upalpha_{1}<\upalpha_{2}<\cdots$  les \'{e}l\'{e}ments 
positifs de $\mathfrak{a}$.  Comme celui-ci est uniform\'{e}ment discret, il existe
$r>0$ tel que $B_{r}(\upalpha_{n})\cap\mathfrak{a}=\{\upalpha_{n}\}$ pour tout $n\in\N$. 
Remplaçant  $\mathfrak{a}$ par
\[ \widetilde{\mathfrak{a}} = r^{-1}\left( \mathfrak{a}-\upalpha_{1}\right)+1, \]
on obtient un autre quasicristal dont les \'{e}l\'{e}ments positifs sont les 
$$\widetilde{\upalpha}_{n}:=r^{-1} (\upalpha_{n}-\upalpha_{1})+1$$
 o\`{u}  $\tilde{\upalpha}_{1}=1$ et $\tilde{\upalpha}_{n+1}-\tilde{\upalpha}_{n}>1$
pour tout $n$.  Alors il existe un constante $C>0$ tel que, pour tout $s$ tel que ${\rm Re}(s)\geq \upsigma >1$,
\[  |\upzeta_{\mathfrak{a}}( s)| \leq C \sum \frac{1}{\widetilde{\upalpha}_{n}^{{\rm Re}(s)}}\leq C \sum \frac{1}{n^{{\rm Re}(s)}} = C\upzeta ({\rm Re}(s)) .\]
On trouve bien que  $\upzeta_{\mathfrak{a}}( s)$ converge uniform\'{e}ment sur tout demi-plan 
de la forme ${\rm Re}(s)\geq \upsigma>1$, d'où le r\'{e}sultat.  Le cas de $L(\mathfrak{m},\upchi, s)$ est complètement analogue où l'on utilise $C\cdot\max_{\upalpha\in\mathfrak{m}} |f(\upalpha')|$ au lieu de $C$.
 \end{proof}

\begin{note} Il faut noter que $\upzeta_{\mathfrak{a}}$ n'est pas une {\em fonction z\^{e}ta au sens de Beurling}.
Voir \cite{DiZh} o\`{u} on introduit formellement les multiplicit\'{e}s qui correspondent aux factorisations distinctes 
d'\'{e}l\'{e}ments, pour avoir un produit eul\'{e}rien.  Du fait de l'absence de factorisation unique pour $\mathfrak{m}$, 
les fonctions z\^{e}ta et L consider\'{e}es ici ne s'écrivent pas en g\'{e}n\'{e}ral comme produits eul\'{e}riens.
\end{note}

\begin{prop}\label{contmer} $\upzeta_{\mathfrak{a}}(s)$ possède une extension méromorphe (unique)
sur tout $\C$, avec un seul p\^{o}le simple en $s=1$.
\end{prop}

\begin{proof}[D\'{e}monstration] On utilise la m\'{e}thode usuelle de l'int\'{e}grale de Riemann-Stieltjes (voir par exemple \cite{MV}, Theorem 1.12).  
Soit $\upalpha_{1}<\upalpha_{2}<\cdots$ comme dans la démonstration de Proposition \ref{analyticcont}.  
En écrivant $\upalpha_{0}:=0$, $\upalpha_{-m}:=-\upalpha_{m}$, $m\in\N$, on définit la $\mathfrak{a}$-{\em  partie entière} comme 
$[x]_{\mathfrak{a}}:=\upalpha_{m}$ si $\upalpha_{m}\leq x\leq \upalpha_{m+1}$ et la $\mathfrak{a}$ -{\em partie fractionnaire} comme
$\{ x\}_{\mathfrak{a}}:=x-[x]_{\mathfrak{a}}\in [0,R] $ o\`{u} $\mathfrak{a}$ est relativement dense par rapport \`{a} $R>0$.  Alors
\begin{align} \upzeta_{\mathfrak{a}}(s) = \sum_{0<\upalpha \leq x} \upalpha^{-s} + \sum_{\upalpha > x} \upalpha^{-s} & = \sum_{0<\upalpha \leq x} \upalpha^{-s} + \int_{x}^{\infty} u^{-s}d[u]_{\mathfrak{a}} \nonumber \\
&=  \sum_{0<\upalpha \leq x} \upalpha^{-s} + \int_{x}^{\infty} u^{-s}du - \int_{x}^{\infty} u^{-s}d\{ u\}_{\mathfrak{a}} \nonumber  \\ 
&=  \sum_{0<\upalpha \leq x} \upalpha^{-s} +\frac{x^{1-s}}{s-1} + \{ x\}_{\mathfrak{a}}x^{-s} + \int_{x}^{\infty} \{ u\}_{\mathfrak{a}} du^{-s} \label{intpartes} \\
&=  \sum_{0<\upalpha \leq x} \upalpha^{-s} +\frac{x^{1-s}}{s-1} + \{ x\}_{\mathfrak{a}}x^{-s} -s \int_{x}^{\infty} \{ u\}_{\mathfrak{a}} u^{-s-1} du, \label{expressionintegrale}
\end{align}
o\`{u} l'\'{e}galit\'{e} (\ref{intpartes}) provient d'une intégration par parties dans l'intégrale de Riemann-Stieltjes (voir Theorem A.3 de \cite{MV}).
La derni\`{e}re expression (\ref{expressionintegrale}) est méromorphe pour ${\rm Re}(s)>0$, avec un seul p\^{o}le simple en $s=1$. (Ici, on note que l'intégrale
en (\ref{expressionintegrale}) converge uniformément dans les compactes du demi-plan ${\rm Re}(s)>0$ et sa intégrande est holomorphe, qui impliquent que l'integral est aussi holomorphe dans le demi-plan.) 
 En intégrant par parties de fa\c{c}on r\'{e}p\'{e}t\'{e}e, on obtient le prolongement souhait\'{e}.
\end{proof}

\begin{note} Grâce \`{a} la Proposition \ref{contmer} on peut d\'{e}finir les $\mathfrak{a}$-{\em nombres de Bernoulli} par 
 \[ B_{\mathfrak{a},n}=\upzeta_{\mathfrak{a}}(-n),\quad n \ \text{ impair}.\]
 \end{note}

On suppose maintenant que $\mathfrak{a}=\mathfrak{m}=\mathcal{M}(\mathfrak{A}, D)\subset\R$ est un id\'{e}al quasicristallin mod\`{e}le ; le reste de cette section s'est consacr\'{e} \`{a} la d\'{e}monstration d'une  
\'{e}quation fonctionnelle pour $\upzeta_{\mathfrak{m}}$.  Pour commencer, il sera nécessaire d'introduire
un peu de terminologie supplémentaire.

D'après Meyer (\cite{Meyer}), l'{\em id\'{e}al quasicristallin mod\`{e}le dual}   
est l'id\'{e}al mod\`{e}le fractionnaire 
\[  \mathfrak{m}^{\vee} =\mathcal{M}[  \mathfrak{A}^{\vee} ,D^{\vee}]\]
associ\'{e} au réseau dual

\[ \mathfrak{A}^{\vee} = \{\upalpha\in K|\; {\rm Tr}(\upalpha\mathfrak{A})\subset \Z\}\]
et à la fen\^{e}tre 
\[ D^{\vee} := (\uppi/3)\cdot \{{\bf x}|\; |{\bf x}\cdot{\bf y}|\leq 1,\; \forall {\bf y}\in D \}   ,\]
qui est un multiple d'ensemble convexe dual de $D$.
Pour tout $n\in\N$, on aura besoin de consid\'{e}rer aussi les id\'{e}aux mod\`{e}les
\[  \mathfrak{m}^{\vee}_{n} := \mathcal{M}(\mathfrak{A}^{\vee},n\cdot D^{\vee})   \]
de sorte que
 \[ \mathfrak{m}^{\vee}=\mathfrak{m}^{\vee}_{1}\subset\mathfrak{m}^{\vee}_{2}\subset\cdots\quad\text{et}\quad \bigcup \mathfrak{m}^{\vee}_{n}= \mathfrak{A}^{\vee}.\]
 
Soit maintenant $\updelta (x)$ la distribution de Dirac ($\delta(f)=f(0)$ pour une fonction $f$ lisse
au voisinage de l'origine) et consid\'{e}rons le {\em peigne de Dirac}
\begin{align}\label{PdeD} \upmu_{\mathfrak{m}}(x) = \sum_{\upalpha\in\mathfrak{m}} \updelta(x-\upalpha) . \end{align}
De plus, soit
\[  \upmu_{\mathfrak{m},\upchi}(x) =\sum_{\upalpha\in\mathfrak{m}} \upchi (\upalpha )\updelta (x-\upalpha)\]
le peigne pondéré associ\'{e} \`{a} $\upchi$.
 Le résultat suivant ({\em formule de Poisson-Meyer}) est formul\'{e} dans \cite{Meyer} avec r\'{e}f\'{e}rence \`{a} 
 \cite{Meyer1}, \cite{Meyer2} pour les preuves. Pour plus de clart\'{e} nous en
donnons ci-dessous une d\'{e}monstration dans notre contexte.

\begin{theo}[Meyer \cite{Meyer2}, \S V.7.3]\label{MeyerPoisson} La transformation de Fourier de $\upmu= \upmu_{\mathfrak{m},\upchi}$ au sens des distributions est donnée par le {\rm peigne de Dirac généralisé}
\begin{align}    \upmu^{\vee}(x) & =  
 \sum_{\upxi\in\mathfrak{A}^{\vee}} \upchi^{\vee}(\upxi ) \updelta (x-\upxi) 
 := \sum_{n=1}^{\infty} \upnu_{n}(x) \label{sommepeines} 
 \end{align}
 o\`{u} 
 \begin{enumerate}
\item[a.] la fonction $\upchi^{\vee}$ est d\'{e}finie par
\[ \upchi^{\vee}(\upxi ) := \frac{2\uppi }{\sqrt{|{\rm disc}(\mathfrak{A})|}}\cdot f^{\vee}(\upxi'),\] avec $f^{\vee}$ le dual de Fourier de $f$ et ${\rm disc}(\mathfrak{A})$ le discriminant de $\mathfrak{A}$ ;
\item[b.] pour chaque $n\in\N$, $\upnu_{n}$ est un peigne pondéré 
 \[ \upnu_{n}(x) = \sum_{\upxi\in \Updelta \mathfrak{m}^{\vee}_{n}} \upchi^{\vee}(\upxi) \updelta (x-\upxi ) \]
 avec support $\Updelta\mathfrak{m}^{\vee}_{1}:=\mathfrak{m}^{\vee}$ quand $n=1$ et
 \[\Updelta\mathfrak{m}^{\vee}_{n}:= \mathfrak{m}_{n}^{\vee}-\mathfrak{m}_{n-1}^{\vee}  = \mathcal{M}\left(\mathfrak{A}^{\vee}, nD^{\vee}-(n-1)D^{\vee}\right)\]
 pour $n\geq 2$ ;
 \item[c.] la convergence de la deuxi\`{e}me série dans {\rm (\ref{sommepeines})}
est tr\`{e}s rapide :
pour tous $N\in\N$ et $R\in\R_{+}$
\begin{align}\label{densityestimate}  \int_{a-R}^{a+R}d|\upnu_{n}| (x)= O(n^{-N})\quad \text{quand $n\rightarrow\infty$}, \end{align}
uniformement en $a$.
\end{enumerate}
\end{theo}

\begin{proof}[D\'{e}monstration] Le poids $\upchi$ se prolonge \`{a} tout $\mathcal{O}_{K}$ par $\upchi (\upalpha )= f(\upalpha')$.  
Alors, pour toute fonction test $g\in \mathcal{S}(K_{\upsigma})$, on peut \'{e}crire
\[ \int_{\R}g(x) \, d\upmu (x) = \sum_{\upalpha\in\mathfrak{m}} f(\upalpha')g(\upalpha ). \]
La fonction en $K_{\infty}$
\[ G(x,y) = f(y)g(x)\]
d\'{e}croît rapidement \`{a} l'infini et sa transform\'{e}e de Fourier 
$G^{\vee}(\upxi,\upeta ) =f^{\vee}(\upxi )g^{\vee}(\upeta )$ est également à d\'{e}croissance rapide \`{a} l'infini.
Appliquant \`{a} $G$ la formule de Poisson classique par rapport au r\'{e}seau $\mathfrak{A}\subset K_{\infty}$
on obtient :
\[  \sum_{\upalpha\in\mathfrak{m}} f(\upalpha')g(\upalpha )= \sum_{\upalpha\in\mathfrak{A}} f(\upalpha ')g(\upalpha )=\frac{2\uppi}{\sqrt{{\rm disc}(\mathfrak{A})}}  
\sum_{\upxi\in\mathfrak{A}^{\vee}} f^{\vee}(\upxi')g^{\vee}(\upxi ),\]
qui fournit la premi\`{e}re egalit\'{e} dans (\ref{sommepeines}).
La fonction $f^{\vee}$ étant aussi de la classe de Schwartz, l'assertion c) est imm\'{e}diate, sauf l'ind\'{e}pendence 
par rapport à $a$.  
Mais  les ensembles mod\`{e}les $\Updelta\mathfrak{m}_{n}^{\vee}$,  
utilis\'{e}s dans la définition des $\upnu_{n} $, font intervenir des fen\^{e}tres du volume 
\[ {\rm Vol}\left(nD^{\vee}-(n-1)D^{\vee}\right) = O(n^{r+s-2}). \]
On peut faire une partition de $nD^{\vee}-(n-1)D^{\vee}$ en $n^{r+s-2}$ parties de volume uniformément borné par un constant independent de $n$.
Puis on obtient pour chaque $n$ les partitions
\[  \Updelta\mathfrak{m}_{n}=\bigsqcup_{i=1}^{n^{r+s-2}} \Updelta\mathfrak{m}_{n,i},\quad |\upnu_{n}| = \sum_{i=1}^{n^{r+s-2}} |\upnu_{n,i}|\]
Les quasicristaux $ \Updelta\mathfrak{m}_{n,i}$  (voir \cite{Meyer}, \cite{Hof})  sont uniform\'{e}ment discrets avec une constante $r>0$
qui ne d\'{e}pend pas de $n,i$.   Mais les poids ont décroissance rapide en $\infty$ qui implique le décroissance en (\ref{densityestimate}) et qu'on peut choisir la constante implicite dans (\ref{densityestimate}) 
ind\'{e}pendamment de $a$. 
\end{proof}



Nous nous proposons d'appliquer le Th\'{e}or\`{e}me \ref{MeyerPoisson} \`{a} l'\'{e}tude des fonctions L
associ\'{e}es \`{a} $\Updelta\mathfrak{m}^{\vee}_{n}$ et $\upchi^{\vee}$:  
\[  L^{\vee}_{n} (\mathfrak{m},\upchi, s) :=L (\Updelta\mathfrak{m}_{n}^{\vee},\upchi^{\vee}, s)= \sum_{0<\upbeta\in\Updelta\mathfrak{m}_{n}^{\vee}} \frac{ \upchi^{\vee} (\upbeta)}{\upbeta^{s}}.\]

\begin{prop}\label{Lconverge} La suite de fonctions $L^{\vee}_{n} (\mathfrak{m},\upchi, s)$ convergent uniform\'{e}ment 
vers z\'{e}ro sur tout compact du demi-plan ${\rm Re}(s)>1$.  De plus la somme
\begin{align}\label{Lsomme}  \sum_{n=1}^{\infty} L^{\vee}_{n} (\mathfrak{m},\upchi, s)\end{align}
converge uniform\'{e}ment  sur un domaine du même type et d\'{e}finit donc
une fonction holomorphe pour ${\rm Re}(s)>1$.
\end{prop}

\begin{proof}[D\'{e}monstration]  Comme d\'{e}j\`{a}  observ\'{e} dans la preuve du Th\'{e}or\`{e}me \ref{MeyerPoisson}, 
les ensembles mod\`{e}les $\Updelta\mathfrak{m}_{n}^{\vee}$
sont uniform\'{e}ment discrets avec une constant $r>0$ indépendante de~$n$.   
Alors, pour $[c, C]\subset \R$, $c>1$, 
\[  | L^{\vee}_{n} (\mathfrak{m},\upchi, s)|\leq A\cdot \max_{\upbeta\in \Updelta \mathfrak{m}^{\vee}_{n}}|f^{\vee}(\upbeta' )| ,\quad {\rm Re}(s)\in [c,C],\]
o\`{u} $A$ est de nouveau une constante indépendante de $n$.  Ainsi puisque la transformée de Fourier $f^{\vee}$ 
de $f$ est dans l'espace de Schwartz, $L^{\vee}_{n} (\mathfrak{m},\upchi, s)$ converge uniform\'{e}ment vers 0 sur les 
compacts pour ${\rm Re}(s)>1$.  \`{A} cause de la d\'{e}croissance rapide de $f^{\vee}$ \`{a} l'infini, la somme 
(\ref{Lsomme}) converge de même.
\end{proof}

On note  \[ L^{\vee}(\mathfrak{m},\upchi,s) :=\sum_{n=1}^{\infty} L^{\vee}_{n} (\mathfrak{m},\upchi, s),\]
qu'on appelle {\em fonction L g\'{e}n\'{e}ralis\'{e}}.
Les fonctions normalis\'{e}s 
\[  \Uplambda  (\mathfrak{m},\upchi, s),\quad 
\Uplambda_{n}^{\vee}  (\mathfrak{m},\upchi, s), \quad \Uplambda^{\vee}  (\mathfrak{m},\upchi, s) \] 
se d\'{e}finissent  en multipliant chacune de $L (\mathfrak{m},\upchi, s)$,  $L^{\vee}_{n}(\mathfrak{m},\upchi,s) $
et $ L^{\vee}(\mathfrak{m},\upchi,s) $ par le facteur classique $ \uppi^{-s/2}\Upgamma (s/2)$.

 
 La {\em fonction th\^{e}ta quasicristalline} associ\'{e}e \`{a} $\mathfrak{m}$, $\upchi$ s'\'ecrit
\[  \uptheta_{\mathfrak{m},\upchi}(t) := \frac{1}{2}\upchi (0) + \sum_{0<\upalpha \in\mathfrak{m}}  \upchi (\upalpha ) e^{-\uppi \upalpha^{2}t} ,\]
et les {\em fonctions th\^{e}ta duales} de niveaux $ n=1,2,\dots $ sont définies par  
\[ \uptheta^{\vee}_{\mathfrak{m},\upchi,1}(t):=\frac{1}{2}\upchi^{\vee} (0)  + \sum_{0<\upbeta \in\mathfrak{m}^{\vee}}  \upchi^{\vee} (\upbeta ) e^{-\uppi \upbeta^{2}t} \]
pour $n=1$ et
\[\uptheta^{\vee}_{\mathfrak{m},\upchi,n}(t) :=  \sum_{0<\upbeta \in\Updelta\mathfrak{m}_{n}^{\vee}}  \upchi^{\vee} (\upbeta ) e^{-\uppi \upbeta^{2}t}
\]
pour $n\geq 2$.
Les quasicristaux $\Updelta\mathfrak{m}_{n}^{\vee}$, étant uniform\'{e}ment discrets avec une constante $r$ 
que l'on peut choisir ind\'{e}pendamment de $n$, ils d\'{e}finissent une suite 
$\{ \uptheta^{\vee}_{\mathfrak{m},\upchi,n}(t) \}$
qui converge uniform\'{e}ment sur tout compact de $(0,\infty)$.  Comme la restriction de $\upchi^{\vee}$ \`{a} $\Updelta\mathfrak{m}_{n}^{\vee}$
tend rapidement \`{a} z\'{e}ro, la limite de cette suite est nulle. 
De la m\^{e}me fa\c{c}on qu'à la Proposition \ref{Lconverge}, la somme
\[ \uptheta^{\vee}_{\mathfrak{m},\upchi}(t) :=\sum_{n=1}^{\infty} \uptheta^{\vee} _{\mathfrak{m},\upchi,n}(t)  \]
converge et d\'{e}finit une fonction lisse. 
\begin{lemm}\label{decroissancethetadual} La fonction 
\[ \uptheta^{\vee}_{\mathfrak{m},\upchi}(t) -\frac{\upchi^{\vee}(0)}{2} \]
est à d\'{e}croissance rapide \`{a} l'infini.
\end{lemm}

\begin{proof}[D\'{e}monstration] C'est une conséquence immédiate des faits suivants : 
1) les $\Updelta\mathfrak{m}_{n}$ sont uniform\'{e}ment discrets de constante $r>0$ ind\'{e}pendante de $n$ et 
2) la fonction $\upchi^{\vee}$ est elle-même à d\'{e}croissance rapide \`{a} l'infini.
\end{proof}

Appliquons maintenant la formule sommatoire distributionnelle de Poisson-Meyer à la fonction 
\[ f_{t}(x):=e^{-\uppi tx^{2}},\quad   f^{\vee}_{t}=t^{-1/2}f_{t^{-1}}.\]
En utilisant l'identit\'{e} $T^{\vee}(f)=T(f^{\vee})$, o\`{u} $T$ est une distribution temp\'{e}r\'{e}e, et le fait 
que $f_{t}$ est paire, on obtient l'\'{e}quation fonctionnelle :
\begin{align}\label{fonceqn}   \uptheta_{\mathfrak{m},\upchi}(t)  = \frac{2\uppi }{\sqrt{|{\rm disc}(\mathfrak{A})|}}t^{-1/2}\uptheta^{\vee}_{\mathfrak{m},\upchi}(t^{-1}).\end{align}

\begin{theo}\label{FE1}  La fonction $ \Uplambda  (\mathfrak{m},\upchi, s)$ possède un prolongement méromorphe 
  \`{a} $\C$ tout entier, avec deux p\^{o}les simples en $s=0$ et $s=1$, 
et des r\'{e}sidus $\upchi(0)$ et $2\uppi \upchi^{\vee}(0)/\sqrt{|{\rm disc}(\mathfrak{A}|}$.  De plus, on a l'\'{e}quation fonctionnelle :
 \[  \Uplambda  (\mathfrak{m},\upchi, s ) =  \frac{2\uppi }{\sqrt{|{\rm disc}(\mathfrak{A})|}} \Uplambda^{\vee}  (\mathfrak{m},\upchi, 1-s). \]
\end{theo}

\begin{proof}[Preuve]

En suivant la preuve classique, l'\'{e}quation fonctionnelle (\ref{fonceqn}) implique que la transform\'ee de Mellin
 \[  {\sf M}\left(\uptheta_{\mathfrak{m},\upchi} -\frac{1}{2}\upchi (0)\right)(s) := \int_{0}^{\infty} \left( \uptheta_{\mathfrak{m},\upchi}(t)-\frac{1}{2}\upchi (0)\right) t^{s/2} \frac{dt}{t} \]
 converge pour tout $s$.  Si ${\rm Re}(s)>1$ on a 
 \[   \int_{0}^{\infty} e^{-\uppi t \upalpha^{2}}t^{s/2} \frac{dt}{t} = \uppi^{-s/2}\Upgamma (s/2) \upalpha ^{-s},   \]
 ce qui donne 
 \[{\sf M}\left(\uptheta_{\mathfrak{m},\upchi}-\frac{1}{2}\upchi (0)\right)(s)  = \Uplambda  (\mathfrak{m},\upchi, s). \]
  Puis 
 \begin{align}\label{intrep}  \Uplambda  (\mathfrak{m},\upchi, s) =&  \int_{0}^{1} \left[  \frac{2\uppi}{{\rm disc}(\mathfrak{A})}\uptheta^{\vee}_{\mathfrak{m},\upchi}(t^{-1}) -\frac{\upchi (0)}{2}  \right]t^{s/2}\frac{dt}{t}
 +\int_{1}^{\infty}\left[   \uptheta_{\mathfrak{m},\upchi}(t)- \frac{\upchi (0)}{2}  \right]t^{s/2}\frac{dt}{t} \nonumber \\
 =&   \frac{2\uppi}{{\rm disc}(\mathfrak{A})} \int_{1}^{\infty}  \left[   \uptheta^{\vee}_{\mathfrak{m},\upchi}(t)- \frac{\upchi^{\vee} (0)}{2}  \right] t^{(1-s)/2}\frac{dt}{t}  + \int_{1}^{\infty} \left[   \uptheta_{\mathfrak{m},\upchi}(t)- \frac{\upchi (0)}{2}  \right] t^{s/2}\frac{dt}{t}  \\
   & +   \frac{\upchi(0)}{s} + \frac{2\uppi}{\sqrt{{\rm disc}(\mathfrak{A}) }}\cdot\frac{\upchi^{\vee}(0)}{s-1}\nonumber .
 \end{align}
Il reste à observer que les int\'{e}grales apparaissant dans (\ref{intrep}) sont holomorphes en $s$ sur $\C$ 
tout entier du fait de la d\'{e}croissance rapide \`{a} l'infini des intégrands (voir
 le Lemme \ref{decroissancethetadual}).  L'\'{e}quation fonctionnelle annoncée est alors conséquence 
 de l'\'{e}quation pour les fonctions th\^{e}ta correspondantes.
 \end{proof}
 
Revenons \`{a} la fonction z\^{e}ta, $\upzeta_{\mathfrak{m}}(s)$.  
On aurait envie de prendre $f=f_{D}$, fonction indicatrice de la fen\^{e}tre $D$,
dans la formule 
sommatoire de Poisson-Meyer.  Toutefois celle-ci n'est valable que pour $f$ lisse et $f_D$ ne l'est pas.
N\'{e}anmoins, on va voir qu'en prenant une suite de fonctions lisses à supports compacts qui converge vers $f_{D}$, 
la limite correspondante du c\^{o}t\'{e} droit de (\ref{intrep}) est uniforme sur les compacts de $\C \setminus\{ 0,1\}$.

Soient donc $\upvarepsilon\rightarrow 0$ et $f_{\upvarepsilon}\rightarrow f_{D}$ une suite de fonctions lisses qui convergent uniformément à $f_{D}$ telles que
\begin{enumerate}
\item ${\rm supp}(f_{\upvarepsilon})= D_{\upvarepsilon}$ est relativement compacte pour tout $\upvarepsilon$ et $D_{\upvarepsilon}\rightarrow D$ uniformément en la topologie de Hausdorff. 
\item $f_{\upvarepsilon}|_{D}\equiv 1$ et $|f_{\upvarepsilon}(x)|\leq 1$ pour tout $x$ ;
\item $f_{\upvarepsilon}\geq 0$ et $f_{\upvarepsilon}(-x)=f_{\upvarepsilon}(x)$ pour tout $x\in K_{\{ \upsigma\}'}$.
\end{enumerate}
Soit $\mathfrak{m}_{\upvarepsilon}=\mathcal{M}(\mathfrak{A}, D_{\upvarepsilon})={\rm supp}(\upchi_{\upvarepsilon})$. 
On note d'abord que 
$ \Uplambda  (\mathfrak{m}_{\upvarepsilon},\upchi_{\upvarepsilon}, s)$ converge
pour ${\rm Re}(s)>1$ vers
\[ \Uplambda (\mathfrak{m},s):=\uppi^{-s/2}\Upgamma (s/2)\upzeta_{\mathfrak{m}}(s) \]
(qui a une extension meromorphique à $\C\setminus \{ 0,1\}$ à cause de la Proposition \ref{contmer}).
Cette convergence a lieu parce que :
\begin{enumerate}
\item[1.]  Tout $\upalpha\in \mathfrak{m}_{\upvarepsilon}\setminus \mathfrak{m}$ satisfait $x_{i}<_{+_{i}} |\upsigma_{i} (\upalpha)|< x_{i}+\upvarepsilon$.  Alors (\`{a} l'exception de $\upalpha$ ``limite''  e.g.\ $ |\upsigma_{i} (\upalpha)|= x_{i}$ pour tout $i$), 
la derni\`{e}re inégalité implique que pour tout  $M>0$ il existe $N$
tel que $|\upalpha|>M$ quand $\upvarepsilon^{-1}\geq N$ ;
\item[2.] $|\upchi_{\upvarepsilon} (\upalpha )|\leq 1$.
\end{enumerate}
Pour les m\^{e}mes raisons on a convergence uniforme sur les compacts des fonctions 
$  \uptheta_{\mathfrak{m_{\upvarepsilon},\upchi_{\upvarepsilon}}}(t)$ vers $  \uptheta_{\mathfrak{m}}(t)$,
donc convergence des integrales
\[ \int_{1}^{\infty} \left[   \uptheta_{\mathfrak{m_{\upvarepsilon},\upchi_{\upvarepsilon}}}(t)- \frac{\upchi (0)}{2}  \right] t^{s/2}\frac{dt}{t}\longrightarrow  \int_{1}^{\infty} \left[   \uptheta_{\mathfrak{m}}(t)- \frac{\upchi (0)}{2}  \right] t^{s/2}\frac{dt}{t}.\]
Alors, d'après (\ref{intrep}), les int\'{e}grales 
\begin{align}\label{intlimite}  \frac{2\uppi}{{\rm disc}(\mathfrak{A})} \int_{1}^{\infty}  \left[   \uptheta^{\vee}_{\mathfrak{m}_{\upvarepsilon},\upchi_{\upvarepsilon}}(t)- \frac{\upchi_{\upvarepsilon}^{\vee} (0)}{2}  \right] t^{(1-s)/2}\frac{dt}{t} \end{align}
convergent aussi uniform\'{e}ment sur les compacts, toujours pour ${\rm Re}(s)>1$.  
En utilisant la sym\'{e}trie $s\mapsto 1-s$, 
on déduit que la famille $\{  \Uplambda^{\vee}  (\mathfrak{m}_{\upvarepsilon},\upchi_{\upvarepsilon}, s)\}$ converge uniformément sur les compacts de $\C\setminus \{ 0,1\}$.  Notons la limite
\[ \Uplambda^{\vee} (\mathfrak{m},s) := \lim_{\upvarepsilon \rightarrow 0} \Uplambda^{\vee}  (\mathfrak{m}_{\upvarepsilon},\upchi_{\upvarepsilon}, s).
\]
On obtient donc l'\'{e}quation fonctionnelle :
\begin{theo}  $\Uplambda (\mathfrak{m},s) =\frac{2\uppi }{\sqrt{|{\rm disc}(\mathfrak{A})|}}\Uplambda^{\vee} (\mathfrak{m},1-s)$ pour tout $s\in \C$.
\end{theo}

\begin{proof}[Démonstration]  Par le Théorème \ref{FE1}, on a pour tout $s\in \C\setminus \{ 0,1\}$
\[\Uplambda (\mathfrak{m},s)= \lim_{\upvarepsilon \rightarrow 0}  \Uplambda  (\mathfrak{m}_{\upvarepsilon},\upchi_{\upvarepsilon}, s)=
\lim_{\upvarepsilon \rightarrow 0}  \frac{2\uppi }{\sqrt{|{\rm disc}(\mathfrak{A})|}} \Uplambda^{\vee}  (\mathfrak{m}_{\upvarepsilon},\upchi_{\upvarepsilon}, 1-s)=\frac{2\uppi }{\sqrt{|{\rm disc}(\mathfrak{A})|}}\Uplambda^{\vee} (\mathfrak{m},1-s).
\]
 \end{proof}





 \section{L'invariant modulaire quantique et un résultat de R. Pink}\label{modinvsec}
 
Nous d\'{e}veloppons ici la relation, d\'{e}couverte 
par Richard Pink (\cite{Pink}), avec les quasicristaux et l'invariant modulaire quantique; voir le Th\'{e}or\`{e}me \ref{pink} ci-dessous.  Cette section, aussi brève soit-elle, 
revêt néanmoins une importance centrale, l'observation de Pink ayant en grande partie motivé le présent travail.
 
Soient $A=A_{\upsigma}\subset\R$ un anneau quasicristallin de rang 1 et $\mathfrak{a}\subset\R$ un id\'{e}al quasicristallin fractionnaire de $A$. En utilisant la fonction z\^{e}ta $\upzeta_{\mathfrak{a}}$ (voir \S \ref{zeta}) on peut d\'{e}finir 
l'{\em invariant modulaire quasicristallin}, à savoir
\[ j(\mathfrak{a}):= \frac{12^{3}}{1-\frac{49}{40} J(\mathfrak{a}) },\quad J(\mathfrak{a}):=\frac{\upzeta_{\mathfrak{a}}(6)^{2}}{\upzeta_{\mathfrak{a}}(4)^{3}}.\] 
Il y a \'{e}videmment invariance par rapport \`{a} la multiplication par les éléments de $K^{\times}$, par suite de la egalité
 $\upzeta_{\upalpha\mathfrak{a}}(s) =\upalpha^{-s}\upzeta_{\mathfrak{a}}(s)$, $\upalpha\in K^{\times}$.
On obtient ainsi
une fonction bien d\'{e}finie
\[   j: {\sf Cl}^{\sf mod}( A)\longrightarrow \R\cup \{\infty\} .\]

\begin{theo}\label{contjclass} L'invariant modulaire quasicristallin
 $  j$ est continu.   
\end{theo}

\begin{proof}[D\'{e}monstration]  Il suffit de montrer que la restriction au mono\"{\i}de $Z\subset  {\sf Cl}^{\sf mod}( A)$ est continue.  Si $\mathfrak{a}_{i}\rightarrow \mathfrak{a}$, alors $\mathfrak{a}_{i}$ converge vers 
$\mathfrak{a}$ dans la topologie de Hausdorff (uniform\'{e}ment sur tout compact) et donc
pour tout $s$ fixé, $\upzeta_{\mathfrak{a}_{i}}(s)\rightarrow \upzeta_{\mathfrak{a}}(s)$.  
La continuit\'{e} de $j$ est une conséquence immédiate. 
\end{proof}



Soit $\uptheta\in \R$.  Pour $\upvarepsilon>0$ on d\'{e}finit
\[  \Uplambda_{\upvarepsilon}(\uptheta ) := \left\{  n\in\Z\, |\; \| n\uptheta \|<\upvarepsilon \right\} \]
o\`{u} $\| x\|$ est la distance de $x$ \`{a} l'entier le plus proche.  On introduit la fonction z\^{e}ta
\[  \upzeta_{\uptheta,\upvarepsilon}(s) = \sum_{0<n\in \Uplambda_{\upvarepsilon}(\uptheta )} n^{-s} \]
qui converge pour $\Re (s)>1$, en tant qu'une sous-somme de celle qui définit la fonction z\^{e}ta classique.
On d\'{e}finit de plus
\[ j_{\upvarepsilon}:\R\longrightarrow \R\]
pour \[   j_{\upvarepsilon}(\uptheta ):= \frac{12}{1-J_{\upvarepsilon}(\uptheta )},\quad J_{\upvarepsilon}(\uptheta ):=
 \frac{49}{40}\frac{ \upzeta_{\uptheta,\upvarepsilon}(6)^{2}}{ \upzeta_{\uptheta,\upvarepsilon}(4)^{3}}.  \]
L'{\em invariant modulaire quantique} (\cite{Ge-C})
est la fonction {\it multivalu\'{e}e}
\[ j^{\rm qt}:{\sf Mod}^{\rm qt}:= \R/{\sf GL}_{2}(\Z ) \multimap \R\cup \{ \infty\},\quad j^{\rm qt}(\uptheta ):= \lim_{\upvarepsilon\rightarrow 0}  j_{\upvarepsilon}(\uptheta ) , \]
o\`{u} la limite est définie comme l'ensemble des points d'adhérence lorsque $\upvarepsilon\rightarrow 0$.

 Soit maintenant $\uptheta$ une unit\'{e} fondamentale quadratique r\'{e}elle de $K=\Q (\uptheta )$ avec $\mathcal{O}_{K}=\Z[\uptheta]$ 
 l'anneau des entiers. Lors d'une conversation priv\'{e}e (\cite{Pink}), Richard Pink a 
 dégagé la formule suivante pour $j^{\rm qt}(\uptheta )$.  Revenant au mono\"{\i}de 
$ {\sf Z}$ d\'{e}fini dans la formule (\ref{defZ}) du \S \ref{monoides}, celui-ci s'écrit dans le cas quadratique sous la forme
\[  {\sf Z}=\bigcup_{x\in [0,1)}\{\mathfrak{a}_{x} ,\mathfrak{a}^{+}_{x}  \}\subset {\sf Cl}^{\sf mod}( A).\]
 
 
   \begin{theo}[R.Pink]\label{pink} Soit $\uptheta$ une unit\'{e} fondamentale quadratique r\'{e}elle.  Alors
 \[  j^{\sf qt}(\uptheta )=\{  j(\mathfrak{a})|\; \mathfrak{a}\in {\sf Z}\}     .\]
 \end{theo}
 
 \begin{proof}[D\'{e}monstration]   On peut supposer que $\uptheta >1$, étant donné que
 $j^{\rm qt}(\uptheta )=j^{\rm qt}(\pm \uptheta^{\pm1})$.  
Notons 
$\Uplambda_{x}:=\Uplambda_{\uptheta^{-x}}(\uptheta)$ et $\Updelta:= \uptheta-\uptheta'$.  Pour $x\geq 0$ et $n\in \Uplambda_{x}$,
 il existe $m\in\Z$ 
 tel que $|n\uptheta +m|<\uptheta^{-x}$.  En particulier, en \'{e}crivant $\upalpha=n\uptheta +m$, on a
 \[  \Uplambda_{x}=\left\{ \left. \frac{\upalpha-\upalpha'}{\Updelta}\right|\; \upalpha\in \mathcal{O}_{K}, \; |\upalpha|< \uptheta^{-x}\right\},\]
 o\`{u} $\upalpha'$ est le conjugu\'{e} de $\upalpha$.  D'où
  \begin{align*}  \frac{\Updelta}{ \uptheta^{m}} \Uplambda_{x+m} = & \left\{  \left.\frac{\upalpha-\upalpha'}{ \uptheta^{m}}\right|\; \upalpha\in \mathcal{O}_{K},  |\upalpha|<  \uptheta^{-x-m}\right\} \\
&  \text{(change de variable $\upalpha \leadsto-\upbeta' \uptheta'{}^{m}$)}\\
   = & \left\{  \left.\frac{-\upbeta' \uptheta'{}^{m}+\upbeta \uptheta^{m}}{ \uptheta^{m}}\right|\; \upbeta\in \mathcal{O}_{K},  |\upbeta'|< \uptheta^{-x}\right\}  \\
  = &  \left\{  \left.\upbeta -\frac{\upbeta'}{(- \uptheta^{2})^{m}}\right|\; \upbeta\in \mathcal{O}_{K},  |\upbeta'|< \uptheta^{-x}\right\} .
 \end{align*}
 En particulier
 \[\lim_{m\rightarrow\infty} \frac{\Updelta}{\uptheta^{m}} \Uplambda_{x+m}= \mathfrak{a}_{x} \]
 et
 \[\lim_{m\rightarrow\infty} j_{\uptheta^{-x-m}}(\uptheta) = j(\mathfrak{a}_{x}).\]
 On en déduit que
 $$\{  j(\mathfrak{a}_{x}) |\; x\in [0,1) \}=\{  j(\mathfrak{a}_{x}) |\; x\geq 0 \} \subset j^{\sf qt}(\uptheta ) $$
 où la première égalité est consequence de l'identité $\mathfrak{a}_{x+1}=\uptheta \mathfrak{a}_{x}$ et
 l'invariance de $j(\mathfrak{a})$ par rapport à $\mathfrak{a}\mapsto \uplambda\mathfrak{a}$, $\uplambda\in\R$.
Donc tout point d'adhérence de $j_{\upvarepsilon}(\uptheta)$ l'est aussi des $j(\mathfrak{a}_{x})$, ce qui
 fournit le r\'{e}sultat.
 \end{proof}
 
 \begin{coro}\label{jqtestCantor} Soit $\uptheta\in \R$ une unit\'{e} fondamentale quadratique.  Alors $j^{\rm qt}(\uptheta)$ est l'image continue d'un ensemble de Cantor.
\end{coro}

\begin{note} Pour d\'{e}montrer que $j^{\rm qt}(\uptheta)$ est de Cantor, il serait suffisant \'{e}tablir que $j$ est {\it injective}, propriet\'{e} qui est conjecturée.
\end{note}

\begin{note}\label{PinkAnalog} Le Th\'{e}or\`{e}me de Pink constitue l'analogue exact du Theorem 4 de \cite{DGIII}.  	
Plus pr\'{e}cis\'{e}ment (pour la notation voir au \S\ref{DrinfeldHayes} ci-dessus) 
soit ${\bf K}={\bf Q}(f)$ o\`{u} $f$ est une unit\'{e} fondamentale.
Si l'on note
\[  {\bf j}^{\rm qt}:{\rm GL}_{2}({\bf Z})\backslash {\bf R}\multimap {\bf R} \cup \{\infty\}\]
l'invariant modulaire quantique en caract\'{e}ristique positive associ\'{e} au corps de fonctions 
${\bf R}={\bf Q}_{\infty}$ (\cite{DGI}, \cite{DGII}), on a
\[ {\bf j}^{\rm qt}(f) = \{ j(\mathfrak{a})|\; \mathfrak{a}\in {\sf Z}\}\] 
o\`{u} maintenant ${\sf Z}=\{ \mathfrak{a}_{0},\dots , \mathfrak{a}_{d-1}\}\cong {\rm Gal}( {\bf H}_{{\bf A}_{\infty_{1}}}/ {\bf H}_{{\bf O}_{\bf K}})$ est le sous groupe de ${\sf Cl}_{ {\bf A}_{\infty_{1}}  }$ d\'{e}crit au 
\S\ref{DrinfeldHayes}.
\end{note}

\section{Le soléno\"{\i}de associ\'{e} \`{a} un quasicristal}\label{solenoid}

D'après la {\it Note} \ref{PinkAnalog} ci-dessus les quasicristaux dans le Th\'{e}or\`{e}me \ref{pink} 
jouent le m\^{e}me r\^{o}le que les id\'{e}aux $\mathfrak{a} \subset A_{\upsigma_{1}}$, et donc les quasicristaux de rang 1 
devraient fournir un analogue des modules de Drinfeld en caract\'{e}ristique nulle.  Dans cette section on d\'{e}veloppe 
l'analogue de la notion {\em analytique} de module Drinfeld, c'est-\`{a}-dire l'analogue du module quotient ${\bf R}/\mathfrak{a}$.

  L'ensemble des quasicristaux de $\R^{n}$ est muni d'une topologie de la mani\`{e}re sui\-vante : \`{a} un quasicristal 
  $\Upomega\subset\R^{n}$ on associe la mesure de Radon $\upmu_{\Upomega}$ qui est le peigne de Dirac 
  de $\Upomega$ (voir la formule (\ref{PdeD})) et on met la topologie faible sur l'espace des mesures $\upmu_{\Upomega}$.  Si on fixe $\Upomega$, 
  on note la cl\^{o}ture 
  \[ \hat{\SI}_{\Upomega}:=    
  \overline{\{ v+\Upomega  \}}_{v\in\R^{n}}
  \]  
  (voir par exemple \cite{BBG}). De manière équivalente $\hat{\SI}_{\Upomega}$ peut \^{e}tre d\'{e}fini en utilisant le pavage de Vorono\"{\i} $P_{\Upomega}$ associ\'{e} à $\Upomega$,  o\`{u} maintenant $\hat{\SI}_{\Upomega}$ est la cl\^{o}ture de
 l'ensemble $\{v+P_{\Upomega}\}_{v\in \R^{n}}$  des translatés de $P_{\Upomega}$.  
Dans \cite{BBG} les auteurs montrent que $\hat{\SI}_{\Upomega}$ est homéo\-morphe \`{a} la limite inverse d'un 
syst\`{e}me de variet\'{e}s ramifi\'{e}es plates (``branched flat manifolds'') de dimension $n$.   
Si $\Upomega$ est un r\'{e}seau,  $\hat{\SI}_{\Upomega}\approx \R^{n}/\Upomega$ 
et on r\'{e}cup\`{e}re la notion classique de tore quotient.  Un point de $\hat{\SI}_{\Upomega}$ est
r\'{e}presenté par un pavage construit avec les mêmes mailles que $\Upomega$ et qui possède des sous-pavages 
arbitrairement grands qui co\"{\i}ncident avec des sous-pavages de $\Upomega$ \`{a} une translation pr\`{e}s.

\begin{prop} Soit $\Upomega\subset \R^{n}$ un quasicristal.  Alors $\hat{\SI}_{\Upomega}$ est un soléno\"{\i}de compact de dimension $n$.  Si $\Upomega=\mathcal{M}$ est un ensemble mod\`{e}le g\'{e}n\'{e}rique, $\hat{\SI}_{\Upomega}$
est minimal.
 \end{prop}

 \begin{proof} \'{E}tant limite inverse de variet\'{e}s ramifi\'{e}es de dimension $n$, $\hat{\SI}_{\Upomega}$
est par définition un soléno\"{\i}de.  Puisque $\Upomega$ est de Delaunay, $\hat{\SI}_{\Upomega}$ 
est compact. Enfin si $\Upomega$ est g\'{e}n\'{e}rique, il est r\'{e}p\'{e}titif, ce qui implique que $\hat{\SI}_{\Upomega}$ 
est minimal (voir par exemple \cite{BG}).
 \end{proof}
 
 Il est instructif de rappeler le cas du soléno\"{\i}de classique 
\[\hat{\SI}^{1} =\lim_{\longleftarrow}\SI^{1},\]
o\`{u} la limite est d\'{e}finie par rapport \`{a} tous les épimorphismes de $\SI^{1}=\R/\Z$.  Le complet\'{e} profini 
$\hat{\Z}$ de $\Z$ s'identifie au noyau de l'épimorphisme $\hat{\SI}^{1}\rightarrow\SI^{1}$,
$\hat{\SI}^{1}/\hat{\Z}\cong \SI^{1}$,
et d\'{e}finit un sous-groupe de $\hat{\SI}^{1}$ qui est une transversale canonique.  De manière équivalente
$\hat{\Z}=\overline{\Z}$ est la cl\^{o}ture topologique de $\Z\subset\R\subset \hat{\SI}^{1}$, 
o\`{u} le plongement $\R\hookrightarrow\hat{\SI}^{1}$ est induit par $\R\rightarrow\SI^{1}$.

Un épimorphisme $\uprho :\SI^{1}\rightarrow \SI^{1}$ 
induit un {\it homéomorphisme} $\hat{\uprho}:\hat{\SI}^{1}\rightarrow \hat{\SI}^{1}$,
de la m\^{e}me mani\`{e}re que $\uprho$ induit un homéomorphisme $\tilde{\uprho}:\R\rightarrow\R$ 
du rev\^{e}tement universel $\R$ de $\SI^{1}$.  Mais si l'on consid\`{e}re plut\^{o}t comme objet 
principal la paire $(\hat{\SI}^{1}, \hat{\Z})$, 
alors $\hat{\uprho}$ est un morphisme de telles paires, $\hat{\uprho}(\hat{\Z})\subset \hat{\Z}$,  et  
\[ \hat{\uprho}^{-1} (\hat{\Z}) = \bigsqcup_{i=1}^{\text{d\'{e}g}(\uprho )} (r_{i} +\hat{\Z}).\]
On peut ainsi considérer $\hat{\uprho}$ comme un {\it rev\^{e}tement de paires} de degré égal à $\text{deg}(\uprho )$.  
Intuitivement, on peut  voir $\hat{\Z}$ comme un élargissement de l'él\'{e}ment
neutre du cercle $\SI^{1}$ et, en conservant la transversale $\hat{\Z}$, 
on r\'{e}cup\`{e}re les propriétés de $\SI^{1}$ en travaillant avec la paire $(\hat{\SI}^{1},\hat{\Z})$.
 
Revenant au solénoïde $\hat{\SI}_{\Upomega}$, le r\^{o}le de $\hat{\Z}$ est jou\'{e} par l'ensemble
 \[ \hat{\Upomega}=  \overline{\{ \upomega+\Upomega  \}}_{\upomega\in\Upomega}\subset\hat{\SI}_{\Upomega},\]
transversale complète\footnote{Une {\em transversale} ({\em complète}) d'une lamination $\mathcal{L}$ est un sous-ensemble $T\subset\mathcal{L}$ tel que $T\cap L$ est discret (non vide) pour la topologie de $L$, pour toute feuille $L\subset\mathcal{L}$.} 
passant par $0$, la {\em compl\'{e}t\'{e}e pro-quasicristalline} de $\Upomega$.  Si $\Upomega$ est un r\'{e}seau 
$\hat{\Upomega}=\{ 0\}$ ; on peut donc voir en g\'{e}n\'{e}ral $\hat{\Upomega}$
 comme un élargissement de l'él\'{e}ment neutre.
 Par contre, si l'on identifie deux points de $\hat{\SI}_{\Upomega}$ qui appartiennent \`{a} la m\^{e}me  $\hat{\Upomega}$ transversale, $x\sim y$ si et seulement si $x,y\in z+\hat{\Upomega}$, qui donne le {\it quotient transversal topologique}  
 d\'{e}fini par $\hat{\Upomega}$. On obtient ainsi en g\'{e}n\'{e}ral 
 une variet\'{e} ramifi\'{e}e plate.  Nous adopterons dans la suite une position interm\'{e}diaire, 
consid\`{e}rant parfois des notions comme ``fonction'', ``torsion'' modulo $\hat{\Upomega}$, 
mais sans identifier ce dernier ensemble \`{a} un point.  Autrement dit on regardera $\hat{\Upomega}$ 
comme la {\it transversale neutre} en d\'{e}veloppant une th\'{e}orie qui n'oublie pas ce fait en passant au quotient.  
 
 

Consid\'{e}rons maintenant l'analogue d'une isogénie entre deux soléno\"{\i}des quasicristallins.  
On dit que deux quasicristaux $\Upomega, \Upomega'\subset\R^{n}$ sont {\em \'{e}quivalents} 
(\cite{Moody}, \cite{Meyer}) s'il existe un couple d'ensembles finis $F,F'\subset\R^{n}$ telle que
\[\Upomega\subset\Upomega'+F'\quad \text{et} \quad \Upomega'\subset\Upomega+F .  \]

\begin{exam}\label{solenoidex} i) Si $\Upomega, \Upomega'$ sont deux r\'{e}seaux isogènes 
alors $\Upomega,\Upomega'$ sont \'{e}quivalents~; cette relation d'\'{e}quivalence est donc
une g\'{e}n\'{e}ralisation de la relation d'isog\'{e}nie.  

ii) Si $\Upomega, \Upomega'\subset\R^{n}$ sont deux ensembles mod\`{e}les d\'{e}finis par rapport au m\^{e}me r\'{e}seau $\Upgamma\subset\R^{n+k}$, alors $\Upomega,\Upomega'$ sont toujours \'{e}quivalents (\cite{Moody}).
\end{exam}

On rappelle qu'un ensemble $X\subset\R^{n}$ est {\em apériodique} si le groupe de ses périodes
\[ {\rm Per}(X) := \{ v\in \R^{n}|\; v+X=X\} \]
est réduit à 0.  Si un quasicristal $\Upomega\subset\R^{n}$ est apériodique, on appelle 
\[ L_{\Upomega}:=\R^{n} +\Upomega\subset\hat{\SI}_{\Upomega}\] 
la {\em feuille canonique} ; elle est clairement homéomorphe \`{a} $\R^{n}$ et dense dans $\hat{\SI}_{\Upomega}$.

\begin{theo}\label{isog}  Soient \[ \Upomega\subset \Upomega'\subset\R^{n}\] deux quasicristaux mod\`{e}les d\'{e}finis 
par le m\^{e}me r\'{e}seau.  Alors,  la correspondence $x+\Upomega\mapsto x+\Upomega'$, $x\in\R^{n}$, d\'{e}finit 
un morphisme de paires {\rm (}{\rm rev\^{e}tement}{\rm )} canonique 
\[ \uprho:(\hat{\SI}_{\Upomega}, \hat{\Upomega})\longrightarrow (\hat{\SI}_{\Upomega'}, \hat{\Upomega}')\]
qui est surjectif.
Plus pr\'{e}cis\'{e}ment, $\uprho: \hat{\SI}_{\Upomega}\rightarrow \hat{\SI}_{\Upomega'}$ est un homéomorphisme 
local surjectif de soléno\"{\i}des telle que $\uprho (\hat{\Upomega})\subset \hat{\Upomega}'$ et 
\begin{align}\label{preimagefinie} \uprho^{-1}(\hat{\Upomega}') \subset \bigcup_{i=1}^{d}(x_{i}+\hat{\Upomega}),\quad x_{i}\in\R^{n},\end{align}
o\`{u} l'union n'est pas  forc\'{e}ment disjointe.
Si $\Upomega, \Upomega'$ sont tous les deux apériodiques, $\uprho: \hat{\SI}_{\Upomega}\rightarrow \hat{\SI}_{\Upomega'}$ 
est un homéomorphisme.  

\end{theo}

\begin{proof}[D\'{e}monstration]  D'après l'{\em Exemple} \ref{solenoidex}, ii), il existe $F$ fini tel que $\Upomega'\subset \Upomega+F$.  Si la suite $x_{i}+\Upomega$ converge dans $\hat{\SI}_{\Upomega}$,  $x_{i}+\Upomega+F$ converge 
aussi dans $\hat{\SI}_{\Upomega+F}$, ce qui implique que $x_{i}+\Upomega'\subset x_{i}+\Upomega+F $ 
converge dans  $\hat{\SI}_{\Upomega'}$.  Donc $ \uprho$ est continue.   
Inversement, supposant que $x_{i}+\Upomega'$ converge dans $\hat{\SI}_{\Upomega'}$,  le m\^{e}me
argument montre que $x_{i}+\Upomega\subset x_{i}+\Upomega'$ converge dans $\hat{\SI}_{\Upomega}$, 
ce qui signifie que $\uprho$ est surjective.  On affirme que $\uprho|_{\hat{\Upomega}}$ est injective.
Supposons au contraire que $\hat{\upgamma}=\lim \upgamma_{i} +\Upomega\not= \hat{\upeta}=\lim \upeta_{i} +\Upomega$ sont des points distincts de $\hat{\Upomega}$ tels que
$\uprho (\hat{\upgamma})=\uprho (\hat{\upeta})$.  On observe alors que, par d\'{e}finition de la topologie, 
la limite $\hat{\upgamma}=\lim \upgamma_{i} +\Upomega$ existe si $\upgamma'_{i}$ est de Cauchy 
et $\hat{\upgamma}\not=\hat{\upeta}$ si et seulement si $(\upgamma_{i}-\upeta_{i})'\not\rightarrow 0$.  
La m\^{e}me chose est vraie en échangeant $\Upomega$ et $\Upomega'$, d'où il suit que 
$\uprho (\hat{\upgamma})=\lim \upgamma_{i}+\Upomega'\not=\lim\upeta_{i}+\Upomega' =\uprho (\hat{\upeta})$, 
contradiction. 
Comme $\Upomega$,
$\Upomega'$ sont relativement denses l'injectivit\'{e} de $\uprho$ en $\hat{\Upomega}$ se prolonge \`{a} 
un voisinage feuillet\'{e} et $\uprho$ est donc bijective et bicontinue dans un voisinage de $0$.
On montre de même l'injectivit\'{e} sur chaque transversale $x+\hat{\Upomega}$, ce qui
prouve que $\uprho$ est un homéomorphisme local.
La préimage $\uprho^{-1}(\hat{\Upomega}')$ est \'{e}gale \`{a} la cl\^{o}ture de
\[  \bigcup_{x\in \Upomega'} (x + \Upomega) \subset \bigcup_{y\in\Upomega, \; f\in F} (y+f+\Upomega) \subset \bigcup_{f\in F} (f+\hat{\Upomega}), \]
qui fournit bien (\ref{preimagefinie}).

Enfin si $\Upomega,\Upomega'$ sont
apériodiques, $\uprho$ induit un homéomorphisme $L_{\Upomega}\approx L_{\Upomega'}$ 
des feuilles canoniques denses ce qui, joint au fait que $\uprho$ est un homéomorphisme local, 
implique qu'il s'agit bien d'un homéomorphisme.
\end{proof}
 
 La propriet\'{e} d'\^{e}tre un presque r\'{e}seau implique que 
  \[ \Upsigma^{n} \Upomega := \Upomega +\cdots +\Upomega\;\;\text{ ($n$ fois)} \]
est aussi un quasicristal (et un id\'{e}al quasicristallin dans le cas o\`{u} $\Upomega$ en est un).  
De plus, $\Upsigma^{n}\Upomega$ est équivalent à $\Upomega$. 
 Le Th\'{e}or\`{e}me \ref{isog} implique qu'il existe une surjection continue
 \[   +:(\hat{\SI}_{\Upomega},\hat{\Upomega})\times(\hat{\SI}_{\Upomega},\hat{\Upomega})\longrightarrow ( \hat{\SI}_{\Upsigma^{2} \Upomega},  \Upsigma^{2}\hat{\Upomega}),\quad (x + \Upomega, y+\Upomega)\longmapsto x+y+\Upsigma^{2}\Upomega, \]
 o\`{u} l'on a identifi\'{e}  $ \Upsigma^{n}\hat{\Upomega}=\widehat{\Upsigma^{n}\Upomega}$.
En g\'{e}n\'{e}ral, pour $\vec{m}\in\N^{k}$ o\`{u} $m_{1}+\cdots +m_{k}=n $, il y a une application
\begin{align}\label{+general} +=+_{\vec{m},n}: (\hat{\SI}_{\Upsigma^{m_{1}}\Upomega},  \Upsigma^{m_{1}}\hat{\Upomega})\times\dots\times (\hat{\SI}_{\Upsigma^{m_{k}}\Upomega},  \Upsigma^{m_{k}}\hat{\Upomega}) \longrightarrow  (\hat{\SI}_{\Upsigma^{n}\Upomega},  \Upsigma^{n}\hat{\Upomega}),\end{align}
induite par 
\[ (x_{1} + \Upomega,\dots ,x_{k}+\Upomega)\longmapsto x_{1} + \cdots + x_{k} + \Upsigma^{n}\Upomega.\]
Si l'on note $+_{3}=+_{\;\vec{1},3}$, $\vec{1}=(1,1,1)$, l'\'{e}galit\'{e} des compositions 
\[    +_{3}(x+\Upomega, y+\Upomega, z+\Upomega)= +(+(x+\Upomega, y+\Upomega), z+\Upomega) =+(x+\Upomega, +(y+\Upomega, z+\Upomega)) , \]
ainsi que des version plus g\'{e}n\'{e}rales nous donnent l'analogue de la loi de l'associativit\'{e}.
On peut donc considérer $\hat{\SI}_{\Upomega}$ comme une g\'{e}n\'{e}ralisation de la notion de  tore, ou mieux 
la collection de soleno\"{\i}des
\[  \{  (\hat{\SI}_{\Upsigma^{n}\Upomega},  \Upsigma^{n}\hat{\Upomega})\}_{n\geq 1}, \]
munie des applications (\ref{+general}).

Pour prouver la continuité de la somme, il faut avoir en main le résultat suivante:

\begin{theo}\label{ConjCont} Soit $\Upomega$ un ensemble modèle avec la fenêtre $D$.  Alors la correspondence 
\begin{align}\label{AppConj} \upalpha\in \Upomega\longmapsto \upalpha'\in D\end{align} induit une application continue et surjective
\[ \hat{\Upomega} \twoheadrightarrow \overline{D} \]
\end{theo}

\begin{proof}[D\'{e}monstration]  La preuve est esencialmente consequence de Proposition 4.3 de \cite{Schlottmann}, qui afirme la continuité de l'application canonique  
\[\upbeta: \left( \hat{\SI}_{\Upomega}, \hat{\Upomega}\right)\longrightarrow\left( \mathcal{F}_{\Upomega},\overline{D}\right) \]
où  $ \mathcal{F}_{\Upomega}$ est une feuilletage de Kronecker qui parametrise d'ensembles modèles ``basés'' sur $\Upomega$ et $\overline{D}$ paraît comme un transversal canonique.  L'application $\upbeta$,
restreint à $\Upomega\subset \hat{\Upomega}$, coincide à la conjugation (\ref{AppConj}).  Un argument détaillé se trouve dans l'Appendice.
\end{proof}

\begin{theo} Soit $\Upomega$ un ensemble modèle. Alors la somme $+_{\vec{m},n}$, définie en {\rm (\ref{+general})}, est continue.
\end{theo}

\begin{proof}[D\'{e}monstration] On considère d'abord la preuve dans le cas particulier ou 
\begin{itemize}
\item[-] $\Upomega$ es una quasicristal modèle de 
dimension 1,
définit par la fenêtre
 $D\subset \mathbb{R}$, 
 et 
 \item[-] $\vec{m} = (1,1)$, $n=2$.
 \end{itemize}
Il suffit démontre que la somme est transversalmente continue: autrement dit, que l'application canonique $\Upomega\times \Upomega\stackrel{+}{\rightarrow} \Upsigma\Upomega$
s'étend à une application continue
\[ \hat{\Upomega}\times \hat{\Upomega}\stackrel{+}{\longrightarrow} \Upsigma\hat{\Upomega} .\]
Soient $\{ \upalpha_{i}+\Upomega \}$, $\{ \upbeta_{i}+\Upomega\}$ des suites dans $\Upomega$ qui convergent à deux points de $\hat{\Upomega}$.  Il faut démontre que 
\begin{align}\label{SuiteSomme}  \{ \upomega_{i}   + \Upsigma\Upomega \} := \{ \upalpha_{i} + \upbeta_{i}  + \Upsigma\Upomega \} \end{align}
est convergente dans $\Upsigma\Upomega$.   
On sait par le Théorème \ref{ConjCont}   que la suite de conjugués $\{ \upomega'_{i}\}$ est convergente à $w\in \overline{\Upsigma D}$, où $\Upsigma D = $ la fenêtre de $\Upsigma\Upomega$, donc, les fenêtres $\upomega'_{i}+\Upsigma D$ associés aux quasicristaux modèles $\upomega_{i}+\Upsigma\Upomega$
convergent.  On peut supposer sans perte de généralité que $\Upsigma D= [-a,a]$.

Supposons que la suite (\ref{SuiteSomme}) ne converge pas dans la topologie quasicristalline;  si l'on note pour $R>0$,
\[  (\upomega_{i} + \Upsigma \Upomega)_{R} :=( \upomega_{i} + \Upsigma \Upomega  )\cap [-R,R] , \]
cela signifie qu'il existe $R>0$ tel que
pour chaque $N\in\N$, ils existent $i_{N}, j_{N}>N$ et $\upomega_{j_{N}} + \upbeta_{j_{N}}\in (\upomega_{j_{N}}+ \Upsigma\Upomega)_{R}$ avec
$\upomega_{j_{N}} + \upbeta_{j_{N}}\not\in (\upomega_{i_{N}}+ \Upsigma\Upomega)_{R}$.   

Puisque $\Upsigma\Upomega$ est aussi un quasicristal, il existe $F\subset\R$ fini tel que $\Upsigma\Upomega +
\Upsigma\Upomega\subset \Upsigma\Upomega +F$.   Donc, en passant à une sous-suite, il existe $f\in F$ avec
\[ \upomega_{j_{N}} + \upbeta_{j_{N}} = \upgamma_{j_{N}} + f \in \Upsigma \Upomega  + \Upsigma \Upomega \subset \Upsigma\Upomega +F.\]
$ \Upsigma \Upomega$ est uniformémente discret, ainsi il n'y a que un nombre fini de posible $\upgamma_{j_{N}}\in \Upsigma\Upomega$ avec $\upgamma_{j_{N}} +f\in [-R,R]$.  Alors, on peut passer à
autre sous-suite  telle que la suite des $\upgamma_{j_{N}}$ est constant.  C'est-à-dire, il existe $\upgamma\in\Upsigma\Upomega$ telle que la suite de contre-exemples à continuité est égale le constante
$\upgamma +f $:
\[ \upomega_{j_{N}} + \upbeta_{j_{N}} = \upgamma +f .\]
Finalente, a autre passage à une sous-suite, on peut supposer que le suite de conjugués $\{ \upomega_{i_{N}}\}$ converge de mode monotone, disons $\upomega_{i_{N}}'\nearrow w$.

Il en résulte la situation suivante: pour chaque $N$, d'une part 
\[  \upomega_{j_{N}}' -a \leq \upgamma' + f'  \leq \upomega_{j_{N}} '+a \]
et de l'autre 
\begin{align}\label{OneOf2Fails} \upomega_{i_{N}} '-a \not\leq  \upgamma' + f'  \quad\text{ou}\quad \upgamma' + f'   \not\leq  \upomega'_{i_{N}} +a. \end{align}
En passant de nouveau à une sous-suite, on peut supposer dans (\ref{OneOf2Fails}) que l'un ou l'autre condition est satisfait pour tout $N$: supposons d'abord qu'il est la
dernière posibilité,  ainsi on a
\begin{align}\label{ContraInequality} \upomega_{i_{N}}' +a<   \upgamma' + f' \leq \upomega_{j_{N}}' +a   .\end{align}  Pour $M > j_{N}$,
nous avons  $i_{M}>j_{N}$, donc $\upomega_{i_{M}}' \geq \upomega_{j_{N}}'$ (par la monotonicité de la suite de conjugués), qui implique que $  \upgamma' + f'\leq    \upomega_{j_{N}}' +a  \leq  \upomega_{i_{M}}' +a $, contradisant
 (\ref{ContraInequality}).  On suppose maintenant la première inégalité dans (\ref{OneOf2Fails}), qui nous donne
\begin{align}\label{ContraInequality2} \upomega_{j_{N}}' -a \leq \upgamma' + f'  <  \upomega_{i_{N}} '-a.\end{align}
 Puis, on prends $M$ telle que $\upomega'_{j_{M}}>\upomega'_{i_{N}}$ et  (\ref{ContraInequality2}) se contradit.
 On conclut que la suite (\ref{SuiteSomme}) converge et la somme s'étend à une application continue de $\hat{\Upomega}\times \hat{\Upomega}$.
La preuve pour $+_{\vec{m},n}$ général est pour l'induction.   Dans le cas de un quasicristal de dimension $>1$ et fenêtre $\subset \R^{k}$, on peut argumenter de manière égale, supposons sans perte de géneralité que
$D$ est un cube  en imposant monotonicité dans chaque coordonné.
\end{proof}

 \'Etant donnée  une suite d'applications de paires 
\[ f_{n}: (\hat{\SI}_{\Upsigma^{n}\Upomega}, \Upsigma^{n}\hat{\Upomega})\longrightarrow   (\hat{\SI}_{\Upsigma^{n}\Upomega}, \Upsigma^{n}\hat{\Upomega})\]
on dira qu'elle définit un {\em endomorphism quasicristallin} si on a des carr\'{e}s commutatifs
\begin{align}\label{endogramgen}
\begin{diagram}
 \hat{\SI}_{\Upsigma^{m_{1}}\Upomega} \times \dots \times  \hat{\SI}_{\Upsigma^{m_{k}}\Upomega} & \rTo^{f_{m_{1}}\times\cdots \times f_{m_{k}}} & \hat{\SI}_{\Upsigma^{m_{1}}\Upomega} \times \dots \times  \hat{\SI}_{\Upsigma^{m_{k}}\Upomega} \\
\dTo^{+_{\vec{m},n}} & & \dTo_{+_{\vec{m},n}} \\
 \hat{\SI}_{\Upsigma^{n}\Upomega}& \rTo_{f_{n}} &  \hat{\SI}_{\Upsigma^{n}\Upomega} ,
\end{diagram}
\end{align}
pour $m_{1}+\cdots +m_{k}=n$.
Un endomorphisme quasicristallin inversible est un {\em isomorphisme quasicristallin}.
On notera 
\[{\rm End}(\hat{\SI}_{\Upomega},\hat{\Upomega}) \]
le mono\"{i}de multiplicatif des endomorphismes, le produit étant d\'{e}fini par la composition.  
Les operations de somme induisent des opérations correspondantes sur les endomorphismes ; 
ainsi on a une opération de somme 
\[  {\rm End}(\hat{\SI}_{\Upomega}, \hat{\Upomega})\times {\rm End}(\hat{\SI}_{\Upomega},\hat{\Upomega})\stackrel{+}{\longrightarrow} {\rm End}(\hat{\SI}_{\Upsigma\Upomega}, \Upsigma\hat{\Upomega}) \]
definie par
\[  f_{n}, g_{n} \longmapsto  f_{n}+g_{n}:\hat{\SI}_{\Upsigma^{n} (\Upsigma\Upomega)}\longrightarrow \hat{\SI}_{\Upsigma^{n} (\Upsigma\Upomega)}.
\]
Autrement dit la collection 
\[ \{  {\rm End}(\hat{\SI}_{\Upsigma^{n}\Upomega}, \Upsigma^{n}\hat{\Upomega}) \}\] 
possède la structure d'un {\em presque anneau} c'est-à-dire une collection de mono\"{\i}des (pour un
certain produit) avec des applications commutatives
\[      {\rm End}( \hat{\SI}_{\Upsigma^{m_{1}}\Upomega}, \Upsigma^{m_{1}}\hat{\Upomega}) \times \dots \times  {\rm End}( \hat{\SI}_{\Upsigma^{m_{k}}\Upomega} ,\Upsigma^{m_{k}}\hat{\Upomega})\stackrel{+}{\longrightarrow }  {\rm End}(\hat{\SI}_{\Upsigma^{n}\Upomega}, \Upsigma^{n}\hat{\Upomega}), \]
où $m_{1}+\cdots +m_{k}=n$ ; ces opérations sont distributives par rapport au produit.
\begin{note}
Par opposition au cas des r\'{e}seaux, il n'est pas n\'{e}cessairement vrai que $N\in\Z$ d\'{e}finit un \'{e}l\'{e}ment de ${\rm End}(\hat{\SI}_{\Upomega}) $. 
On a seulement le ``quasiendomorphisme" :
\[ N: \hat{\SI}_{\Upomega}\longrightarrow \hat{\SI}_{N\Upomega},\quad x+\Upomega\longmapsto Nx +N\Upomega \]
o\`{u} $N\Upomega:=\{ Nx|\; x\in\Upomega\}\subset\Upsigma^{N}\Upomega$.
\end{note}

Consid\'{e}rons le cas de $\mathfrak{a}\subset K\subset\R$ un id\'{e}al quasicristallin mod\`{e}le par rapport \`{a} l'anneau quasicristallin $A=A_{\upsigma}$ de rang 1. La proposition suivante
identifie les sommes quasicristallines d'un tel id\'{e}al quasicristallin mod\`{e}le.

\begin{prop}\label{propdesomme} Soit ${\bf x}\in\R^{r+s-1}$, $u\in \mathcal{O}_{K}^{\times}$, $\mathfrak{a}_{\bf x}=\mathfrak{a}_{\mathfrak{A}, {\bf x} }(u)$. Alors 
\[  \Upsigma^{n}\mathfrak{a}_{\bf x}=\mathfrak{a}_{\bf y},\quad \text{o\`{u}}\quad {\bf y}=-\log_{|\uppi'(u)|}n +{\bf x}.\]  En g\'{e}n\'{e}ral, $\mathfrak{a}_{{\bf x}_{1}}+\mathfrak{a}_{{\bf x}_{2}}=
\mathfrak{a}_{\bf y}$, o\`{u} ${\bf y}=-\log_{\uppi'(u)} (\uppi'(u)^{{\bf x}_{1}}+\uppi'(u)^{{\bf x}_{2}})$.
\end{prop}

\begin{proof}[D\'{e}monstration] On donne la preuve dans le cas de $K/\Q$ quadratique r\'{e}elle ; le cas g\'{e}n\'{e}ral 
est tout à fait analogue.
Par d\'{e}finition, $\mathfrak{a}_{x}+\mathfrak{a}_{x}\subset \mathfrak{a}_{-\log_{\uppi'(u)}2+x}$.  
Soit alors $\upgamma\in  \mathfrak{a}_{x-\log_{\uptheta}2}$ ; 
on peut supposer que $\upgamma'>0$ sans perte de généralité. On a donc $\upgamma'<2\uptheta^{-x}$.  
Soit $\upalpha\in\mathfrak{a}_{x}$ tel que $\upalpha'>0$ et soit $\upbeta=\upgamma-\upalpha$ avec
$0<\upbeta'$.  On peut choisir $\upalpha$ tel que $\upbeta'<\uptheta^{-x}$, parce que l'ensemble 
de $\upalpha'$, $\upalpha\in\mathfrak{a}_{x}$,
est dense dans la fen\^{e}tre $D_{x}$. 
Alors $\upgamma=\upalpha+\upbeta\in\mathfrak{a}+\mathfrak{a}$.  
\end{proof}

La pertinence de la Proposition \ref{propdesomme} tient à ce que si l'on suppose 
l'injectivit\'{e} (conjecturale) de l'invariant modulaire $j$ dans ${\sf Cl}^{\rm mod}(A)$, 
alors $j(\mathfrak{a})\not=j(\Upsigma^{n}\mathfrak{a})$ pour tout $n\ne 1$.  Ceci implique que la structure 
que l'on d\'{e}veloppe ici en partant des collections de sommes de soleno\"{\i}des 
$\{( \hat{\SI}_{\Upsigma^{n}\mathfrak{a}}\,\Upsigma^{n}\hat{\mathfrak{a}})\}$ 
(ou de manière équivalente les sommes $\{ \Upsigma^{n}\mathfrak{a}\}$), cette structure ne produit pas un,
mais un ensemble d'invariants modulaires, tout comme l'invariant modulaire quantique de la Section \ref{modinvsec}.

Notons maintenant qu'il existe une action de $A$ sur $(\hat{\SI}_{\mathfrak{a}},\hat{\mathfrak{a}})$ donnée par
\[ \upalpha :(\hat{\SI}_{\mathfrak{a}},\hat{\mathfrak{a}})\longrightarrow (\hat{\SI}_{\mathfrak{a}},\hat{\mathfrak{a}}),\quad x+ \mathfrak{a}\longmapsto \upalpha x + \upalpha\mathfrak{a}\longmapsto \upalpha x + \mathfrak{a} \]
o\`{u} l'on utilise l'inclusion $\upalpha\mathfrak{a}\subset\mathfrak{a}$ pour d\'{e}finir la deuxi\`{e}me application
\footnote{Ici il faut remarquer que pour $\upalpha\not=0$, $\upalpha\mathfrak{a}$ et $\mathfrak{a}$ sont
\'{e}quivalents : dans le cas où $\mathfrak{a}=\mathfrak{a}_{x}$, ceci est conséquence du fait que {\it tous} 
les ensembles mod\`{e}les d\'{e}finis à l'aide d'un m\^{e}me r\'{e}seau sont \'{e}quivalents (voir \cite{Meyer}, \cite{Moody}).}.
Cette action s'étend à tous les $\Upsigma^{n}\mathfrak{a}$ et on
a le carr\'{e} commutatif suivant :
\begin{align}\label{endogram}
\begin{diagram}
(\hat{\SI}_{\Upsigma^{m_{1}}\mathfrak{a}}, \Upsigma^{m_{1}}\hat{\mathfrak{a}})\times\cdots\times(\hat{\SI}_{\Upsigma^{m_{k}}\mathfrak{a}}, \Upsigma^{m_{k}}\hat{\mathfrak{a}})  & \rTo^{\upalpha\times\dots\times\upalpha} & (\hat{\SI}_{\Upsigma^{m_{1}}\mathfrak{a}}, \Upsigma^{m_{1}}\hat{\mathfrak{a}})\times\cdots\times(\hat{\SI}_{\Upsigma^{m_{k}}\mathfrak{a}}, \Upsigma^{m_{k}}\hat{\mathfrak{a}})  \\
\dTo^{+} & & \dTo_{+} \\
(\hat{\SI}_{\Upsigma^{n}\mathfrak{a}}, \Upsigma^{n}\hat{\mathfrak{a}}) & \rTo_{\upalpha} & (\hat{\SI}_{\Upsigma^{n}\mathfrak{a}}, \Upsigma^{n}\hat{\mathfrak{a}}).
\end{diagram}
\end{align}
Autrement dit chaque $\upalpha\in A$ d\'{e}finit un \'{e}l\'{e}ment de ${\rm End}(\hat{\SI}_{\mathfrak{a}})$, 
d'où l'on tire une injection $A\hookrightarrow {\rm End}(\hat{\SI}_{\mathfrak{a}})$ de (presque) anneaux.
De plus pour tout $f=\{ f_{n}\}\in {\rm End}(\hat{\SI}_{\mathfrak{a}})$, les restrictions de $f_{n}$ \`{a} la feuille canonique $L_{\Upsigma^{n}\mathfrak{a}}\approx\R\subset\hat{\SI}_{\Upsigma^{n}\mathfrak{a}}$ définissent des
endomorphismes de $\R$ et c'est ainsi que l'on peut identifier  $A\subset {\rm End}(\hat{\SI}_{\mathfrak{a}})\subset\R$.  
Cependant pour que $r\in\R$ représente une telle restriction on doit avoir $r\mathfrak{a}\subset\mathfrak{a}$, qui implique
que $r\in \mathcal{O}_{K}$.  Mais ce sont pr\'{e}cisement les \'{e}l\'{e}ments de $A\subset\mathcal{O}_{K}$ qui 
pr\'{e}servent $\mathfrak{a}$ (\`{a} cause de sa propriét\'{e} caract\'{e}ristique $|\upgamma'|\leq 1$).  
Nous avons ainsi d\'{e}montr\'{e} le

\begin{theo} ${\rm End}(\hat{\SI}_{\mathfrak{a}}, \hat{\mathfrak{a}})=A$.
\end{theo}

Par analogie avec la th\'{e}orie classique des courbes elliptiques on dira que $\hat{\SI}_{\mathfrak{a}}$ 
est à {\em multiplication quasicristalline} par rapport \`{a} $A$
et que $\hat{\SI}_{\mathfrak{a}}$ est un {\em $ A$-module quasicristallin}.

Si $\mathfrak{m}\subset A$ est un id\'{e}al quasicristallin entier, on dit que $x\in \hat{\SI}_{\mathfrak{a}}$ 
est un point de {\em $\mathfrak{m}$-torsion} si pour tout $\upalpha\in\mathfrak{m}$
\[ \upalpha x\in \hat{\mathfrak{a}}.\]
On note $(\hat{\SI}_{\mathfrak{a}},\hat{\mathfrak{a}})[\mathfrak{m}]$ l'ensemble des points de $\mathfrak{m}$-torsion.

Remarquons que pour $\mathfrak{m}= A$, on a $(\hat{\SI}_{\mathfrak{a}},\hat{\mathfrak{a}})[A]=\hat{\mathfrak{a}}$ ; 
c'est l'analogue du fait que, pour un tore $\T=\C/\Uplambda$ à Multiplication Complexe par rapport \`{a} 
l'anneau de entiers $\mathcal{O}_{K}$, l'ensemble $\T[\mathcal{O}_{K}]$  des points de 
$\mathcal{O}_{K}$-torsion est réduit à $\{ 0\}$.

 La torsion est clairement compatible avec la somme, induisant  une application 
\[ +: (\hat{\SI}_{\mathfrak{a}},\hat{\mathfrak{a}})[\mathfrak{m}]\times  (\hat{\SI}_{\mathfrak{a}},\hat{\mathfrak{a}})[\mathfrak{m}]\longrightarrow (\hat{\SI}_{\Upsigma\mathfrak{a}}, \Upsigma\hat{\mathfrak{a}})[\mathfrak{m}] \]
qui est distributive par rapport \`{a} l'action de $A$.  Autrement dit la collection des paires
$\{  (\hat{\SI}_{\Upsigma^{n}\mathfrak{a}}, \Upsigma^{n}\hat{\mathfrak{a}})[\mathfrak{m}]\}_{n\geq 1}$ 
constitue un {\em presque $A$-module} :  pour chaque partition $m_{1}+\cdots +m_{k}=n$ on a l'operation
\[     (\hat{\SI}_{\Upsigma^{m_{1}}\mathfrak{a}}, \Upsigma^{m_{1}}\hat{\mathfrak{a}})[\mathfrak{m}] \times \dots \times (\hat{\SI}_{\Upsigma^{m_{k}}\mathfrak{a}}, \Upsigma^{m_{k}}\hat{\mathfrak{a}})[\mathfrak{m}])\stackrel{+}{\longrightarrow } (\hat{\SI}_{\Upsigma^{n}\mathfrak{a}}, \Upsigma^{n}\hat{\mathfrak{a}})[\mathfrak{m}], \]
qui est distributive par rapport à l'action multiplicative de $A$.

Nous terminerons avec une discussion de constructions associ\'{e}es au soléno\"{\i}de d'un quasicristal,
lesquelles constituent une tentative de complétion dudit solénoïde par rapport \`{a} la somme. 
Ces constructions sont motiv\'{e}es par l'envie de pouvoir travailler avec de ``vrais'' anneaux ou id\'{e}aux, 
par exemple l'anneau $\langle A\rangle $ ou l'id\'{e}al $\langle\mathfrak{a}\rangle$ engendr\'{e} par l'anneau 
quasicristallin $A$ ou un $A$-id\'{e}al $\mathfrak{a}$ : malheureusement, comme nous l'avons vu, 
ils sont tous deux \'{e}gaux \`{a} l'anneau $\mathcal{O}_{K}$ lui-m\^eme, le probl\`{e}me étant que ce dernier 
anneau n'est pas suffisamment grand pour distinguer les anneaux engendr\'{e}s par $A$ et les $\mathfrak{a}$.

 La premi\`{e}re construction est la plus \'{e}vidente : il s'agit de prendre la limite du syst\`{e}me direct des 
 {\it projections} entre les soléno\"{\i}des sommes :
\[ \check{\mathbb{S}}_{\mathfrak{a}}:= \lim_{\longrightarrow}\left(  \hat{\mathbb{S}}_{\mathfrak{a}}\rightarrow \hat{\mathbb{S}}_{\Upsigma^{2}\mathfrak{a}}\rightarrow \cdots \right) . \]
D'après la Proposition \ref{propdesomme},
 \begin{align}\label{limitedessommes} \bigcup \Upsigma^{n}\mathfrak{a}=\mathcal{O}_{K},\end{align}
et donc la limite des transversales 
\[  \check{\mathfrak{a}}:=\lim_{\longrightarrow} \Upsigma^{n}\hat{\mathfrak{a}} \]
contient l'image de $\mathcal{O}_{K}\subset\R\approx L_{\mathfrak{a}}\subset\hat{\SI}_{\mathfrak{a}}$ dans 
$ \check{\mathbb{S}}_{\mathfrak{a}}$, o\`{u} $L_{\mathfrak{a}}$ est la feuille canonique.  
En particulier, $\check{\mathfrak{a}}\subset  \check{\mathbb{S}}_{\mathfrak{a}}$ est dense ; par contre, $\check{\mathfrak{a}}\not=  \check{\mathbb{S}}_{\mathfrak{a}}$ parce que l'image $L_{\check{\mathfrak{a}}}$ de $L_{\mathfrak{a}}$ en $ \check{\mathbb{S}}_{\mathfrak{a}}$ satisfait $L_{\check{\mathfrak{a}}}\cap \check{\mathfrak{a}}=\mathcal{O}_{K}$.  
En particulier $\check{\mathfrak{a}}$  ne d\'{e}finit pas un solénoïde transverse. N\'{e}anmoins
on consid\`{e}rera la paire $(\check{\mathbb{S}}_{\mathfrak{a}},  \check{\mathfrak{a}})$ comme un cas limite ou singulier.

\begin{theo}\label{torequant1} $(\check{\mathbb{S}}_{\mathfrak{a}},  \check{\mathfrak{a}})$ est une paire de $\mathcal{O}_{K}$-modules.
Le quotient $\check{\mathbb{S}}_{\mathfrak{a}}/\check{\mathfrak{a}}$ est isomorphe
 au tore quantique 
$\T(\uptheta ):= \R/\mathcal{O}_{K}$, o\`{u} $\mathcal{O}_{K}=\Z [\uptheta ]$.
\end{theo}
\begin{proof}[D\'{e}monstration] Il est clair que $ \check{\mathbb{S}}_{\mathfrak{a}}$ est un groupe : 
si $x\in  \hat{\mathbb{S}}_{\Upsigma^{m} \mathfrak{a}}, y\in  \hat{\mathbb{S}}_{\Upsigma^{n} \mathfrak{a}}$, $m\leq n$,
la somme $x+y$ est d\'{e}finie dans $\hat{\mathbb{S}}_{\Upsigma^{n} \mathfrak{a}}$ en $y$, en envoyant $x$ par 
l'application $ \hat{\mathbb{S}}_{\Upsigma^{m} \mathfrak{a}}\rightarrow  \hat{\mathbb{S}}_{\Upsigma^{n}\mathfrak{a}} $  
ce qui donne un \'{e}l\'{e}ment de $\hat{\mathbb{S}}_{\Upsigma^{2m} \mathfrak{a}}$.  Cette d\'{e}finition est consistante 
dans la mesure où elle est compatible avec toutes les applications qui d\'{e}finissent la limite.  La somme fournit
donc une opération binaire qui est commutative et inversible dans $ \check{\mathbb{S}}_{\mathfrak{a}}$.  
Il est clair que la limite $\check{\mathfrak{a}}$ est ferm\'{e}e par rapport \`{a} la somme.  
\`{A} cause de (\ref{limitedessommes}), l'action naturel de $\mathcal{O}_{K}$ sur $\check{\SI}_{\mathfrak{a}}$ 
est continue et respecte le sous-groupe $\check{\mathfrak{a}}$.
 De plus il y a un syst\`{e}me  compatible d'épimorphismes vers le tore quantique 
 \[ \hat{\mathbb{S}}_{ \Upsigma^{n}\mathfrak{a}}\longrightarrow \T(\uptheta ):= \R/\mathcal{O}_{K}\]  induit par
 $x+\Upsigma^{n}\mathfrak{a}\mapsto x+\mathcal{O}_{K}$, o\`{u} l'on observe que cette application
 identifie la transversale $x+\Upsigma^{n}\hat{\mathfrak{a}}$ passant par $x$
 au point $x+\mathcal{O}_{K}\in \T (\uptheta )$.  Ce syst\`{e}me induit \`{a} son tour un épimorphisme
 $\check{\mathbb{S}}_{\mathfrak{a}}\rightarrow \T (\uptheta )$ de noyau $\check{\mathfrak{a}}$.
 \end{proof}

 Le deuxi\`{e}me construction est une variante de celle-ci dans le contexte de l'analyse non standard,
 le point étant qu'il est possible de r\'esoudre le prob\`eme de distinguer $\langle A\rangle$ de $\mathcal{O}_{K}$ 
 en travaillant dans un mod\`{e}le non standard de ce dernier anneau.  Pour un ensemble infini $X$ on note
 $\bast X$ l'ultrapuissance de $X$ par rapport \`{a} un ultrafiltre non principal $\mathfrak{u}\subset {\sf 2}^{\N}$ fix\'{e}.
 Par d\'{e}finition on a donc
\[  \bast X:= X^{\N}/\sim_{\mathfrak{u}} \]
o\`{u}
\[   \{x_{i}\}\sim_{\mathfrak{u}}  \{x'_{i}\}  \Longleftrightarrow \{ i|\; x_{i}=x_{i}'\} \in\mathfrak{u}. \]
Alors l'ultrapuissance $\bast \mathcal{O}_{K}$ est un anneau  et une {\it extension \'{e}l\'{e}mentaire} 
de $\mathcal{O}_{K}$ (\cite{Las}), ce qui implique que $\bast \mathcal{O}_{K}$ et $\mathcal{O}_{K}$ sont 
indistinguables par les propositions de la logique du premier ordre.

Soit $\bast  A\subset \bast\mathcal{O}_{K}$ l'ultrapuissance de $A$  et consid\'{e}rons l'anneau engendr\'{e} par $\bast A$ :
\[   \langle \bast A\rangle :=  \bigcup_{n\geq 1} \Upsigma^{n}\bast \! A.\]
\begin{prop} $ {\rm Frac}( \langle\bast \!A\rangle)=\bast K$.\end{prop}
\begin{proof}[D\'{e}monstration] 
On sait par le travail de Salem que $K=\Q (\uptheta)$ o\`{u} $\uptheta$ est de PV, donc $\bast K=\bast\Q (\uptheta )$.   
Soit $\mathcal{K}:= {\rm Frac}( \langle\bast \!A\rangle)$.  
Puisque $A$ consiste en tous les nombres PVS (complexes), 
\[ \mathcal{K}\supset \Q\{\uptheta^{\bast  n}|\; \bast n\in\bast\Z \}.\] 
Reste à montrer que $\bast\Q \subset\mathcal{K}$.  Mais si $\bast  a\in\Z$ est arbitraire, 
$\bast  a\uptheta^{\bast \!n}\in\bast A$ si $\bast \!n$ est assez grand, parce que 
$(\bast  a\uptheta^{\bast \!n})'=\bast  a(\uptheta')^{\bast \!n}$ est la classe d'une suite ayant 
des valeurs absolues $\leq 1$. Donc $\bast \!a\in
\mathcal{K}$.  D'où $\bast\Z\subset \mathcal{K}$, puis $\bast\Q\subset \mathcal{K}$.\end{proof}

De m\^{e}me, pour tout $A$-id\'{e}al quasicristallin $\mathfrak{a}$, on d\'efinit le $\langle\bast A\rangle$-id\'{e}al 
\[   \langle \bast  \mathfrak{a}\rangle :=  \bigcup_{n\geq 1}  \Upsigma^{n}\bast \mathfrak{a}.\]

Les intersections ou unions d'ensembles d\'{e}nombrables qui sont triviales dans le mod\`{e}le standard 
$\mathcal{O}_{K}$ peuvent devenir non triviales dans $\bast \mathcal{O}_{K}$.  En effet :

\begin{prop}  $\langle \bast A\rangle\subsetneq\bast\mathcal{O}_{K}$.  Si ${\bf x}<{\bf y}$, alors
$ \langle \mathfrak{a}_{\bf y}\rangle \supsetneq  \langle \mathfrak{a}_{\bf x}\rangle$.
\end{prop}
\begin{proof}[D\'{e}monstration] 
Soit 
\[ \bast\upalpha=\bast \{ \upalpha_{i}\}\in \bast \mathcal{O}_{K} \text{ tel que } \upalpha_{i}\not\in \Upsigma^{i}A \text{ pour tout }i\in \N.\] 
Comme $\Upsigma^{n}A\subsetneq \Upsigma^{n+1}A$ pour tout $n\in\N$, $\upalpha_{j}\not\in \Upsigma^{i}A$ pour tout $j\geq i$
et donc 
\[ \bast\upalpha = \bast \{\upalpha_{i}\}\not\in \bast( \Upsigma^{n}A)=\Upsigma^{n}\bast  A\]
pour tout $n$.  
La preuve de la seconde affirmation est identique. 
  \end{proof}

Consid\'{e}rons maintenant $\bast\,\R$ comme un corps topologique muni de la 
topologie ordonn\'ee, ce qui en fait un ensemble totalement discon\`{e}xe. 

 \begin{prop} $(\langle  \bast  \mathfrak{a}\rangle, +)$ est un sous-groupe de Delaunay de $\bast\, \R$.
 \end{prop}
 
 \begin{proof}[D\'{e}monstration]  Il est clair que $ \langle\bast   \mathfrak{a}\rangle$ est 
 relativement dense dans $\bast\,\R$. Soit $r_{n}$ tel que $\Upsigma^{n} \mathfrak{a}$ est uniform\'{e}ment 
 discret par rapport \`{a} $r_{n}$.  Alors $\Upsigma^{n}\bast  \mathfrak{a}$ est aussi
 uniform\'{e}ment discret par rapport a $r_{n}$.  Si $\bast \upvarepsilon>0$ est infinitésimal,
 $\bast\upvarepsilon <r_{n}$ pour tout $n$, donc $\langle \bast   \mathfrak{a}\rangle$
 est $\bast\upvarepsilon$-uniform\'{e}ment discret.
 \end{proof}
 
 Pour tout $\mathfrak{a}$ id\'{e}al quasicristallin de $A$, le quotient de Hausdorff 
\[  \bast\, \T_{\mathfrak{a}} := \bast\,\R/ \langle \bast \mathfrak{a} \rangle \]
est un {\em tore quasicristallin non standard}.  On note que $ \bast\, \T_{\mathfrak{a}}$ n'est pas compact au 
sens usuel parce que la topologie ordonnée sur $\bast\,\R$ n'est pas 2-d\'{e}nombrable.

Si l'on note $\bast\, \R_{\upvarepsilon}\subset\,\bast\,\R$ le sous espace vectoriel des infinit\'{e}simaux, 
le quotient \[ \bbull\R := \bast\,\R/  \bast\, \R_{\upvarepsilon}\]
joue le r\^{o}le de rev\^{e}tement universel de certaines laminations de dimension 1;  en particulier
\[ \mathcal{F}(\uptheta ) = \{ \text{le feuilletage de Kronecker de pente }\uptheta \} = \bbull\R/\bast\, \Z (\uptheta ), \]
o\`{u} $\bast\, \Z (\uptheta )\subset\bast\,\Z$ est le {\em groupe des approximations diophantiennes} (voir \cite{Ge-C}).
Dans le cas présent, $ \langle \bast A\rangle$ étant $\bast\upvarepsilon$-uniform\'{e}ment discret 
pour tout $\bast\upvarepsilon\in \bast\, \R_{\upvarepsilon}$, on a
\[  \bast\, \R_{\upvarepsilon} \cap \langle \bast A\rangle= \{ 0\} \] ce qui implique qu'il existe
un plongement  
\[ \langle \bast A\rangle\hookrightarrow \bbull \R .\]
Mais l'image de $ \langle \bast A\rangle$ dans $\bbull\R$ est dense : en fait  
$\langle \bast A\rangle\cap \R=\mathcal{O}_{K}$ qui est dense en $\R$. Donc $ \bast\, \T_{\mathfrak{a}}$ est 
un objet v\'{e}ritablement non standard : on perd la propriet\'{e} de Hausdorff
en annulant les infinit\'esimaux.  Le r\'{e}sultat suivant est l'analogue du Th\'{e}or\`{e}me \ref{torequant1}.

\begin{theo} $\bbull\R/\langle\bast A\rangle$ est isomorphe au tore quantique $\R/\mathcal{O}_{K}$.
\end{theo}

\begin{proof}[D\'{e}monstration]  On note que $\langle\bast A\rangle\supset \langle A\rangle=\mathcal{O}_{K}$ et que 
l'intersection des ensembles $\langle\bast A\rangle, \R\subset \bbull\R$ 
est $\langle A\rangle \cap\R=\mathcal{O}_{K}$.  De plus chaque ``feuille'' $\bbull r+\bbull\R$ possède une intersection 
non nulle avec $\langle\bast A\rangle$, d'où
$\bbull\R/\langle\bast A\rangle=\R/\langle A\rangle =\R/\mathcal{O}_{K}$.
\end{proof}

 
 

\section{Fonction exponentielle}\label{exp}

Dans cette section finale on d\'{e}veloppe une trigonom\'{e}trie quasicristalline que l'on l'utilise pour d\'{e}finir 
entre autres une exponentielle, laquelle fournit un premier analogue des modules de Drinfeld en caract\'{e}ristique nulle.

Rappelons d'abord quelques formules de trigonométrie ordinaire dans la représen\-tation en produits infinis.
Dans ce qui suit on utilise les fonctions {\it sinus} et {\it cosinus} ``absolues", c'est-à-dire {\it non normalisées} 
par rapport à $\uppi$, définies par les formules suivantes :
\[{\tt s}(x) :=x\prod_{n=1}^{\infty} \left(1- \frac{x^{2}}{n^{2}}\right) ,\quad {\tt c}(x) = {\tt s}' (x):=\prod_{n=1}^{\infty} \left(1- \frac{x^{2}}{(n-1/2)^{2}}\right) .\]

On rappelle également la formule de Wallis :
\[  \frac{\uppi}{2} = \prod_{n=1}^{\infty} \frac{4n^{2}}{(2n-1)(2n+1)} =  \prod_{n=1}^{\infty} \frac{n^{2}}{(n-1/2)(n+1/2)} .\]
Notons que le $n$-ième terme du produit est le quotient du carré de la racine, $n$, de ${\tt s}(x)$, 
divisée par le produit $(n-1/2)(n+1/2)$ des racines de ${\tt c}(x)$ voisines de $n$.
Donc, si l'on écrit $\upalpha_{n}=n$, $\upbeta_{n}= n-1/2$, les ensembles $\Z_{+}=\{ \upalpha_{n}\}$, $\Z_{\frac{1}{2},+}=\{ \upbeta_{n}\}$
sont les racines positives de ${\tt s}(x)$, ${\tt c}(x)$ et on a\footnote{N.B. Il n'est pas possible de réintégrer
le facteur $1/\upbeta_{1}$ dans le produit, en écrivant plutôt
$\prod ( \upalpha_{n}/\upbeta_{n})^{2}$ : le résultat ne converge pas. } :
\begin{align}\label{WallisReimagined}\uppi = \frac{1}{\upbeta_{1}}  \prod_{n=1}^{\infty} \frac{\upalpha_{n}^{2}}{\upbeta_{n}\cdot \upbeta_{n+1}}.\end{align}

Supposons que l'on remplace les fonctions absolues en multipliant le réseau $\Z$ par une constante 
$\upxi$ : $\Z\rightarrow \upxi\Z$.   Notons ${\tt s}_{\upxi} $ la fonction sinus associée :
\[ {\tt s}_{\upxi}(x) =  x\prod \left(1- \frac{x^{2}}{\upxi^{2}n^{2}}\right).   \]
Puis
\[  {\tt s}_{\upxi} =\upxi \circ   {\tt s}\circ \upxi^{-1} ,\quad  {\tt c}_{\upxi}(x) :={\tt s}'_{\upxi}(x) = {\tt c}( \upxi^{-1}x)\]
et les ensembles de zéros de ${\tt s}_{\upxi}(x), {\tt c}_{\upxi}(x)$ sont $\upxi\Z, \upxi\Z_{\frac{1}{2}}$.  Si l'on définit  $\uppi_{\upxi}$ 
en utilisant (\ref{WallisReimagined}),
avec les ensembles de zéros  $\upxi\Z, \upxi\Z_{\frac{1}{2}}$ au lieu de $\Z, \Z_{\frac{1}{2}}$, nous voyons que
\[ \uppi_{\upxi}=\uppi/\upxi,\]
et les réseaux normalisés coïncident :
\[ \uppi \Z = \uppi_{\upxi}\upxi \Z.\]
 On en conclut que les fonctions sinus et cosinus classiques, s'écrivent dans cette notation
 \[ \sin (x) = {\tt s}_{\uppi}(x)  =\uppi {\tt s}(\uppi ^{-1}x),\quad \cos (x) = {\tt c}_{\uppi}(x)  = {\tt c}(\uppi^{-1} x),\]
 proviennent donc d'une construction canonique qui ne dépend que la classe projective 
 de $\Z$.  En effect, si en guise de notion ``absolue'' du sinus nous avions utilisé $\upxi \Z$ au lieu de $\Z$,
 nous aurions obtenu
 \[ \sin_{\upxi}(x) = ({\tt s}_{\upxi})_{\uppi_{\upxi}}(x) =  \uppi_{\upxi}{\tt s}_{\upxi}( \uppi^{-1}_{\upxi}x)=
 \uppi\upxi^{-1}{\tt s}_{\upxi}(\upxi  \uppi^{-1} x)
 ={\tt s}_{\uppi}(x) = \sin (x).\]

Soit $\mathfrak{a}\subset K\subset\R$ un id\'{e}al quasicristallin fractionnaire de rang 1. On commence par définir le sinus et cosinus ``absolus''
\[ {\tt s}_{\mathfrak{a}}(x) := x\prod_{0<\upalpha\in\mathfrak{a}} \left(1- \frac{x^{2}}{\upalpha^{2}}\right) ,\quad {\tt c}_{\mathfrak{a}}(x) :={\tt s}'_{\mathfrak{a}}(x).\]
On observe immédiatement que les deux fonctions sont lisses,  $ {\tt s}_{\mathfrak{a}}(x)$ est impaire et  ${\tt c}_{\mathfrak{a}}(x)$ est paire.
Soit
\[ {\tt e}_{\mathfrak{a}} (ix) := {\tt c}_{\mathfrak{a}}(x) + i{\tt s}_{\mathfrak{a}} (x)\]
l'exponentielle associée.
  On observe 
que l'image de ${\tt e}_{\mathfrak{a}}(ix)$ est contenue dans $\C^{\ast}\subset\C$ : c'est une cons\'{e}quence
imm\'{e}diate de la formule de produit qui d\'{e}finit ${\sf s}_{\mathfrak{a}}$ et 
\[ {\tt c}(0) = \lim_{x\rightarrow 0} \frac{{\tt s}(x)}{x} =1.\] 
La Figure 1 montre l'image par ${\tt e}_{\mathfrak{a}}(ix)$  d'un segment (compact)
de $\R\subset  \hat{\SI}_{\mathfrak{a}}^{\C}$, o\`{u} $\mathfrak{a}=A$, l'anneau quasicristallin 
associ\'{e} \`{a} $K=\Q(\upvarphi )$, avec $\upvarphi$ le nombre d'or.  On observe la formation d'une 
suite de cardio\"{\i}des autosimilaires, manifestation experimentale du th\'{e}or\`{e}me suivant :

\begin{figure}
\centering
\includegraphics[width=1\textwidth]{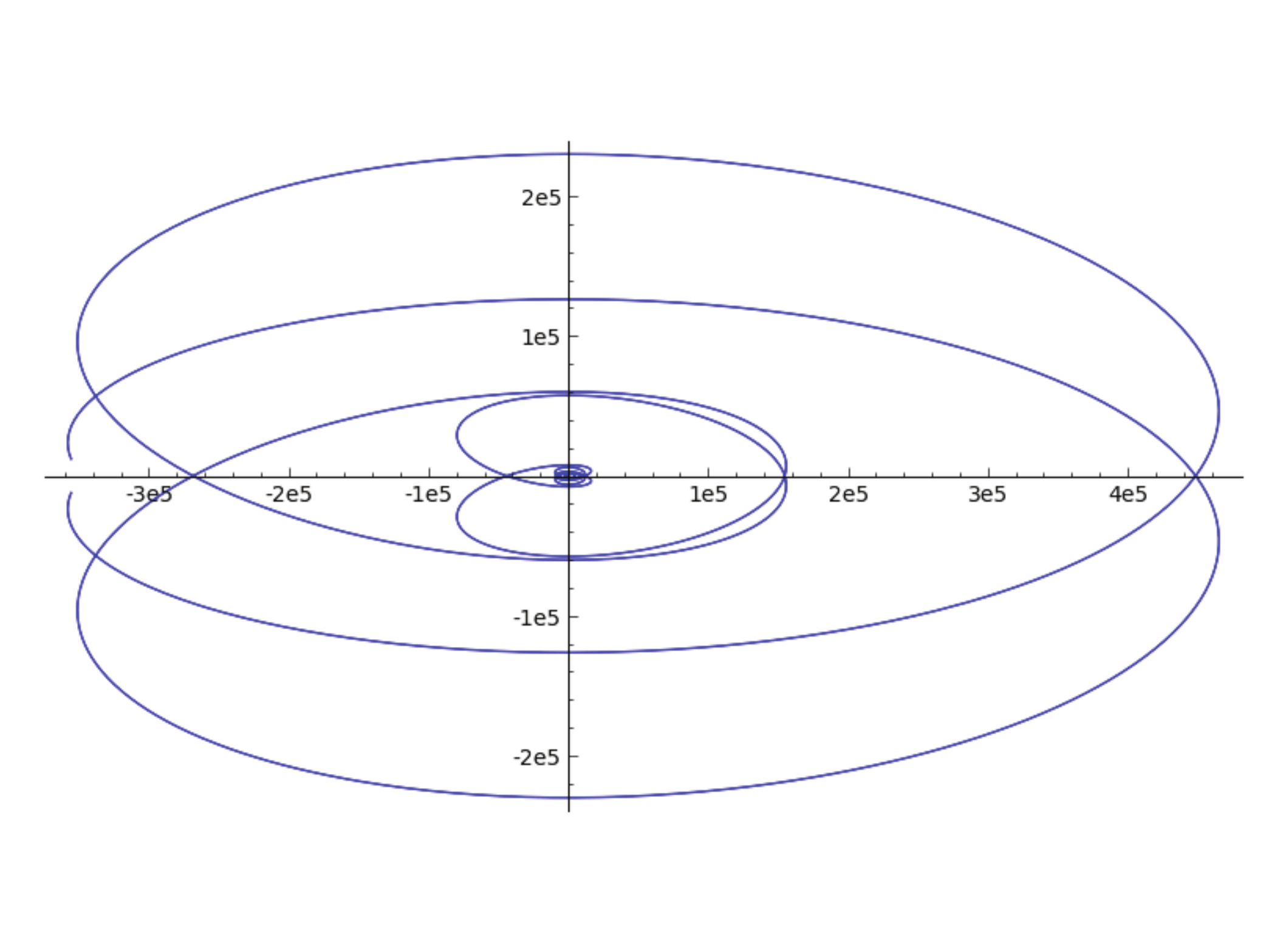}
\caption{Une feuille d'un soléno\"{\i}de quasicristallin associ\'{e} au nombre d'or}
\end{figure}

\begin{theo}[$\mathfrak{a}$-periodicit\'{e}]\label{aperiodicite}  Soient $\mathfrak{a}$ un id\'{e}al quasicristallin mod\`{e}le 
 de rang {\rm 1} et  $ \hat{\SI}_{\mathfrak{a}}^{\C}:=\hat{\SI}_{\mathfrak{a}} \times i\R$ le soléno\"{\i}de complexifi\'{e},
un solénoïde non compact dont toutes les feuilles sont isomorphes \`{a} $\C$. La fonction de variable complexe $z\mapsto  {\tt e}_{\mathfrak{a}}(iz)$ est continue par rapport \`{a} 
 la topologie transversale de $\hat{\SI}_{\mathfrak{a}}^{\C}$ et s'\'{e}tend donc en une fonction continue 
 \[    {\tt e}_{\mathfrak{a}}(i\,\cdot ):\hat{\SI}_{\mathfrak{a}}^{\C} \longrightarrow \C^{\ast}\]
holomorphe le long des feuilles.  
L'image de $\hat{\mathfrak{a}}\subset\hat{\SI}_{\mathfrak{a}}$ est
$  {\tt e}_{\mathfrak{a}}(i\,\hat{\SI}_{\mathfrak{a}})\cap \R$.
  Si $\mathfrak{a}$ est générique la restriction de $ {\tt e}_{\mathfrak{a}}$ \`{a} $\hat{\SI}_{\mathfrak{a}}$  
  n'est pas injective.
 \end{theo}
 
\begin{proof}[D\'{e}monstration] On commence par démontrer que ${\tt s}(z)$ a une extension à $\hat{\SI}_{\mathfrak{a}}^{\C}$; il suffit de démontrer l'existence d'une extension à $\hat{\SI}_{\mathfrak{a}}$.  On rappelle que $\upgamma_{i}\rightarrow 0$ en $\hat{\mathfrak{a}}$ si et seulement si pour tout $R>0$, il existe $M\in\N$ tel que
$\mathfrak{a}-\upgamma_{i}$ et $\mathfrak{a}$ coincide sur l'intervalle $[-R,R]$ pour tout $i\geq M$.  
Par conséquent,  par sa définition comme produit infini, la fonction
${\tt s}_{\mathfrak{a}}(x)$, dans chaque intervalle $[-R,R]$, est uniformément proche de
${\tt s}_{\mathfrak{a}}(x-\upgamma_{i})$, $i\geq M$.  Autrement dit, 
${\tt s}_{\mathfrak{a}}(x)$ est transversalment continue en $0$, uniformément sur les compacts.  
De la même manière ${\tt s}_{\mathfrak{a}}(x)$ possède la
même propriété de convergence en chaque point du complété $\hat{\mathfrak{a}}\subset\hat{\SI}_{\mathfrak{a}}$, 
d'où l'on déduit l'existence d'une extension à ce dernier.  La convergence, transversalement et uniformément
sur les compacts, de la fonction holomorphe ${\tt s}_{\mathfrak{a}}(z)$ implique la même propriété 
pour ses dérivées.  En particulier,  ${\tt c}_{\mathfrak{a}}(z)$ possède
également une extension à $\hat{\SI}_{\mathfrak{a}}^{\C}$.  
 Il est clair que si $ {\tt e}_{\mathfrak{a}}(i\, \hat{z})=1$, ${\tt s}_{\mathfrak{a}}(\hat{z})=0$, puis $\hat{x}\in\hat{\mathfrak{a}}$.  
Si $\mathfrak{a}$ est r\'{e}p\'{e}titif, $\hat{\SI}_{\mathfrak{a}}$ est un soléno\"{\i}de compact minimal.  En particulier, 
l'image de la feuille canonique $L_{\mathfrak{a}}\approx \R$ est dense dans l'image de $\hat{\SI}_{\mathfrak{a}}$, 
ce qui implique qu'elle possède des auto-intersections.
 \end{proof}

\begin{prop}\label{a'qc} L'ensemble des zéros 
\[ \mathfrak{b} :=\{ \upbeta\in\R|\; {\tt c}_{\mathfrak{a}}(\upbeta)=0\}\] est de Delaunay et on a la représentation en produit :
\begin{align}\label{cosprod}{\tt c}_{\mathfrak{a}}(x)=\prod_{0<\upbeta\in\mathfrak{b}} \left(1- \frac{x^{2}}{\upbeta^{2}}\right).\end{align}
\end{prop}

\begin{proof}[Démonstration]   En effet, écrivons ${\tt c}_{\mathfrak{a}} (x)={\tt s}_{\mathfrak{a}} '(x)$
\[{\tt c}_{\mathfrak{a}} (x)= \prod_{0<\upalpha\in\mathfrak{a}} \left(1- \frac{x^{2}}{\upalpha^{2}}\right) - \sum_{0<\upalpha\in\mathfrak{a}}  \left( \frac{2x^{2}}{\upalpha^{2}}\right)
\prod_{0<\upgamma\not=\upalpha, \;\upgamma\in\mathfrak{a}} \left(1- \frac{x^{2}}{\upgamma^{2}}\right).
\]
En utilisant la dernière, on constate qu'entre chaque paire de zéros consécutifs de $\mathfrak{a}$, ${\tt c}_{\mathfrak{a}}$ change de signe précisément un fois.  Autrement dit il y a un et un seul élément de $\mathfrak{b}$ entre chaque paire de zéros consécutifs 
de $\mathfrak{a}$ et on peut énumérer les éléments de $\mathfrak{a},\mathfrak{b}$ :
\[  \cdots  \upalpha_{-2}<\upbeta_{-2}<\upalpha_{-1}<\upbeta_{-1}< 0<\upbeta_{1} <\upalpha_{1} < \upbeta_{2} <\upalpha_{2} <\cdots , \]
où $\upalpha_{-n}:=-\upalpha_{n},\; \upbeta_{-n}:=-\upbeta_{n}$.
Donc, $\mathfrak{b}$ est relativement dense.  Par la démonstration du Théorème \ref{aperiodicite}, ${\tt c}_{\mathfrak{a}}={\tt s}'_{\mathfrak{a}}$ s'étend de façon continue
au solenoïde compact $\hat{\SI}_{\mathfrak{a}}$, alors ${\tt c}_{\mathfrak{a}}$ est bornée.   
Si $\mathfrak{b}$ n'est pas uniformément discret, il existe une suite $\upbeta_{n}< \upalpha_{n}<\upbeta_{n+1}$
d'éléments consécutifs tels que $\upbeta_{n+1}-\upbeta_{n}\rightarrow 0$.  Mais le dernier implique que ${\tt c}_{\mathfrak{a}}$ a des valeurs arbitrairement grands entre tels extremes consecutifs, c'est-à-dire, il existe une suite $\upbeta_{n}<x_{n}<\upbeta_{n+1} $ tels que $|{\tt c}_{\mathfrak{a}}(x_{n})|\rightarrow \infty$,   
c'est-à-dire, ${\tt c}_{\mathfrak{a}}$ n'est pas borné.  
Donc, on conclut que $\mathfrak{b}$ est en ensemble de Delaunay.  Le produit dans la ligne (\ref{cosprod}) converge uniformément sur les compacts (comme fonction de $x\in\C$), alors il définit une fonction
en $\C$ avec la même diviseur que ${\tt c}_{\mathfrak{a}}$ et la même valeur, 1,  en $x=0$.  Donc, les deux fonctions sont égales.
\end{proof}


Par contre, l'exponentielle usuelle ne s'étend pas à $\hat{\SI}_{\mathfrak{a}}$.

Comme dans la preuve de la Proposition \ref{a'qc}, notons les éléments positifs de $\mathfrak{a}$, $\mathfrak{b}$ sous 
la forme $0<\upalpha_{1}<\upalpha_{2}<\cdots $, $0<\upbeta_{1}<\upbeta_{2}<\cdots $.  
En suivant à nouveau (\ref{WallisReimagined})
on définit
 \[ \uppi_{\mathfrak{a}} := \frac{1}{\upbeta_{1}}  \prod_{n=1}^{\infty} \frac{\upalpha_{n}^{2}}{\upbeta_{n}\cdot \upbeta_{n+1}}.   \]
Les fonctions trigonometriques {\it normalisées} de la variable complexe $z$ sont alors définies par 
 \begin{align}\label{defsinus} \sin_{\mathfrak{a}}(\uppi_{\mathfrak{a}}z) := ({\tt s}_{\mathfrak{a}})_{\uppi_{\mathfrak{a}}}( \uppi_{\mathfrak{a}}z)=\uppi_{\mathfrak{a}}z \prod_{0<\upalpha\in\mathfrak{a}} \left(1-\frac{z^{2}}{\upalpha^{2}}\right) 
 \end{align}
 et
 \[   \cos_{\mathfrak{a}}(\uppi_{\mathfrak{a}}z) := ({\tt c}_{\mathfrak{a}})_{\uppi_{\mathfrak{a}}}( \uppi_{\mathfrak{a}}z) =\prod_{0<\upbeta\in\mathfrak{b}} \left(1-\frac{z^{2}}{\upbeta^{2}}\right) 
 .\]
Enfin l'{\em exponentielle quasicristalline normalisée} associ\'{e}e \`{a} $\mathfrak{a}$ est donnée par
 \[ 
 \exp_{\mathfrak{a}}(\uppi_{\mathfrak{a}}iz) :=\text{\tt e}_{\uppi_{\mathfrak{a}}}(i \uppi_{\mathfrak{a}}z):=  \cos_{\mathfrak{a}}(\uppi_{\mathfrak{a}}z) +i \sin_{\mathfrak{a}}(\uppi_{\mathfrak{a}}z)  .
  \]

 Notons 
 \[ \hat{\sf C}_{\mathfrak{a}} :=    \exp_{\mathfrak{a}}(\uppi_{\mathfrak{a}}i\hat{\mathfrak{a}})\subset  \hat{\E}_{\mathfrak{a}} :=   \exp_{\mathfrak{a}}(\uppi_{\mathfrak{a}}i\hat{\SI}_{\mathfrak{a}})\subset\C^{\ast} .\]
 Ci-dessous nous utiliserons les notations suivantes :
\[ s_{\mathfrak{a}}(x) 
:=\sin_{\mathfrak{a}} (\uppi_{\mathfrak{a}}x),\quad c_{\mathfrak{a}}(x) :=\cos_{\mathfrak{a}} (\uppi_{\mathfrak{a}}x),\quad e_{\mathfrak{a}}(x) :=\exp_{\mathfrak{a}} (\uppi_{\mathfrak{a}}ix). \]
 Observons que, puisque $ s_{\mathfrak{a}}(\hat{\mathfrak{a}})\equiv 0$, 
 \[  1\in \hat{\sf C}_{\mathfrak{a}}=  c_{\mathfrak{a}}(\hat{\mathfrak{a}}) =  \hat{\E}_{\mathfrak{a}} \cap \R  \]
 est l'image continue d'un ensemble de Cantor, donc est elle-même un ensemble de Cantor, 
 ou bien finie.  On regarde $\hat{\sf C}_{\mathfrak{a}}$ comme un ``élargissement'' de $1$, de la m\^{e}me
 fa\c{c}on que l'on a regard\'{e} $\hat{\mathfrak{a}}\subset \hat{\SI}_{\mathfrak{a}}$ comme un élargissement de $0$.
La paire \[ ( \hat{\E}_{\mathfrak{a}}, \hat{\sf C}_{\mathfrak{a}})\] est l'objet que l'on voudrait considérer comme 
un analogue d'un module de Drinfeld de rang 1 en caract\'{e}ristique z\'{e}ro.

 Comme on le verra, ceci se complique du fait du défaut d'injectivit\'{e} de $e_{\mathfrak{a}}$.  
 Ce serait trop esp\'{e}rer que de vouloir faire de $e_{\mathfrak{a}}$ un homomorphisme dans le sens 
 usuel entre la structure sommatoire dans la suite des soleno\"{\i}des $\{ \hat{\SI}_{\Upsigma^{n}\mathfrak{a}}\} $ 
 et la structure de groupe de $\C^{\ast}$.  Plus pr\'{e}cis\'{e}ment, le diagramme
 \[ \begin{diagram}
 \hat{\SI}_{\mathfrak{a}} \times  \hat{\SI}_{\mathfrak{a}} & \rTo^{+} & \hat{\SI}_{\Upsigma\mathfrak{a}} \\
 \dTo^{ e_{\mathfrak{a}}\times e_{\mathfrak{a}} } & \text{\footnotesize{non}} \circlearrowleft  &\dTo_{e_{\Upsigma\mathfrak{a}}} \\
 \C^{\ast} \times  \C^{\ast} &\rTo_{\times} & \C^{\ast}
 \end{diagram}\]
 {\em ne commute pas}.  
Dans ce contexte il convient alors de considerer une notion de morphisme plus faible.   
 
 Soient
$G=(G,+ ) $ et $ H=(H, + )$ des structures munies de produits commutatifs partiellement définis et $|\cdot |:H\longrightarrow\R$ une valuation.   On rappelle qu'une fonction  
 \[ f:  G\longrightarrow H\]
 est un {\it quasi morphisme} si la somme $f(g) + f(h)$ est définie quand $g+ h$ l'est et il existe un constant $D$,
 ne dépendant que de $G$ et $H$ telle que
 \[ \left| f(g + h) -f(g) -f(h) \right| \leq D .\]

Dans le cas des solenoïdes quasicristallins il est commode d'introduire la variante suivante de la notion de quasi morphisme~:
on suppose qu'il existe
des structures auxiliares $\Upsigma G$, $\Upsigma H$ où les sommes prennent  formellement leurs valeurs, et que $\Upsigma$ est maintenant muni d'une valuation $|\cdot |$. Un {\it quasi morphisme adapté} consiste alors en
un couple de fonctions
\[ f: G\longrightarrow H,\quad \Upsigma f: \Upsigma G\longrightarrow \Upsigma H\]
telles que
 \[ \left| \Upsigma f(g + h) -f(g) -f(h) \right| \leq D .\]

 \begin{theo} L'exponentielle quasicristalline definit un quasi morphisme adapté
 \[ e_{\mathfrak{a}} : (\hat{\SI}_{\mathfrak{a}} , +)\longrightarrow (\C^{\times}, \times )\]
au sens qu'il existe une constante $D>0$ telle que
 \[  \left| \frac{  e_{\Upsigma\mathfrak{a}} (x+y )}{e_{\mathfrak{a}}(x)e_{\mathfrak{a}}(y) } \right| <D.\]
 \end{theo}
 
 \begin{proof}[Démonstration] L'affirmation suit immediatement du fait que les exponentielles $ e_{\Upsigma\mathfrak{a}}, e_{\mathfrak{a}}$, fonctions continues sur les solenoïdes compacte
 $\hat{\SI}_{\mathfrak{a}}, \hat{\SI}_{\Upsigma\mathfrak{a}}$, ont des images compactes qui omettent zéro. 
 \end{proof}
 
 \begin{note} Vues comme des fonctions sur les solenoïdes complexes $\hat{\SI}^{\C}_{\mathfrak{a}}$ et $ \hat{\SI}^{\C}_{\Upsigma\mathfrak{a}}$, les exponentielles $ e_{\Upsigma\mathfrak{a}}, e_{\mathfrak{a}}$ 
 ne définissent {\it pas} un quasi morphisme adapté.  En effet, sinon 
 on aurait pour tous $z,w\in\C$
\[ \left| \frac{ e_{\Upsigma\mathfrak{a}}(z+w)}{ e_{\mathfrak{a}}(z)  e_{\mathfrak{a}}(w)}\right| \leq D \]
pour un certain $D>0$. 
Mais d'après le théorème de Liouville, ce quotient holomorph serait constant, soit 
\[  \frac{ e_{\Upsigma\mathfrak{a}}(z+w)}{ e_{\mathfrak{a}}(z)  e_{\mathfrak{a}}(w)} \equiv C\]
avec en fait  $C=1$ puisque $1$ est un valeur atteinte.  Les exponentielles respecteraient donc le produit
\begin{align}\label{condhom} e_{\Upsigma\mathfrak{a}}(z+w) = e_{\mathfrak{a}}(z) e_{\mathfrak{a}}(w),\end{align}
et l'inversion.  
Restreignant (\ref{condhom}) à $\R$, on obtiendrait une application holomorphe non constante $e:\R\rightarrow \C^{\ast}$, dont l'image serait contenue dans le cercle unité $\SI^{1}$.  En particulier,
les images des solenoïdes $\hat{\SI}_{\mathfrak{a}}$, $\hat{\SI}_{\Upsigma\mathfrak{a}}$ par $e_{\mathfrak{a}}$, $e_{\Upsigma \mathfrak{a}}$ seraient égales à $\SI^{1}$.

Or pour une unité de Pisot  $\uprho \in\R$, $|e_{\mathfrak{a}}(\uprho )|\not=1$.  Si par exemple on choisit  $\uprho = \upvarphi$ (le nombre d'or) et $\mathfrak{a} = \mathfrak{a}_{0}$, $\upvarphi \in \mathfrak{a}$ est l'élément positif le plus petit et nous avons
\[  e_{\mathfrak{a}}(i\upvarphi ) = c _{\mathfrak{a}}(\upvarphi ) =s' _{\mathfrak{a}}(\upvarphi ) =  -2\prod_{\upvarphi<\upalpha\in\mathfrak{a}} \left( 1-\frac{\upvarphi^{2}}{\upalpha^{2}}   \right) .   \]
\'Ecrivons les éléments positifs de $\mathfrak{a}$ comme
\[ 0 < \upalpha_{1}=\upvarphi  < \upalpha_{2}< \upalpha_{3}<\cdots  .\]
On sait que, comme  quasicrystal, $\mathfrak{a}$ est formé de trois intervalles, de longueurs respectives$\upvarphi$,  1 et $\upvarphi^{-1}$~;
donc, pour tout $n$, soit  $\upalpha_{n+1}= \upalpha_{n}+\upvarphi$, soit $\upalpha_{n+1}= \upalpha_{n}+1$, soit   $\upalpha_{n+1}=\upalpha_{n}+\upvarphi^{-1}$.  Voici quelques exemples  
\begin{itemize}
\item[] $\upalpha_{2}=\upvarphi +1 = \upvarphi^{2}$, puisque
$\upvarphi +\upvarphi^{-1}\not\in\mathfrak{a}$.  
\item[]  $\upalpha_{3} = \upalpha_{2} + \upvarphi = \upvarphi^{3} $ car $ \upvarphi^{2}+1,\upvarphi^{2}+\upvarphi^{-1} \not\in\mathfrak{a}$.
\item[]  $\upalpha_{4}= \upalpha_{3}+1= \upvarphi^{3} +1$.
\item[] $\upalpha_{5} = \upalpha_{4} + \upvarphi^{-1} =  \upvarphi^{3} +1 + \upvarphi^{-1}= \upvarphi^{3} + \upvarphi$.
\end{itemize}
   Observons que $\upalpha_{2}/\upvarphi=\upvarphi <2$~; on voit par induction que $\upalpha_{n}/\upvarphi < n$ pour tout $n$.  En effet, si l'on suppose ceci vaut pour $n$, on a 
\[  \frac{\upalpha_{n+1}}{\upvarphi} = \frac{\upalpha_{n}}{\upvarphi}+ \left\{   \begin{array}{ll}
1   & \text{si  $\upalpha_{n+1} =\upalpha_{n}+\upvarphi$ }\\
\upvarphi^{-1}   & \text{si  $\upalpha_{n+1} =\upalpha_{n}+1$ }\\
\upvarphi^{-2} & \text{sinon}
\end{array} \right.  < n+1. \]
Ainsi, pour tout  $n\geq 2$, 
\[   1- \frac{\upvarphi^{2}}{\upalpha_{n}^{2}}<  1- \frac{1}{n^{2}} \]
et donc
\[ | e_{\mathfrak{a}}(\upvarphi ) |<  2 \prod_{n\geq 2}  \left( 1- \frac{1}{n^{2}}  \right) = 1. \]
Ainsi $e_{\mathfrak{a}}$ ne peut être un quasi morphisme de $\C$.

 \end{note}
 
 \section{Vers un concept de module de Drinfeld en caractéristique nulle}
 
 Dans cette section, nous présentons une proposition une première approche vers une notion de module de Drinfeld en caractéristique zéro. Malgré ton caractère partiel nous espérons qu'elle fournit une piste intéressante et peut-être féconde,  en direction d'une théorie de Drinfeld-Hayes en charactéristique zéro.
 
 Il s'agit essentiellement d'un essai pour remédier au manque d'injectivité de l'exponentielle quasicristalline, lequel
 représente un obstacle à l'utilisation de cette exponentielle pour la définition d'un équivalent de module de Drinfeld en ca\-ractéristique nulle.
 
On observe dans se contexte que~:
\begin{enumerate} 
\item[1.] dans le cas de l'exponentielle classique, celle-ci est une combination lineaire complexe du sinus et de sa dérivée, définissant une fonction injective sur $\SI^{1}$. Il n'est donc pas nécessaire d'utiliser davantage de dérivées en la définition de l'exponentielle~; au demeurant la suite des dérivées du sinus est périodique de période 2. 
\item[2.] dans le cas d'une courbe elliptique $\T = \C/\Uplambda$, on utilise la fonction $\wp $ de Weierstrass et sa dérivée~; comme l'application de Weierstrass
\[\T \ni z\longmapsto [\wp (z), \wp'(z),1] \in\PR^{2} \]est déjà injective, il n'est pas nécessaire  
de considérer plus de dérivées.
\item[3.] pour un module de Drinfeld, l'exponentielle est essentiellement le sinus~; sa dérivée est identiquement égale à 1 et il ne servirait à rien de l'inclure. 
 \end{enumerate}
 Par contre, on ne dispose d'aucune information montrant que la suite des dérivées des fonctions trigonometriques quasicristallines 
 \begin{align}\label{derivees} s_{\mathfrak{a}}(x),\;\; c_{\mathfrak{a}}(x) =s'_{\mathfrak{a}}(x),\;\; s_{\mathfrak{a}}^{(1)}(x):= c'_{\mathfrak{a}}(x), \;\; c_{\mathfrak{a}}^{(1)}(x):= (s_{\mathfrak{a}}^{(1)})'(x), \;\;   \dots \end{align} est périodique~; il semble même improbable qu'elle le soit. 
 
 Cela dit le corps $\C$ n'est pas suffisamment grand pour accomoder une image injective d'un solenoïde de dimension~1.  Dans l'approche du présent paragraphe, on va chercher une ultrapuissance de $\C$ qui admette une inclusion du solenoïde,  monomorphique si le solenoïde a une structure additive. Notons d'abord que tout ceci est vrai dans le cas du solenoïde classique:
 
\begin{exam} Soit \[\hat{\SI}^{1}=\lim_{\longleftarrow} \SI^{1}\] le solenoïde classique,
limite inverse des surjections $\SI^{1}\rightarrow \SI^{1}$, $x\mapsto nx$, $n\in\N$, et soit \[ \bast \,\C=\C^{\N}/\mathfrak{u}\] 
une ultrapuissance de $\C$ par rapport à  un ultrafiltre non principal $\mathfrak{u}\subset 2^{\N}$. On définit l'application
 \[  \hat{e}: \hat{\SI}^{1} \longrightarrow \bast \,\C,\quad\hat{x}= (x_{n}) \longmapsto \bast\{  e(x_{n})  \} ,\quad e(x):= \exp (2\uppi i x). \]
On note que $\hat{e}$ est bien définie et que c'est un homomorphisme.  Si l'on a $\hat{e}(\hat{x}) = \hat{e}(\hat{y})$, alors il existe $X\in\mathfrak{u}$, nécessairement infini, tel que $e (x_{n}) = e(y_{n})$
pour tout $n\in X$.  Mais ceci implique que les deux suites cohérentes coïncident à partir d'un certain rang.  On en conclut
que $\hat{e}$ est bien injective.
\end{exam}

Notons \[ e_{\mathfrak{a}}^{(n)} := i s_{\mathfrak{a}}^{(n)}  + c_{\mathfrak{a}}^{(n)}\] et 
pour $X\subset\N = \{ 0,1,2,\dots\}$ définissons
\[    e_{\mathfrak{a}}^{(X)}: \hat{\SI}_{\mathfrak{a}}\longrightarrow \C^{\infty},  \quad        e_{\mathfrak{a}}^{(X)} (\hat{x}):=
 \left( e_{\mathfrak{a}}^{(m)}(\hat{x}) \right)_{m\in X} .  \]
 Pour chaque ultrafiltre $\mathfrak{u}\subset {\sf 2}^{\N}$, l'ensemble  $\{ e_{\mathfrak{a}}^{(X)}\}$ induit une application
 \[   e_{\mathfrak{a}}^{(\mathfrak{u})}:  \hat{\SI}_{\mathfrak{a}}\longrightarrow  \bast \C_{\mathfrak{u}}.\]
 \begin{conj}  Il existe un ultrafiltre $\mathfrak{u}\subset {\sf 2}^{\N}$ tel que $ e_{\mathfrak{a}}^{(\mathfrak{u})}$ est injective.
\end{conj}

Soit 
\[ e_{\mathfrak{a}}^{(\infty)}:= e^{(\N)}_{\mathfrak{a}}.\]  
Le résultat suivant est peut-être un premier pas vers la justification de la conjecture ci-dessus~: 


 

    \begin{prop} $e^{(\infty)}_{\mathfrak{a}}$ est injective.
  \end{prop}
  
  \begin{proof}[Démonstration]  On rappelle que les points du solénoïde $\hat{\SI}_{\mathfrak{a}}$ paramétrisent les quasicristaux dans la clôture des translatés de $\mathfrak{a}$, dans
   l'espace topologique de tous les quasicristaux de dimension 1:
  \[ \hat{\SI}_{\mathfrak{a}} = \overline{\{  \mathfrak{a}-r\}_{r\in\R} }   .\]  
  Notons $\mathfrak{a}_{r}:= \mathfrak{a}-r $~;
   on a l'injection $\R\hookrightarrow \hat{\SI}_{\mathfrak{a}},\quad r\longmapsto \mathfrak{a}_{r} $.
   Si 
  $\hat{r}=\lim_{i} r_{i}$ est un point d'accumulation, notons $\mathfrak{a}_{\hat{r}} = \lim \mathfrak{a}_{r_{i}}$ le quasicristal associé.
  Il suffit alors de montrer l'injectivité de l'application
  \[ s^{(\infty)}_{\mathfrak{a}}:\hat{\SI}_{\mathfrak{a}}\longrightarrow \R^{\infty},\quad s_{\mathfrak{a}}^{(\infty)}(\hat{r}) = \left( s_{\mathfrak{a}}^{(n)}(\hat{r})\right)_{n\geq 0}.\]
Soit $r\in\R\subset\hat{\SI}$.  Puisque
  \[   s_{\mathfrak{a}}|_{\R}(x) = \sum \frac{s_{\mathfrak{a}}^{(n)}(r)}{n!} (x-r)^{n} ,\]
  la fonction de la variable réelle
  \[  s_{r} (x): = s_{\mathfrak{a}}|_{\R}(x+r)=\sum \frac{s_{\mathfrak{a}}^{(n)} (r)}{n!} x^{n} \]
 vérifie que l'ensemble ${\sf Z}(s_{r})$ de ses zéros est justement $\mathfrak{a}_{r}$.
Pour tout $\hat{r}\in\hat{\SI}$, on définit alors 
   \[  s_{\hat{r}} (x): =\sum \frac{s^{(n)} (\hat{r})}{n!}  x^{n} . \]
 Par continuité des dérivées $s^{(n)}$, i$ s_{\hat{r}} (x)$ s'étend à  $\hat{\SI}$ et  
 \[ {\sf Z}(s_{\hat{r}} ) = \text{les zeros de }s_{\hat{r}} \text{ en } \R=  \lim  {\sf Z}(s_{r} ) = \lim \mathfrak{a}_{r} =\mathfrak{a}_{\hat{r}} .\] 
Si $s^{(\infty)}_{\mathfrak{a}}(\hat{r}_{1})=s^{(\infty)}_{\mathfrak{a}}(\hat{r}_{2})$,  on a l'égalité des fonctions $s_{\hat{r}_{1}} = s_{\hat{r}_{2}}$, qui implique que
$  \mathfrak{a}_{\hat{r}_{1}} =   \mathfrak{a}_{\hat{r}_{2}}  $,
  d'où \ $\hat{r}_{1}=\hat{r}_{2}$.

 \end{proof}
 
Pour obtenir un authentique analogue de la théorie de Drinfeld basée sur $e^{(\infty)}_{\mathfrak{a}}$, il resterait à résoudre les points suivants: 
 \begin{itemize}
\item[A.] Déterminer en quel sens l'application $e^{(\infty)}_{\mathfrak{a}}$ est un homomorphisme~;
\item[B.] Trouver une classe $\mathcal{C}$ de polynômes additifs tels que, pour chaque $\upalpha\in A$, il existe $P_{\upalpha}(X)\in\mathcal{C}$ et
un diagramme commutatif
 \[ \begin{diagram}
  \hat{\SI}_{\mathfrak{a}} & \rTo^{\cdot\upalpha} & \hat{\SI}_{\mathfrak{a}} \\
 \dTo^{  e^{(\infty)}_{\mathfrak{a}} } & &\dTo_{e^{(\infty)}_{\mathfrak{a}}} \\
 \C^{\infty} &\rTo_{P_{\upalpha}} & \C^{\infty}
 \end{diagram}\]

\end{itemize}

Nous nous arrêterons provisoirement ici, réservant l'élucidation de ces deux points à une future étude.

\section{Appendice : Continuité de la conjugaison galoisienne}

Soit $\Upomega\subset\R$ un quasicristal de dimension 1.  Notons 
\[   \Upomega_{R}:= \Upomega\cap B_{R}(0) \]  
où $B_{R}(0)=[-R,R]$ est la boule de rayon $R$ centrée à $0$. 
Une suite de translations $\{ \upalpha_{i} +\Upomega\}$, $\upalpha_{i}\in\Upomega$, converge dans la topologie quasicristalline si pour tout $R>0$, il existe $N_{R}>0$ tel que pour tous $i,j>N_{R}$,
\[   (\Upomega-\upalpha_{i})_{R}=  (\Upomega-\upalpha_{j})_{R} . \]
La réunion
\[  \Upomega_{\hat{\upalpha}}:=\bigcup_{ \substack{R>0 \\ i>N_{R}}}  (\Upomega-\upalpha_{i})_{R}    \]
est un quasicristal.  La complétion quasicristalline $ \hat{\Upomega} $ est l'ensemble de tels $ \Upomega_{\hat{\upalpha}}$, muni de la topologie quasicristalline définie par la base d'ouverts 
\[ \mathcal{O}_{R}( \Upomega_{\hat{\upalpha}}) = \{ \Upomega'\in \hat{\Upomega}\;:\;\; \Upomega_{R}'=( \Upomega_{\hat{\upalpha}})_{R}\},\quad R>0.\]
Comme on l'a vu au \S \ref{solenoid}, $\hat{\Upomega}$ est une transversale fermée du solenoïde $\hat{\SI}_{\Upomega}$~; c'est donc un espace de Stone (un compact totalement discontinu).

Désormais on suppose que $\Upomega $ est un ensemble modèle de dimension 1, basé sur le réseau $\mathcal{O}_{K}\subset \R^{d}$
avec la fenêtre de paramètre ${\bf x}\in(\R_{+})^{d-1}$
\[ W=W_{{\bf x}}= (-\uptheta^{-x_{1}} , \uptheta^{-x_{1}} ) \times\cdots \times(-\uptheta^{-x_{d-1}} , \uptheta^{-x_{d-1}} )  \]
où
\[  \overline{W}=\overline{W}_{\bf x} =[-\uptheta^{-x_{1}} , \uptheta^{-x_{1}} ] \times\cdots \times [-\uptheta^{-x_{d-1}} , \uptheta^{-x_{d-1}} ] . \]
Notons que si $\uptheta^{-x_{i}} \not\in \mathcal{O}_{K} $ pour $i=1,\dots ,d-1$, $W$ et $\overline{W}$ définissent  le même ensemble modèle.  Sinon,  on obtient un nombre fini
de nouveaux points d'ensemble modèle en utilisant $\overline{W}$ au lieu de $W$.



On dit que $\Upomega$  is  {\em répétitif} si pour tout $R>0$, l'ensemble des $R$-periodes 
\[   {\rm Per}_{R}(\Upomega ) := \{ x\in \R: \; (\Upomega-x)_{R} = \Upomega_{R} \} \]
est relativement dense.

\begin{lemm}\label{replemma} $\Upomega$ est répétitif si et seulement si il est défini en utilisant la fenêtre ouverte $W$.  
\end{lemm}

\begin{proof}[Démonstration]  Supposons que $\Upomega$ se définit à l'aide de $W$.  \'Ecrivons
\begin{align}\label{completelistinofOmega} \Upomega_{R}=\pm \{ 0=\upalpha_{0} <\upalpha_{1} < \cdots < \upalpha_{k} \} \end{align} 
et pour chaque $i=1,\dots ,k$, 
\[  \boldsymbol \upalpha'_{i} :=  ( \upsigma_{1}(\upalpha_{i} ), \dots , \upsigma_{d-1}(\upalpha_{i} )  ) \] 
où $ \upsigma_{1},\dots , \upsigma_{d-1}$ sont les automorphismes galoisiens non triviaux de $K/\Q$.
En particulier, 
\begin{enumerate}
\item[1.] D'après les definitions de $\Upomega$ et $\Upomega_{R}$, pour tout $i=0,\dots ,k$, $ \boldsymbol \upalpha'_{i}\in W$ et $\upalpha_{i}\in (0,R)$.
\item[2.] Puisque  (\ref{completelistinofOmega}) fournit une liste complète des éléments de $\Upomega_{R}$, il n'existe pas $0<\upgamma\in \mathcal{O}_{K}$  telle qu'il existe
 $i$ avec
\[ i\leq k-1  \text{ et } \upgamma < \upalpha_{i+1} -\upalpha_{i}\quad \text{ou}\quad i=k\text{ et }  \upgamma <R-\upalpha_{k} \]
et
 \[  \pm( \upalpha_{i}+\upgamma)\in \Upomega. \] 
\end{enumerate}
De cette dernière propiété on tire que pour tous $i<k$ et $\upgamma< \upalpha_{i+1}-\upalpha_{i}$, 
\[ \pm \boldsymbol\upgamma ' \not\in W\pm  \boldsymbol \upalpha'_{i} 
\]
tandis que pour $i=k$ et $\upgamma< R-\upalpha_{k}$, on a de même 
\[ \pm \boldsymbol\upgamma ' \not\in W\pm  \boldsymbol \upalpha'_{k} .
\]
Considérons l'ensemble
\[ \Upsilon   = \big\{0\not= \upgamma \in\mathcal{O}_{K} \; : \;\;  \upgamma< \upalpha_{i+1}-\upalpha_{i}\text{ pour un certain }i<k \text{ or }  \upgamma< R-\upalpha_{k} \big\}\subset (0,\upalpha),\]
où \[ \upalpha:=\max\big\{ \max_{i<k}(\upalpha_{i+1}-\upalpha_{i}), R-\upalpha_{k}\big\} .\]
L'ensemble des conjugués $\Upsilon'$ est un sous-ensemble d'un ensemble modèle dans $\{ 0\}\times\R^{d-1}$, défini en utilisant la fenêtre $ (0,\upalpha)$ en $\R\times\{0\}$~; $\Upsilon'$ est donc uniformément 
discret et l'on peut choisir 
$\updelta>0$ suffisamment petit tel que, pour tout $\upbeta $ avec $\|\boldsymbol\upbeta' \|<\updelta$,
\begin{enumerate}
\item[a.] $\boldsymbol \upbeta'\pm\boldsymbol\upalpha_{i}' \in  W$~: c'est-à-dire, $\upbeta \pm \upalpha_{i}\in\Upomega$ et 
\item[b.] Pour $\upgamma\in\Upsilon$ et tout $i$ pour lequel soit  $\upgamma< \upalpha_{i+1}-\upalpha_{i}$, soit $i=k$, $\upgamma< R-\upalpha_{k}$
\[ \boldsymbol \upgamma' \not\in W \pm\boldsymbol\upalpha'_{i} -\boldsymbol\upbeta'. 
\]  En particulier, \begin{align} \label{thenoninclusion}\upbeta \pm \upalpha_{i}+\upgamma\not\in \Upomega\end{align}
 pour $\upgamma < \upalpha_{i+1}-\upalpha_{i}$  dans les cas  $i\leq k-1$, où $\upgamma<R-\upalpha_{k}$ dans le cas $i=k$.
\end{enumerate}
La condition a.\ implique que $\pm\upalpha_{i} = (\pm\upalpha_{i} +\upbeta)-\upbeta\in (\Upomega-\upbeta)_{R}$, qui fournit l'inclusion $\Upomega_{R}\subset (\Upomega-\upbeta)_{R}$.  La condition b.\
donne l'inclusion inverse~: en effet, supposons qu'il existe  \[ x=\upalpha-\upbeta\in  (\Upomega-\upbeta)_{R}\setminus \Upomega_{R},\quad \upalpha\in\Upomega.\]
On peut supposer sans perte de généralité que $x>0$~; le cas $x<0$ peut être réglé en utilisant un argument  identique.    De plus, $x$ doit se trouver dans les intervalles de $(0,R)$ délimités par les $\upalpha_{i}$. 
Alors, il existe $\upgamma>0$ tel que  $\upgamma<\upalpha_{i+1}-\upalpha_{i}$ ou $\upgamma<R-\upalpha_{k}$ avec 
\[  x = \upalpha -\upbeta =\upalpha_{i}+\upgamma.\]
Mais
\[ \upalpha= \upbeta + \upalpha_{i}+\upgamma\in\Upomega ,\]
en contradiction avec (\ref{thenoninclusion}).
On en conclut que  $(\Upomega-\upbeta)_{R} =\Upomega_{R}$, et donc $\upbeta\in {\rm Per}_{R}(\Upomega )$.  L'ensemble des $\upbeta$ avec $\|\boldsymbol\upbeta'\|<\updelta$ est un quasicristal, 
il est donc relativement dense, et ainsi  ${\rm Per}_{R}(\Upomega )$
est également relativement dense, ce qui démontre que $\Upomega$ est répétitif.  

Si par contre $\Upomega$ est défini à l'aide $\overline{W}$ et n'est pas définissable avec $W$, il existe un ensemble fini de  $\upalpha\in\Upomega$ avec $\boldsymbol \upalpha'\in \partial \overline{W}$.  Le sous ensemble $\Upomega_{R}\subset\Upomega$ pour $R>|\upalpha|$
ne peut se manifester ailleurs dans $\Upomega$~: autrement dit il n'existe pas de $\upbeta$ avec $(\Upomega-\upbeta)_{R}=  \Upomega_{R}$. En particulier, pour un tel $R$,  $  {\rm Per}_{R}(\Upomega ) = \{0\}$ et donc
$\Upomega$ n'est pas répétitif.
\end{proof}

\begin{coro}\label{minimalcoro} Si $\Upomega$ est defini à l'aide de $W=W_{\bf x}$, $\hat{\Upomega}$ est minimal~: pour tout $\Upomega'\in \hat{\Upomega}$, $\hat{\Upomega}'=\hat{\Upomega}$.  
\end{coro}

\begin{proof} Le quasicristal de dimension 1 $\Upomega$ peut être décrit comme un mot bi-infini sur un alphabet fini, où chaque lettre de l'alphabet indexe une différence d'éléments consécutifs 
de  $\Upomega$.  Pour $\Upomega$ un ensemble modèle défini avec $W$, d'après le Lemme \ref{replemma}, le mot bi-infini qui indexe $\Upomega$ est  répétitif. La minimalité
de $\hat{\Upomega}$ suit de la Proposition 4.3 (p. 78) de \cite{BG}.
\end{proof}

\begin{coro} Si $\Upomega$ est défini à l'aide de
\begin{enumerate}
\item[1.] $W$,  alors $\hat{\Upomega}$ est un ensemble de Cantor~;  
\item[2.] si c'est par $\overline{W}$ et qu'il n'est pas définissable par $W$, alors $\hat{\Upomega}$ est un espace de Stone mais non un ensemble de Cantor.  L'ensemble $\Upomega\subset\hat{\Upomega}$
est dense et tous ses points sont isolés~; $ \hat{\Upomega}_{0}:=\hat{\Upomega}\setminus\Upomega$ est un ensemble de Cantor  homéomorphe à $\hat{\Upomega}'$ pour tout $\Upomega'\in \hat{\Upomega}\setminus \Upomega$.
\end{enumerate}
\end{coro}

\begin{proof}[Démonstration] 1.\ est conséquence du Corollaire \ref{minimalcoro}, puisque ce dernier  montre que chaque point est un un point d'accumulation, et donc $\hat{\Upomega}$ est parfait.  Quant à 2., 
puisque $\Upomega$ n'est pas répétitif, pour n'importe quel $\upalpha\in\Upomega$, il ne peut  exister une suite  $\upalpha_{i} +\Upomega$ qui converge vers $\upalpha+\Upomega$. Les points de $\Upomega$ sont donc isolés.
Mais si l'on quitte $\Upomega\subset\hat{\Upomega}$, les quasicristaux restant ne contiennent aucun translaté d'un morceau de $\Upomega_{R}$, $R>|\upalpha |$, pour
tout $\upalpha$ tel que $\upalpha'\in \partial \overline{W}$.  
Les quasicristaux $\Upomega'\in  \hat{\Upomega}_{0}$ sont donc minimaux et leurs completions coïncident avec $ \hat{\Upomega}_{0}$.  On en conclut que $ \hat{\Upomega}_{0}$ est parfait et donc est un ensemble de Cantor.
\end{proof}

\begin{theo}\label{conjconvergence} Si $\{ \upalpha_{i} +\Upomega\}$ converge dans la topologie quasicristalline, la suite des conjuguées $\{ \upalpha'_{i}\}$ converge dans $\R$.  Autrement dit,
l'application
\[ \Upomega \longrightarrow W, \quad  \upalpha +\Upomega\longrightarrow \upalpha'\]
s'étend en une application continue $\hat{\Upomega}_{\rm qc}\rightarrow \overline{W}$.
\end{theo}

\begin{proof}  Supposons d'abord que $\Upomega$ est défini par $W$, donc est répétitif.  Soit
\[ \Upomega_{\hat{\upalpha}} := \lim  \left( \upalpha_{i} +\Upomega\right) \in\hat{\Upomega}.\]   Puisque $\upalpha_{i} +\Upomega\subset\mathcal{O}_{K}$ pour tout $i$,  $\Upomega_{\hat{\upalpha}}\subset \mathcal{O}_{K}$.  D'après le Lemma 4.1 de \cite{Schlottmann}, 
\[   \bigcap_{\upomega\in \Upomega_{\hat{\upalpha}}} ( \upomega'-\overline{W}) = \{ c_{\hat{\upalpha}}\}  \]
pour un $c_{\hat{\upalpha}}\in \R$.  Notons alors que pour $\Upomega_{\upalpha}:=\upalpha +\Upomega\in \hat{\Upomega}$, puisque $\upalpha' \in \upomega'- W$ pour tout $\upomega\in \upalpha +\Upomega$, 
$c_{\upalpha} = \upalpha'\in W$.   Alors l'application
\[ \hat{\Upomega}\longrightarrow \R,\quad  \Upomega_{\hat{\upalpha}}\longmapsto c_{\hat{\upalpha}}  \]
étend la conjugaison et il reste à montrer que application est continue~: il s'en suivra alors  que $c_{\hat{\upalpha}}\in \overline{W}$. 
Soit $V'\ni c_{\hat{\upalpha}}$ un voisinage ouvert~; on a 
\[   \bigcap_{\upomega\in \Upomega_{\hat{\upalpha}}}  [( \upomega'-\overline{W}) \setminus V'  ] =\emptyset   .\]
Puisque chaque $ [( \upomega'-\overline{W}) \setminus V'  ] $ est compact, il existe un ensemble fini $F\subset\Upomega_{\hat{\upalpha}}$ tel que
\[  \bigcap_{\upomega\in F}  [( \upomega'-\overline{W}) \setminus V'  ] =\emptyset . \] 
Ceci montre qu'il existe $R>0$ tel que
\[   \bigcap_{\upomega\in (\Upomega_{\hat{\upalpha}})_{R}}  [( \upomega'-\overline{W}) \setminus V'  ] =\emptyset  , \]
ce qui implique  que \[  \bigcap_{\upomega\in(\Upomega_{\hat{\upalpha}})_{R}}  ( \upomega'-\overline{W}) \subset V' .\]
Donc, étant donné $\Upomega_{\hat{\upbeta}}\in\hat{\Upomega}$ avec
\[   (\Upomega_{\hat{\upbeta}})_{R} =  (\Upomega_{\hat{\upalpha}})_{R}  , \] on a 
 $c_{\hat{\upbeta}} \in V'$, ce qui démontre le théorème lorsque $\Upomega$  se définit à l'aide de $W$.  Si $\Upomega$ se définit par $\overline{W}$, il n'est pas
répétitif, mais l'argument ci-dessus montre que la conjugaison restreinte à l'ensemble de Cantor $\hat{\Upomega}_{0} = \hat{\Upomega}\setminus\Upomega$ est continue. Puisque les points de $\Upomega$
sout isolés dans $\hat{\Upomega}$, il s'ensuit que  la conjugaison est continue sur tout $\hat{\Upomega}$.
\end{proof}

\end{document}